\documentclass[a4paper]{amsart}
\usepackage{eucal}
\usepackage{enumerate}
\usepackage{tikz-cd}

\theoremstyle{plain}
\newtheorem{thm}{Theorem}[subsection]
\newtheorem*{thm*}{Theorem}
\newtheorem*{thm1}{Theorem 1}
\newtheorem*{thm2}{Theorem 2}
\newtheorem{prop}{Proposition}[subsection]
\newtheorem*{prop*}{Proposition}
\newtheorem{lemma}{Lemma}[subsection]
\newtheorem*{lemma*}{Lemma}

\newtheorem*{coro*}{Corollary}
\newtheorem{conj}{Conjecture}[subsection]

\theoremstyle{definition}

\newtheorem*{dfn*}{Definition}
\newtheorem{rem}{Remark}[subsection]
\newtheorem*{rem*}{Remark} 

\newtheorem*{ex*}{Example}
\newtheorem{hyp}{Hypothesis}[subsection]

\newcommand{\A}{\mathbb A}
\newcommand{\C}{\mathbb{C}}
\newcommand{\R}{\mathbb R}
\newcommand{\Z}{\mathbb Z}
\newcommand{\Q}{\mathbb Q}
\newcommand{\Or}{\mathcal O}
\newcommand{\Gal}{\operatorname{Gal}}
\newcommand{\Hom}{\operatorname{Hom}}
\newcommand{\Aut}{\operatorname{Aut}}

\newcommand{\coh}{\operatorname{coh}}

\newcommand{\Res}{\operatorname{Res}}
\newcommand{\GL}{\operatorname{GL}}
\newcommand{\Temp}{\operatorname{Temp}}
\newcommand{\Ad}{\operatorname{Ad}}
\newcommand{\Fr}{\operatorname{Fr}}
\newcommand{\ab}{\operatorname{ab}}
\newcommand{\Gm}[1]{\mathbb{G}_{\operatorname{m},#1}}
\newcommand{\gr}{\operatorname{gr}}
\newcommand{\aut}{\operatorname{aut}}
\newcommand{\id}{\operatorname{id}}
\newcommand{\diag}{\operatorname{diag}}

\newcommand{\St}{\operatorname{St}}
\newcommand{\Sh}{\operatorname{Sh}}
\newcommand{\hol}{\operatorname{hol}}
\newcommand{\ahol}{\operatorname{ahol}}
\newcommand{\can}{\operatorname{can}}
\newcommand{\sub}{\operatorname{sub}}

\newcommand{\dR}{\operatorname{dR}}
\newcommand{\Lie}{\operatorname{Lie}}

\newcommand{\art}{\operatorname{art}}
\newcommand{\W}{\mathcal{W}}
\newcommand{\mot}{\operatorname{mot}}
\newcommand{\Coh}{\operatorname{Coh}}

\numberwithin{equation}{subsection}

\begin{document} 

\title[Period relations for automorphic forms on unitary groups]{Period relations for automorphic forms on unitary groups and critical values of $L$-functions}
\author{Lucio Guerberoff}
\address{Department of Mathematics, University College London, 25 Gordon Street, London WC1H 0AY, UK}
\email{l.guerberoff@ucl.ac.uk}

\begin{abstract} In this paper we explore some properties of periods attached to automorphic representations of unitary groups over CM fields and the critical values of their $L$-functions. We prove a formula expressing the critical values in the range of absolute convergence in terms of Petersson norms of holomorphic automorphic forms. On the other hand, we express the Deligne period of a related motive as a product of quadratic periods and compare the two expressions by means of Deligne's conjecture. 
\end{abstract}

\maketitle

\section{Introduction}\label{sec:intro}
The goal of the present paper is to study critical values of the standard $L$-functions of cohomological automorphic representations of unitary groups, and relate them to the motivic expression predicted by Deligne's conjecture (\cite{deligne}). This extends previous results by Harris~(\cite{harriscrelle}) from quadratic imaginary fields to arbitrary CM fields $L$. Let $K$ be the maximal totally real subfield of $L$, and let $G$ be a similitude unitary group attached to an $n$-dimensional hermitian space over $L$. Fix a CM type $\Phi$ for $L/K$, and suppose that $G$ has sigature $(r_{\tau},s_{\tau})$ at each $\tau\in\Phi$. Let $\pi$ be a cuspidal cohomological automorphic representation of $G(\A)$. We can parametrize the weight of $\pi$ by a tuple of integers $((a_{\tau,1},\dots,a_{\tau,n})_{\tau\in\Phi};a_{0})$. See Section~\ref{sec:automorphicmotives} for details. We assume that the corresponding algebraic representation of $G_{\C}$ is defined over $\Q$, and that $\pi^{\vee}\cong\pi\otimes\|\nu\|^{2a_{0}}$, where $\nu$ is the similitude factor. Let $\psi$ be an algebraic Hecke character of $L$ of infinity type $(m_{\tau})_{\tau:L\hookrightarrow\C}$. Our main theorem is the following.

\begin{thm1} Suppose that $\pi$ satisfies Hypothesis \ref{hypomult} and contributes to antiholomorphic cohomology. If $m>n$ is an integer satisfying 
\[ m\leq\min\{a_{\tau,r_{\tau}}+s_{\tau}+m_{\tau}-m_{\overline\tau},a_{\tau,s_{\tau}}+r_{\tau}+m_{\overline\tau}-m_{\tau}\}_{\tau\in\Phi},\]
then
\[ L^{S}\left(m-\frac{n-1}{2},\pi\otimes\psi,\St\right)\sim(2\pi i)^{[K:\Q](mn-n(n-1)/2)-2a_{0}}D_{K}^{\lfloor\frac{n+1}{2}\rfloor/2}P(\psi)Q^{\hol}(\pi).\]
\end{thm1}
In this expression, the members belong to $E(\pi)\otimes E(\psi)\otimes\C$, where $E(\pi)$ and $E(\psi)$ are number fields over which $\pi_{f}$ and $\psi$ are defined, and $\sim$ means up to multiplication by an element of $E(\pi)\otimes E(\psi)\otimes L'$, with $L'$ being the Galois closure of $L$ in $\C$. We refer the reader to Section~\ref{sec:doubling}, in particular to Theorem \ref{maintheorem}, for a precise and detailed explanation of the notation. Let us mention that the factor $Q^{\hol}(\pi)$ is an automorphic quadratic period attached to $\pi$. This is basically defined as a Petersson norm of an arithmetic holomorphic vector in $\pi$. The element $P(\psi)$ is an explicit expression involving CM periods attached to $\psi$. 

The method of proof of Theorem 1 follows the lines of \cite{harriscrelle}, and is based on earlier work by Shimura (\cite{shimurazeta}). It is based on the doubling method, and allows us to write the $L$-function as an integral of a holomorphic automorphic form against a certain Eisenstein series. Roughly speaking, Li proved in \cite{li} that the $L$-function can be written in terms of global and local Piatetski-Shapiro-Rallis zeta integrals and an inner product between automorphic forms. The inequality that $m$ needs to satisfy in the hypotheses of the theorem governs the existence of a differential operator for automorphic vector bundles, constructed in \cite{harrisvb2}. By carefully choosing the sections defining the Eisenstein series and using these differential operators, we can see that the zeta integrals are rational over $L'$, and we can interpret the inner product as the automorphic quadratic period. See Section \ref{sec:doubling} for more details. In the final sections of the paper, we interpret the formula in Theorem 1 motivically to obtain period relations.

\subsection{Background and motivation} The first results concerning the expression of special values of automorphic $L$-functions as Petersson norms were due to Shimura, especially in the case of Hilbert modular forms (see~\cite{shimurazeta}, \cite{shimuraduke}, \cite{shimurarelations}, \cite{shimuraperiods88}). Petersson norms are to be interpreted as quadratic periods, as in Shimura's conjectures concerning the factorization of periods. In~\cite{harriscrelle}, Harris generalized the quadratic periods to the setting of coherent cohomology of Shimura varieties, in this case attached to unitary groups of hermitian spaces over quadratic imaginary fields. In this paper, we treat the case of arbitrary CM fields. The motivic interpretation, given below, relates the expression of Theorem 1 to Deligne's conjecture on critical values. It should be noted that quadratic periods have their own importance independently of any reference to motives. In particular, a proper understanding of them is key to the construction of $p$-adic $L$-functions.

Period relations have been studied by several authors such as Shimura, Yoshida, Oda, Schappacher, Panchishkin, Blasius, Hida and Harris, among others. One of the basic principles in their prediction is Tate's conjecture. For instance, Tate's conjecture is used crucially in Blasius's proof of Deligne's conjecture for Hecke $L$-series (\cite{blasiusannals}). See also \cite{blasiusperiods}. It has long been known by specialists, dating back to Shimura, Deligne and Langlands, that there should be motivic relations, and hence period relations, predicted by relations between automorphic forms on different groups. In this paper, this manifestation takes the form of period relations for unitary groups of different signatures.

For related recent results, we remark that Jie Lin in her recent Paris thesis (\cite{thesislin}) conjectures a similar formula as that of Theorem 1, but without the discriminant factor, which is assumed to belong to the coefficient field. This is used to prove results generalizing those of \cite{grobnerharris}. 

\subsection{Motivic interpretation} We can interpret motivically the formula in Theorem 1 as follows. Suppose that $\Pi$ is a cuspidal, cohomological, self-dual automorphic representation of $\GL_{n}(\A_{K})$. There is a conjectural motive $M$ over $K$, with coefficients in a number field $E\subset\C$, attached to $\Pi$. The $\ell$-adic realizations of $M$ have already been constructed (see \cite{chl}, \cite{shingalois}, \cite{ch}, \cite{sorensen}), and the existence of $M$ will mostly play a heuristic role. We refer to Section~\ref{sec:factorization} for details about motives. For simplicity in what follows, we fix the embedding $E\hookrightarrow\C$, and all the $L$-values and periods will be considered to be complex numbers via this embedding. Under certain assumptions, we can descend $\Pi_{L}$ to an automorphic representation $\pi$ of the unitary group $G$ (\cite{labesse}; see also \cite{mok}, \cite{kmsw}). Write $\psi|_{\A_{K}}=\psi_{0}\|\cdot\|^{-w}$, with $w$ being the weight of $\psi$ and $\psi_{0}$ a finite order character, and let $\chi=\psi^{2}(\psi_{0}\circ N_{L/K})^{-1}$. The relation between the standard $L$-function of $\pi\otimes\psi$ and $M$ is encompassed in the following formula:
\[ L\left(s-\frac{n-1}{2},\pi\otimes\psi,\St\right)=L(M\otimes RM(\chi),s+w).\]
Here $RM(\chi)$ is the restriction of scalars from $L$ to $K$ of the motive $M(\chi)$ attached to $\chi$. In Section~\ref{sec:factorization}, we determine explicitly the set of critical integers of $L(M\otimes RM(\chi),s)$. The motive $M$ is regular of weight $n-1$, in the sense that the Hodge components $M_{\sigma}^{pq}$ have dimension $0$ or $1$ for each $\sigma$. We let the Hodge numbers be $(p_{i}(\sigma),q_{i}(\sigma))$, where $p_{1}(\sigma)>\dots>p_{n}(\sigma)$. We assume that $m_{\tau}\neq m_{\overline\tau}$ for any $\tau$, and take the CM type $\Phi$ in such a way that $m_{\tau}>m_{\overline\tau}$ for $\tau\in\Phi$. We show that if $M\otimes RM(\chi)$ has critical values, then for each $\sigma$ there exists $r_{\sigma}=0,\dots,n$ such that
\[ n-1-2p_{r_{\sigma}}(\sigma)<2m_{\tau}-2m_{\overline\tau}<n-1-2p_{r_{\sigma}+1}(\sigma).\]
Moreover, suppose that the signatures of $G$ are given by $r_{\tau}=r_{\sigma}$ for $\tau\in\Phi$ extending $\sigma$ (starting from $\Pi$ and $\psi$, we can always find $G$ with these signatures). The set of critical integers of the form $m+w$ for $M\otimes RM(\chi)$ is governed by two inequalities (see~(\ref{setofcriticalvalues})), one of which is precisely the inequality on Theorem 1. Thus, we can see the corresponding values of the $L$-function of $\pi\otimes\psi$ as critical values of $M\otimes RM(\chi)$. Deligne's conjecture predicts that these are, up to multiplication by an element in the coefficient field, equal to the Deligne periods $c^{+}(M\otimes RM(\chi))(m+w)$.

The motive $M$ is also equipped with a polarization $M\cong M^{\vee}(1-n)$. The methods of Section~\ref{sec:factorization} allow us to write the $\sigma$-periods $c_{\sigma}^{+}(M\otimes RM(\chi))$ in terms of another set of periods $Q_{j,\sigma}$ of $M$, called quadratic periods. We refer to the main text for a precise definition. We prove the following result, first proved in \cite{harriscrelle} when $K=\Q$.

\begin{thm2} If $m+w$ is critical for $M\otimes RM(\chi)$, then
\[ c^{+}\left(M\otimes RM(\chi)(m+w)\right)\sim(2\pi i)^{[K:\Q]mn+w\sum_{\sigma}r_{\sigma}-s_{\sigma}}\delta(M)Q(\chi)\prod_{\sigma}\prod_{j=1}^{s_{\sigma}}Q_{j,\sigma}.\]
\end{thm2}
We can actually obtain a more precise formula involving only the $c_{\sigma}^{+}(M\otimes RM(\chi))$ for a single $\sigma$ (see Theorem \ref{thmfact}). Here, $Q(\chi)$ is an explicit expression involving CM periods attached to $\chi$. We stress that the conjectural existence of the motive $M$ attached to $\Pi$ plays a heuristic role, but Theorem 2 is proved for any family of realizations $M$ (such as a motive for absolute Hodge cycles) with the properties of being regular and polarized over a totally real field. Comparing this expression with that of Theorem 1, we match the CM periods and get that the prediction of Deligne's conjecture for $M\otimes RM(\chi)$ is translated in the following statement:
\begin{equation}\label{relation} \prod_{\sigma}\prod_{j=1}^{s_{\sigma}}Q_{j,\sigma}\sim(2\pi i)^{-[K:\Q]n(n-1)/2}\delta(M)^{-1}D_{K}^{n/2}Q^{\hol}(\pi).\end{equation}
We can state this relation without making reference to the quadratic periods $Q_{j,\sigma}$ by interpreting them as automorphic periods obtained from automorphic representations of different unitary groups, whose signatures are $(n,0)$ at all places except at one place, where the signature is $(n-1,1)$. These automorphic representations contribute in coherent cohomology to the different stages of the Hodge filtration (as opposed to the single holomorphic stage which gives rise to $Q^{\hol}(\pi)$). We show that the relation (\ref{relation}) is predicted by Tate's conjecture as well, as we explained above. If one is willing to assume this conjecture, then this implies Deligne's conjecture for the motives $M\otimes RM(\chi)$. We should stress here that what it actually implies is a slightly weaker version of Deligne's conjecture. Namely, the automorphic methods only allow us to relate the expressions up to multiples by elements in the Galois closure $L'$ of $L$ in $\C$. Also, we only obtain the version of Deligne's conjecture obtained by fixing an embedding of the coefficient field. This last issue does not arise when $K=\Q$, since the Hodge components $M_{\sigma}^{pq}$ of the motives in question are free over $E\otimes\C$, something which is almost never true when $K$ is bigger than $\Q$ (for example, this is already false for motives $M(\chi)$ attached to algebraic Hecke characters). 

\subsection{Organization of the paper} In Section \ref{sec:factorization}, we recall the basic facts and main properties about motives, realizations and their periods. The reader who is interested solely in the main theorem on critical values of automorphic $L$-functions on unitary groups can skip Section~\ref{sec:factorization} and go directly to Sections~\ref{sec:automorphicmotives} and \ref{sec:doubling}. The reason we include Section~\ref{sec:factorization} before is that we use some of the terminology regarding Hodge-de Rham structures and polarizations in Section~\ref{sec:automorphicmotives}. In Section~\ref{sec:factorization}, we make emphasis in regular polarized motives over totally real fields, and in Theorem \ref{thmfact}, we prove the factorization of $c_{\sigma}^{+}(M\otimes RM(\chi))$ in term of quadratic periods. 

In Section \ref{sec:automorphicmotives}, we introduce unitary groups and their associated Shimura varieties. We set up the notation for the parameters of representations giving rise to automorphic vector bundles, and give a brief overview of the main properties of the Hodge-de Rham structures attached to cohomological automorphic representations, constructed in \cite{harrismotives}. We write down the action of complex conjugation and the polarizations in terms of automorphic forms, and we give an automorphic definition of quadratic periods in the setting of coherent cohomology. 

Section \ref{sec:doubling} contains our main theorem on critical values of cohomological automorphic $L$-functions. In the first subsections we set up the doubling method and we recall the relation, proved by Li (\cite{li}), between the standard $L$-functions and Piatetski-Shapiro-Rallis zeta integrals, and in Subsection \ref{ssec:mainthm} we prove the main theorem.

Finally, in Section~\ref{sec:relations}, we postulate period relations obtained by comparing the results of Sections~\ref{sec:factorization} and \ref{sec:doubling} by means of Deligne's conjecture. This section is hypothetical in nature. We start by recalling some basic facts about transfer and descent for automorphic representations of unitary groups and $\GL_{n}$, along with several motivic expectations, including the existence of the motive $M$ described above, as well different relations between the Hodge-de Rham structures attached to nearly equivalent automorphic representations of unitary group, which are consequences of Tate's conjecture. We show how they imply the period relations embodied in Deligne's conjecture. The arguments are heuristic, depending on these motivic expectations, but we can write down the predicted period relations concretely in terms of automorphic forms on unitary groups.

\subsection*{Acknowledgements} The author wants to thank Michael Harris for many helpful answers, his advice and his support. The author also thanks Daniel Barrera, Don Blasius and Jie Lin for several useful discussions, and for comments on earlier versions of this paper. Finally, the author thanks the referee for helpful comments and suggestions. 

Part of this work was carried out while the author was a guest at the Max Planck Institute for Mathematics in Bonn, Germany. It is a pleasure to thank the Institute for its hospitality and the excellent working conditions.

\subsection*{Notation and conventions}
We fix an algebraic closure $\C$ of $\R$, a choice of $i=\sqrt{-1}$, and we let $\overline\Q$ denote the algebraic closure of $\Q$ in $\C$. We let $c\in\Gal(\C/\R)$ denote complex conjugation on $\C$, and we use the same letter to denote its restriction to $\overline\Q$. Sometimes we also write $c(z)=\overline z$ for $z\in\C$. We let $\Gamma_{\Q}=\Gal(\overline\Q/\Q)$.

For a number field $K$, we let $\A_{K}$ and $\A_{K,f}$ denote the rings of ad\`eles and finite ad\`eles of $K$ respectively. When $K=\Q$, we write $\A=\A_{\Q}$ and $\A_{f}=\A_{\Q,f}$. If $K$ is a number field (resp. a finite extension of $\Q_{p}$ for some prime number $p$), and $\overline K$ is a fixed algebraic closure, let $K^{\ab}$ be the maximal abelian extension of $K$ inside $\overline K$. We let $\art_{L}:\A_{L}^{\times}\to\Gal(K^{\ab}/K)$ (resp. $\art_{L}:L^{\times}\to\Gal(K^{\ab}/K)$) be the Artin reciprocity map of class field theory. We normalize these maps so that the global reciprocity map is compatible with the local ones, and the local map takes a uniformizer to a geometric Frobenius element.

A CM field $L$ is a totally imaginary quadratic extension of a totally real field $K$. A CM type $\Phi$ for $L/K$ is a choice of one of the two possible extensions to $L$ of each embedding of $K$. 

All vector spaces will be finite-dimensional except otherwise stated. By a variety over a field $K$ we will mean a geometrically reduced scheme of finite type over $K$.  

For a field $K$, we let $\Gm{K}$ denote the usual multiplicative group over $K$. For any algebraic group $G$ over K, we let $\Lie(G)$ denote its Lie algebra and $\Ad:G\to\GL_{\Lie(G)}$ the adjoint representation. A reductive algebraic group will always be assumed to be connected.

We let $\mathbb{S}=R_{\C/\R}\Gm{\C}$. We denote by $c$ the complex conjugation map on $\mathbb{S}$, so for any $\R$-algebra $A$, this is $c\otimes_{\R}1_A:(\C\otimes_{\R}A)^\times\to(\C\otimes_{\R}A)^\times$. We usually also denote it by $z\mapsto\overline z$, and on complex points it should not be confused with the other complex conjugation on $\mathbb{S}(\C)=(\C\otimes_{\R}\C)^{\times}$ on the second factor.

A tensor product without a subscript between $\Q$-vector spaces will always mean tensor product over $\Q$. For any number field $K$, we denote by $J_{K}=\Hom(K,\C)$. For $\sigma\in J_{K}$, we let $\overline\sigma=c\sigma$. Let $E$ and $K$ be number fields, and $\sigma\in J_K$. If $\alpha,\beta\in E\otimes \C$, we write $\alpha\sim_{E\otimes K,\sigma}\beta$ if either $\beta=0$ or if $\beta\in(E\otimes \C)^{\times}$ and $\alpha/\beta\in(E\otimes K)^{\times}$, viewed as a subset of $(E\otimes \C)^{\times}$ via $\sigma$. There is a natural isomorphism
$E\otimes\C\simeq\prod_{\varphi\in J_E}\C$ given by $e\otimes z\mapsto(\varphi(e)z)_{\varphi}$ for $e\in E$ and $z\in\C$. Under this identification, we denote an element $\alpha\in E\otimes\C$ by $(\alpha_{\varphi})_{\varphi\in J_E}$. We write $\alpha\sim_{E;K;\sigma}\beta$ if either $\beta=0$ or $\beta\in(E\otimes\C)^{\times}$ and $\alpha_{\varphi}/\beta_{\varphi}\in\varphi(E)\sigma(K)\subset\C$ for all $\varphi\in J_{E}$. The relation $\alpha\sim_{E\otimes K,\sigma}\beta$ implies $\alpha\sim_{E;K;\sigma}\beta$, but the converse is not necessarily true. When $K$ is given from the context as a subfield of $\C$, we write $\sim_{E\otimes K}$ (resp. $\sim_{E;K}$) for $\sim_{E\otimes K,1}$ (resp. $\sim_{E;K;1}$), where $1:K\hookrightarrow\C$ is the given embedding.

Suppose that $\mathbf{r}=(r_{\varphi})_{\varphi\in J_E}$ is a tuple of nonnegative integers. Given $Q_{1},\dots,Q_{n}$ in $E\otimes\C$ (with $n\geq r_{\varphi}$ for all $\varphi$), we denote by
\[ \prod_{j=1}^{\mathbf{r}}Q_{j}\in E\otimes\C \]
the element whose $\varphi$-th coordinate is $\prod_{j=1}^{r_{\varphi}}Q_{j,\varphi}$. In particular, this defines $x^{\mathbf{r}}$ for $x\in E\otimes\C$.

We choose Haar measures on local and adelic points of unitary groups as in the Introduction of \cite{harriscrelle}.

\section{Factorization of Deligne's periods}\label{sec:factorization}
In this section, we start by introducing the notation we will use for motives and realizations. We work with the category of realizations (as in \cite{jannsen}, \S2) instead of the category of motives for absolute Hodge cycles. All of the unexplained notions can be found in \cite{deligne} (see also \cite{schbook}, \cite{jannsen}, \cite{panch} and \cite{yoshida}). We then recall the basic facts about periods, and introduce polarized regular realizations and their quadratic periods, extending the results of \cite{harriscrelle} from $\Q$ to a totally real field. The main result of the section is Theorem \ref{thmfact}, expressing the $\sigma$-periods of motives of the form $M\otimes RM(\chi)$ in terms of quadratic periods.

\subsection{Motives and realizations}\label{ssec:motives}  Let $K$ be a number field. We fix an algebraic closure $\overline K$ of $K$ and we let $\Gamma_{K}=\Gal(\overline K/K)$. For a number field $E$, by a \emph{pure realization} over $K$ with coefficients in $E$, of weight $w\in\Z$, we will mean the following data.

\begin{itemize} 
\item For each $\sigma\in J_{K}$, an $E$-vector space $M_{\sigma}$ of dimension $d$ (independent of $\sigma$), together with $E$-linear isomorphisms $F_{\sigma}:M_{\sigma}\to M_{\overline\sigma}$ which satisfy $F_{\sigma}^{-1}=F_{\overline\sigma}$. Each $M_{\sigma}$ is endowed with a $\Q$-Hodge structure of weight $w$
\[ M_{\sigma}\otimes \C=\bigoplus_{pq}M_{\sigma}^{pq},\]
where each $M_{\sigma}^{pq}$ is an $E\otimes\C$-submodule, such that $F_{\sigma,c}=F_{\sigma}\otimes c$ sends $M_{\sigma}^{pq}$ to $M_{\overline\sigma}^{pq}$. We let $F_{\sigma,\C}=F_{\sigma}\otimes1_{\C}$. The filtration on $M_{\sigma}\otimes\C$ induced by the Hodge decomposition is called the \emph{Hodge filtration}.
\item A free $E\otimes K$-module $M_{\dR}$ of rank $d$, together with a decreasing filtration $F^{\bullet}(M_{\dR})$ by (not necessarily free) $E\otimes K$-submodules; this is called the \emph{de Rham filtration}.
\item For each finite place $\lambda$ of $E$, an $E_{\lambda}$-vector space $M_{\lambda}$ of dimension $d$, endowed with a continuous action of $\Gamma_{K}$.
\item For each $\sigma\in J_{K}$, an $E\otimes \C$-linear isomorphisms $I_{\infty,\sigma}:M_{\sigma}\otimes \C\to M_{\dR}\otimes_{K,\sigma}\C$ compatible with the Hodge and de Rham filtrations. We also require that $c_{\dR,\sigma}I_{\infty,\sigma}=I_{\infty,\overline{\sigma}}F_{\sigma,c}$, where $c_{\dR,\sigma}=1_{M_{\dR}}\otimes_{K,\sigma}c$.
\item For each $\sigma\in J_{K}$, for each extension $\tilde\sigma:\overline K\hookrightarrow\C$ of $\sigma$ to $\overline K$, and for each $\lambda$, $E_{\lambda}$-linear isomorphisms $I_{\tilde\sigma,\lambda}:M_{\sigma}\otimes_{E}E_{\lambda}\to M_{\lambda}$. Moreover, if $\sigma$ is real, then the automorphism $F_{\sigma}\otimes_{E}1_{E_{\lambda}}$ of $M_{\sigma}\otimes_{E}E_{\lambda}$ corresponds, via $I_{\tilde\sigma,\lambda}$ to the action of the element of $\Gamma_{K}$ given by the complex conjugation deduced from $\tilde\sigma$.
\end{itemize}

For an object $M$ as above, the integer $d$ is called the rank of $M$. A morphisms between a pure realization $M$ and a pure realization $N$ is defined to be a family of maps $((f_{\sigma})_{\sigma\in J_{K}},f_{\dR},(f_{\lambda})_{\lambda})$, where $f_{\sigma}:M_{\sigma}\to N_{\sigma}$ is an $E$-linear morphism of Hodge structures for each $\sigma$, $f_{\dR}:M_{\dR}\to N_{\dR}$ are $E\otimes K$-linear maps preserving the de Rham filtrations, and $f_{\lambda}:M_{\lambda}\to N_{\lambda}$ are $E_{\lambda}$-linear $\Gamma_{K}$-equivariant maps. We require all these maps to correspond under the comparison isomorphisms. The category $\mathcal{R}(K)_{E}$ is defined to be the category whose objects are direct sums of pure realizations. It is a semi-simple Tannakian category over $E$, whose objects are simply called \emph{realizations}. A \emph{Hodge-de Rham structure} consists of the same data, without the $\ell$-adic realizations (see \cite{harrismotives}). 

As in Deligne's article \cite{deligne}, a motive will mean a pure motive for absolute Hodge cycles (see \cite{dmos}, Section 6, for details). We denote by $\mathcal{M}(K)_{E}$ the category of motives over $K$ with coefficients in $E$. There is a fully faithful functor from the category of motives $\mathcal{M}(K)_{E}$ to $\mathcal{R}(K)_{E}$, which identifies $\mathcal{M}(K)_{E}$ with the full Tannakian subcategory of $\mathcal{R}(K)_{E}$ generated by the cohomologies of smooth projective varieties over $K$ (\cite{jannsen}). If $M$ is a motive, we also denote by $M$ the realization it defines in $\mathcal{R}(K)_{E}$.

Since $E\otimes \C\simeq\C^{J_{E}}$, we can write $M_{\sigma}\otimes \C\simeq\bigoplus_{\varphi\in J_{E}}M_{\sigma}(\varphi)$, where $M_{\sigma}(\varphi)=(M_{\sigma}\otimes\C)\otimes_{E\otimes\C,\varphi}\C\simeq M_{\sigma}\otimes_{E,\varphi}\C$. Since the Hodge decomposition $M_{\sigma}^{pq}$ is $E\otimes\C$-stable, this gives each factor $M_{\sigma}(\varphi)$ a decomposition
\[ M_{\sigma}(\varphi)=\bigoplus_{p,q}M_{\sigma}^{pq}(\varphi),\]
where $M_{\sigma}^{pq}(\varphi)=M_{\sigma}^{pq}\otimes_{E\otimes \C,\varphi}\C$. This has the property that complex conjugation sends $M_{\sigma}^{pq}(\varphi)$ to $M_{\sigma}^{qp}(\overline\varphi)$. We put 
\[ h^{pq}_{\sigma}(\varphi)=\dim_{\C}M_{\sigma}^{pq}(\varphi)=\dim_{\C}\gr^{p}(M_{\dR})\otimes_{E\otimes K,\varphi\otimes\sigma}\C.\]
Thus, $h_{\sigma}^{pq}(\varphi)=h_{\sigma}^{qp}(\overline\varphi)=h_{\overline\sigma}^{pq}(\overline\varphi)$. We say that $M\in\mathcal{R}(K)_{E}$ is \emph{regular} if $h_{\sigma}^{pq}(\varphi)\leq 1$ for every pair of integers $p,q$ and every $\sigma\in J_{K}$, $\varphi\in J_E$. If $M$ is regular of rank $d$ and pure of weight $w$, then given $\sigma$ and $\varphi$, there are $d$ numbers $p_{1}(\sigma,\varphi),\dots,p_{d}(\sigma,\varphi)$ with the property that $M_{\sigma}^{pq}(\varphi)\neq 0$ (and has complex dimension $1$) if and only if $p=p_{i}(\sigma,\varphi)$ for some $i$. For fixed $\sigma$ and $\varphi$, we order these numbers in such a way that $p_1(\sigma,\varphi)>\dots>p_d(\sigma,\varphi)$. We let $q_{i}(\sigma,\varphi)=w-p_{i}(\sigma,\varphi)$. Note that $q_{i}(\sigma,\varphi)=p_{d+1-i}(\sigma,\overline\varphi)=p_{d+1-i}(\overline{\sigma},\varphi)$.

For a realization $M$, we denote its dual by $M^{\vee}$, and for realizations $M,N$, we denote their tensor product (over $E$) by $M\otimes_{E}N$. We similarly adopt the standard linear algebra notation for exterior products. If $M$ is a realization with coefficients in $\Q$, and $N$ is a realization with coefficients in $E$, then we can naturally see $M\otimes(R_{E/\Q}N)$, denoted by $M\otimes N$, as a realization with coefficients in $E$, where $R_{E/\Q}N$ is restriction of coefficients from $E$ to $\Q$. 

Let $E$ and $E'$ be number fields, and write $E\otimes E'\cong\prod_{j=1}^{m}F_{j}$, where $F_{j}$ are number fields. Let $M\in\mathcal{R}(K)_{E}$ and $N\in\mathcal{R}(K)_{E'}$, of rank $d$ and $e$ respectively. We define realizations $(M\otimes N)^{(j)}\in\mathcal{R}(K)_{F_{j}}$ of rank $de$ by taking $(M\otimes N)^{(j)}=(M\times_{E}F_{j})\otimes_{F_{j}}(N\times_{E'}F_{j})$. By $M\otimes N$ we mean the collection $\{(M\otimes N)^{(j)}\}_{j=1}^{m}$, and we often say that it is a realization with coefficients in $E\otimes E'$. 

We denote by $\mathcal{M}^{0}(K)_{E}$ the category of Artin motives over $K$ with coefficients in $E$. This is equivalent to the category of continuous, finite-dimensional representations of $\Gamma_{K}$ on $E$-vector spaces. In this setting, we can describe the realizations of an Artin motive $M$, viewed as a representation $V$ of $\Gamma_{K}$, as follows. For every $\sigma\in J_{K}$, $M_{\sigma}\cong V$, the isomorphism depending on an extension of $\sigma$ to $\overline K$. The Hodge structures are purely of type $(0,0)$. Also, $M_{\lambda}\cong V\otimes_{E}E_{\lambda}$ and $M_{\dR}\cong(V\otimes\overline K)^{\Gamma_{K}}$. If $\epsilon:\A_{K}^{\times}/K^{\times}\to E^{\times}$ is a finite order character, we denote by $[\epsilon]$ the Artin motive in $\mathcal{M}^{0}(K)_{E}$ given by the character $\Gamma_{K}\to E^{\times}$ obtained from $\epsilon$ by class field theory.

Let $\chi:\A_{K}^{\times}/K^{\times}\to\C^{\times}$ be an algebraic Hecke character. Recall that this means that $\chi$ is continuous, and that for every embedding $\sigma\in J_{K}$, there exist an integer $n_{\sigma}$, such that if $v$ is the infinite place of $K$ induced by $\sigma$ and $x\in (K^{\times}_{v})^{+}$, then
\[ \chi(x)=\sigma(x)^{-n_{\sigma}}\quad\text{if }v\text{ is real};\]
\[ \chi(x)=\sigma(x)^{-n_{\sigma}}\overline\sigma(x)^{-n_{\overline\sigma}}\quad\text{if }v\text{ is complex}.\]
The integer $n_{\sigma}+n_{\overline\sigma}=w(\chi)$ is independent of $\sigma$, and is called the weight of $\chi$. The tuple $(n_{\sigma})_{\sigma\in J_{K}}$ is called the infinity type of $\chi$. Let $T^{K}=\Res_{K/\Q}\Gm{K}$. Consider the group of characters $X^{*}(T^{K})$, which is naturally identified with $\Z^{J_{K}}$. For $\eta\in X^{*}(T^{K})$, we denote by $X(\eta)$ the set of algebraic Hecke characters $\chi$ of $K$ of infinity type $\eta$. Let $\Q(\chi)\subset\C$ denote the field generated by the values of $\chi$ on $\A_{K,f}^{\times}$ (the finite id\`eles). Then $\Q(\chi)$ is either $\Q$ or a CM field. Let $E\subset\C$ be a number field containing $\Q(\chi)$. We denote by $M(\chi)\in\mathcal{M}(K)_{E}$ the motive of weight $w(\chi)$ attached to $\chi$, as in \cite{schbook}. For $n\in\Z$, the motive attached to the character $\chi(x)=\|x\|^{n}$, where $\|\cdot\|$ is the id\`elic norm, is the Tate motive $\Q(n)\in\mathcal{M}(K)_{\Q}$. For any $M\in\mathcal{R}(K)_{E}$, we let $M(n)=M\otimes\Q(n)$.

For later use, we record the Hodge decomposition of $M(\chi)$. Let $\Q(\eta)$ be the field of definition of the character $\eta$, that is, the fixed field in $\overline{\Q}$ of the stabilizer of $\eta\in X^{*}(T^{K})$ under the natural action of $\Gamma_{\Q}$. Then $\Q(\eta)\subset\Q(\chi)$. For any $\gamma\in\Gamma_{\Q}$ (or $\gamma\in\Aut(\C)$), we let $\chi^{\gamma}\in X(\eta^{\gamma})$ be the algebraic Hecke character of infinity type $\eta^{\gamma}$ whose values on $\A_{K,f}^{\times}$ are obtained by applying $\gamma$ to the values of $\chi$. The tuple of integers parametrizing $\eta^{\gamma}$ is given by $(n(\sigma,\gamma))_{\sigma\in J_{K}}$, where $n(\sigma,\gamma)=n_{\gamma^{-1}\sigma}$. Note that $\eta^{\gamma}$, and hence $n(\sigma,\gamma)$, only depend on the restriction of $\gamma$ to $\Q(\eta)$. In particular, we can define $n(\sigma,\varphi)$ for any $\varphi\in J_{\Q(\chi)}$. Then $M(\chi)_{\sigma}^{pq}(\varphi)\neq 0$ if and only if $p=n(\sigma,\varphi)$ and $q=n(\sigma,\overline\varphi)$. Note however that different embeddings $\varphi$ could give rise to the same numbers, so $M(\chi)_{\sigma}^{n(\sigma,\varphi),n(\sigma,\overline\varphi)}\neq M(\chi)_{\sigma}^{n(\sigma,\varphi),n(\sigma,\overline\varphi)}(\varphi)$ in general. Since $M(\chi)$ is of rank $1$, it is regular and $p_{1}(\sigma,\varphi)=n(\sigma,\varphi)$.

In order to properly define the $L$-function of a realization $M$, we need to impose that assumption that $(M_{\lambda})_{\lambda}$ is a strictly compatible system of $\lambda$-adic representations over $E$ (see \cite{deligne}, 1.1 for details). We will assume from now on without further mention that this holds. If $\varphi\in J_{E}$, then $L(\varphi,M,s)=\prod_{v}L_{v}(\varphi,M,s)$, where $L_{v}(\varphi,M,s)$ is the corresponding Euler factor at $v$ of the system $(M_{\lambda})_{\lambda}$. We also assume that $L(\varphi,M,s)$ converges absolutely for $\operatorname{Re}(s)$ large enough and has a meromorphic continuation to $\C$. We define $L^{*}(M,s)=(L(\varphi,M,s))_{\varphi\in J_{E}}\in\C^{J_{E}}\simeq E\otimes \C$. We can complete the $L$-function with the factors at infinity $L_{\sigma}(\varphi,M,s)$ for each $\sigma\in J_{K}$. This is done following Serre's recipe (\cite{serre}) as in 5.2 of \cite{deligne}. We define $L_{\infty}(\varphi,M,s)=\prod_{\sigma\in J_{K}}L_{\sigma}(\varphi,M,s)$ and $L_{\infty}^{*}(M,s)=(L_{\infty}(\varphi,M,s))_{\varphi\in J_{E}}\in\C^{J_{E}}\simeq E\otimes\C$. It is not hard to see that $L_{\infty}(\varphi,M,s)$ does not depend on $\varphi$ (see 2.9 of \cite{deligne}). Note that $L^{*}(M(n),s)=L^{*}(M,s+n)$, and similarly for the factors at infinity.

Let $M\in\mathcal{R}(K)_{E}$. We say that an integer $n$ is \emph{critical} for $M$ if for every $\varphi\in J_{E}$ (or equivalently, for one $\varphi$), neither $L_{\infty}(\varphi,M,s)$ nor $L_{\infty}(\varphi,M^{\vee},1-s)$ have a pole at $s=n$. We say that $M$ has critical values if there exists an integer critical for $M$, and we say that $M$ is critical if $0$ is critical for $M$. Thus, $M$ has critical values if and only if $M(n)$ is critical for some $n\in\Z$. 

\subsection{Periods}\label{ssec:periods} Let $M\in\mathcal{R}(K)_{E}$. From now on, unless otherwise stated, $K$ is a totally real number field. In this subsection we recall the definition of periods. We refer to \cite{deligne}, \cite{panch} and \cite{yoshida} for details. For each $\sigma\in J_{K}$, we define $\delta_{\sigma}(M)\in(E\otimes\C)^{\times}$ to be the determinant of $I_{\infty,\sigma}$, calculated with respect to an $E$-basis of $M_{\sigma}$ and an $E\otimes K$-basis of $M_{\dR}$. This is well defined modulo $(E\otimes K)^{\times}$ (contained in $(E\otimes\C)^{\times}$ via $\sigma$). We define $\delta(M)=\delta(\Res_{K/\Q}M)$, with respect to the unique element of $J_{\Q}$.

A realization $M\in\mathcal{R}(K)_{E}$ is said to be \emph{special} if it is pure of some weight $w$, and if for every $\sigma\in J_{K}$, $F_{\sigma,\C}$ acts on $M_{\sigma}^{w/2,w/2}$ by a scalar $\varepsilon=\pm1$, independent of $\sigma$. It is easily checked that a pure realization with critical values is special (see for instance (1.3.1) of \cite{deligne}). Suppose that $M$ is special. For each $\sigma\in J_{K}$, let $M_{\sigma}^{\pm}\subset M_{\sigma}$ denote the $\pm$-eigenspace for $F_{\sigma}$. It's easy to see that the dimension of $M_{\sigma}^{\pm}$ is independent of $\sigma$, and we denote this common dimension by $d^{\pm}=\dim_{E}M^{\pm}_{\sigma}$. We can also choose appropriate terms $F^{\pm}(M_{\dR})\subset M_{\dR}$ such that $M^{\pm}_{\dR}=M_{\dR}/F^{\mp}(M_{\dR})$ is a free $E\otimes K$-module of rank $d^{\pm}$ and the map
\[ I_{\infty,\sigma}^{\pm}:M_{\sigma}^{\pm}\otimes \C\to M_{\dR}^{\pm}\otimes_{K,\sigma}\C \]
given by the composition of the projection $M_{\dR}\otimes_{K,\sigma}\C\to M_{\dR}^{\pm}\otimes_{K,\sigma}\C$ with $I_{\infty,\sigma}$ and with the inclusion $M_{\sigma}^{\pm}\otimes\C\hookrightarrow M_{\sigma}\otimes\C$ is an $E\otimes \C$-linear isomorphism. We define $c^{\pm}_{\sigma}(M)=\det(I_{\infty,\sigma}^{\pm})\in(E\otimes \C)^{\times}$, where the determinants are computed in terms of an $E$-basis of $M_{\sigma}^{\pm}$ and an $E\otimes K$-basis of $M_{\dR}^{\pm}$. Note that these quantities are defined modulo $(E\otimes K)^{\times}\subset(E\otimes \C)^{\times}$ (via $\sigma$). The relation between these periods and the usual Deligne periods $c^{\pm}(\Res_{K/\Q}M)$, which we denote by $c^{\pm}(M)$, is given by the following factorization formula, proved in \cite{yoshida} or \cite{panch} (we also include a similar formula for the $\delta$'s):
\begin{align}\label{form:yoshida1} c^{\pm}(M)\sim_{E\otimes K'} D_{K}^{d^{\pm}/2}\prod_{\sigma}c^{\pm}_{\sigma}(M).\\
\label{form:yoshida2} \delta(M)\sim_{E\otimes K'} D_{K}^{d/2}\prod_{\sigma}\delta_{\sigma}(M).\end{align}
Here $D_{K}$ is the discriminant of $K$, and $K'\subset\overline\Q$ is the Galois closure of $K$ in $\overline\Q$.

The following is the main conjecture of \cite{deligne}.

\begin{conj}[Deligne] If $M$ is critical and $L(\varphi,M,0)\neq 0$ for some $\varphi$, then 
	\[ L^{*}(M,0)\sim_{E}c^{+}(M).\]\end{conj}
The conjecture is aimed at motives rather than general realizations. We formulate a weaker version that we will need later. The hypotheses are as in Deligne's conjecture. 

\begin{conj}\label{deligneweak2} Let $F\subset\C$ be a number field. By the weak Deligne conjecture up to $F$-factors, we mean the statement
\[ L(M,s)\sim_{E;F}c^{+}(M). \]
\end{conj}
Regarding the notation in this conjecture, recall that this means that the $\varphi$-components differ by a multiple in $\varphi(E)F$.

\begin{rem}\label{rem:twist} If $M$ is of rank $d$ and $t\in\Z$, then $\delta_{\sigma}(M(t))\sim(2\pi i)^{td}\delta_{\sigma}(M)$. If $M$ is special, then so is $M(t)$. If $t$ is even, then $c^{\pm}_{\sigma}(M(t))\sim(2\pi i)^{td^{\pm}}\cdot c_{\sigma}^{\pm}(M)$, while if $t$ is odd, then $c^{\pm}_{\sigma}(M(t))\sim(2\pi i)^{td^{\mp}}\cdot c_{\sigma}^{\mp}(M)$. In all these formulas, $\sim$ means $\sim_{E\otimes K,\sigma}$.

Let $\epsilon:\Gamma_{K}\to E^{\times}$ be a continuous character, and let $[\epsilon]\in\mathcal{M}^{0}(K)_{E}$ be the corresponding Artin motive. For each $\sigma\in J_{K}$, choose $\tilde\sigma:\overline K\hookrightarrow\C$ an extension of $\sigma$ to $\overline K$. Then $[\epsilon]_{\sigma}\cong E$ and $[\epsilon]_{\dR}\cong(E\otimes\overline K)^{\Gamma_{K}}$. The map
\[ I_{\infty,\sigma}^{-1}:(E\otimes\overline K)^{\Gamma_{K}}\otimes_{K,\sigma}\C\cong E\otimes\C \]
is given by $((e\otimes\lambda)\otimes z)\mapsto(e\otimes\tilde\sigma(\lambda)z)$. If $\{\zeta\}$ is an $E\otimes K$-basis of $[\epsilon]_{\dR}$, and we choose the natural $E$-basis $\{1\}$ of $E$, then $\det(I_{\infty,\sigma})=\tilde\sigma(\zeta)^{-1}$, so that $\delta_{\sigma}([\epsilon])\sim_{E\otimes K,\sigma}\tilde\sigma(\zeta)^{-1}$ (see also p. 104, \cite{schbook}). Now, note that the Frobenius automorphism $F_{\sigma}$ acts on the one-dimensional $E$-vector space $[\epsilon]_{\sigma}$ by the sign $\epsilon(c_{\sigma})$, where $c_{\sigma}\in\Gamma_{K}$ is a complex conjugation attached to the place $\sigma$. Suppose that this scalar $\varepsilon$ does not depend on $\sigma$. This means that $[\epsilon]$ is special. If $\varepsilon=-1$, then $c_{\sigma}^{+}([\epsilon])\sim1$ and $c_{\sigma}^{-}([\epsilon])\sim\delta_{\sigma}([\epsilon])$. If $\varepsilon=1$, then $c_{\sigma}^{+}([\epsilon])\sim\delta_{\sigma}([\epsilon])$ and $c_{\sigma}^{-}([\epsilon])\sim1$. Finally, note that if $M$ is a special realization, then so is $M\otimes_{E}[\epsilon]$ and $c^{\pm}_{\sigma}(M\otimes_{E}[\epsilon])\sim c_{\sigma}^{\pm\varepsilon}(M)\delta_{\sigma}([\epsilon])^{d^{\pm\varepsilon}}$.

\end{rem}

\begin{rem}\label{rem:dual} It's easy to see that $F^{\mp}(M_{\dR}^\vee)\subset M_{\dR}^{\vee}$ is the annihilator of $F^{\pm}(M_{\dR})$. It follows that there are natural isomorphisms $(F^{\pm}(M_{\dR}))^{\vee}\simeq(M^{\vee})^{\pm}_{\dR}$. In particular, $F^{\pm}(M_{\dR})$ is $E\otimes K$-free of rank $d^\pm$.
\end{rem}

\begin{rem}\label{rem:c+reg} Let $M\in\mathcal{R}(K)_{E}$ be a realization of rank $d$ and weight $w$, assumed to be special. Since $E\otimes K$ is a product of fields, we can extend any $E\otimes K$-basis of the free rank $d^{\pm}$ module $F^{\pm}(M_{\dR})$ to an $E\otimes K$-basis of the free rank $d$ module $M_{\dR}$. Replacing $M_{\dR}$ by the appropriate $F^{+}(M_{\dR})$ or $F^{-}(M_{\dR})$ as an intermediate step, we can choose the bases consistently, so it follows that we can find a basis $\{\omega_{1},\dots,\omega_{d}\}$ of $M_{\dR}$ such that $\{\omega_{1},\dots,\omega_{d^{\pm}}\}$ is a basis of $F^{\pm}(M_{\dR})$. For simplicity of notation, $\omega_{i}$ will also denote the element $\omega_{i}\otimes_{K,\sigma}1\in M_{\dR}\otimes_{K,\sigma}\C$, since $\sigma$ will be understood throughout.

Let $\{e_{1},\dots,e_{d^{+}}\}$ (resp. $\{f_{1},\dots,f_{d^{-}}\}$) be an $E$-basis of $M_{\sigma}^{+}$ (resp. of $M_{\sigma}^{-}$), and write
\[ I_{\infty,\sigma}^{-1}(\omega_{j})=\sum_{i=1}^{d^{+}}a_{ij,\sigma}^{+}e_{i}+\sum_{i=1}^{d^{-}}a_{ij,\sigma}^{-}f_{i},\quad j=1,\dots,d \]
in $M_{\sigma}\otimes\C$, with $a_{ij,\sigma}^{\pm}\in E\otimes\C$. Then
\begin{equation}\label{form:c+dual} c_{\sigma}^{\pm}(M^{\vee})\sim_{E\otimes K,\sigma}\det(P_{\sigma}^{\pm}),\end{equation}
where $P_{\sigma}^{\pm}=\left((a_{ij,\sigma}^{\pm})_{i,j=1,\dots,d^{\pm}}\right)$. Indeed, this follows from Remark~\ref{rem:dual}: letting $I_{\infty,\sigma}^{\vee}$ be the comparison maps for $M^{\vee}$, the equations above mean precisely that
\[ I_{\infty,\sigma}^{\vee}(e_{j}^\vee)=\sum_{i=1}^{d^+}a_{ji,\sigma}^+\omega_{i}^\vee,\]
where $\{e_{1}^\vee,\dots,e_{d^+}^\vee\}$ is the dual basis of $M_{\sigma}^{\vee,+}$ and $\{\omega_{1}^\vee,\dots,\omega_{d^+}^\vee\}$ is the dual basis of $(F^+(M_{\dR}))^\vee$, thus proving~(\ref{form:c+dual}) for $c_{\sigma}^+$. The case of $c_{\sigma}^-$ is completely similar.

Furthermore, suppose that $\{\Omega_{1},\dots,\Omega_{d}\}$ is an $E\otimes\C$-basis of $M_{\sigma}\otimes\C$ with the property that the change of basis matrix with respect to $\{\omega_{1},\dots,\omega_{d}\}$ is unipotent. More precisely, suppose that
\[ I_{\infty,\sigma}(\Omega_{i})=\omega_{i}+\sum_{j=1}^{i-1}r_{ji,\sigma}\omega_{j} \]
with $r_{ji,\sigma}\in E\otimes\C$. We let $r_{ii,\sigma}=1$ and $r_{ji,\sigma}=0$ if $j>i$. Write
\begin{equation}\label{form:comblinealO} \Omega_{j}=\sum_{i=1}^{d+}\tilde a_{ij,\sigma}^{+}e_{i}+\sum_{i=1}^{d-}\tilde a_{ij,\sigma}^{-}f_{i},\quad j=1,\dots,d\end{equation}
and let $\tilde P_{\sigma}^{\pm}=\left((\tilde a_{ij,\sigma}^{\pm})_{i,j=1,\dots,d^{\pm}}\right)$. If $R_{\sigma}^{\pm}$ denotes the matrix constructed from the first $d^{\pm}$ rows and columns of $R_{\sigma}=(r_{ij,\sigma})_{i,j=1,\dots,d}$, then $R_{\sigma}^{\pm}$ is upper triangular with diagonal entries $1$, and the formula $\tilde P_{\sigma}^{\pm}=P_{\sigma}^{\pm}R_{\sigma}^{\pm}$ implies that
\begin{equation}\label{form:widec+} c^{\pm}(M^{\vee})\sim_{E\otimes K,\sigma}\det\tilde P_{\sigma}^{\pm}.\end{equation}

Similarly, if we let 
\[ P_{\sigma}=\left(\begin{array}{c}\left(a_{ij,\sigma}^{+}\right)_{\substack{1\leq i\leq d^{+}\\1\leq j\leq d}} \\\left(a_{ij,\sigma}^{-}\right)_{\substack{1\leq i\leq d^{-}\\1\leq j\leq d}} \end{array}\right)\in\GL_{d}(E\otimes \C) \]
and
\[ \tilde P_{\sigma}=\left(\begin{array}{c}\left(\tilde a_{ij,\sigma}^{+}\right)_{\substack{1\leq i\leq d^{+}\\1\leq j\leq d}} \\\left(\tilde a_{ij,\sigma}^{-}\right)_{\substack{1\leq i\leq d^{-}\\1\leq j\leq d}} \end{array}\right)\in\GL_{d}(E\otimes \C), \]
then $\delta_{\sigma}(M)\sim_{E\otimes K,\sigma}\det(P_{\sigma})^{-1}$, and the formula $\tilde P_{\sigma}=P_{\sigma}R_{\sigma}$ implies that
\[ \delta_{\sigma}(M)\sim_{E\otimes K,\sigma}\det(\tilde P_{\sigma})^{-1}.\]
\end{rem}

\subsection{Polarizations and quadratic periods}\label{ssec:quadraticperiods} Let $M\in\mathcal{R}(K)_{E}$ be pure of weight $w$, and let $\epsilon:\Gamma_{K}\to E^{\times}$ be a continuous character. Let $A=[\epsilon]$ denote the corresponding rank $1$ Artin motive. An $A$-\emph{polarization} of $M$ is a morphism of realizations
\[ \langle,\rangle:M\otimes_{E}M\to A(-w) \]
which is non-degenerate in the sense that the induced map $M\to M^\vee(-w)\otimes_{E}A$ is an isomorphism of realizations. If $\epsilon=1$, so $A(-w)=E(-w)$, we speak of polarizations. We will mostly be interested in polarizations, but the added versatility of an $A$-polarization will be useful in one particular, yet very important case, and it gives the advantage that many arguments do not need to be repeated. We will not make use of the $\lambda$-adic polarizations, so in fact we only require the existence of the Betti and de Rham polarizations, compatible as they must be (for this we could simply work in the category of Hodge-de Rham structures instead of realizations). The fact that $\langle,\rangle$ is a morphism of realizations immediately implies the \emph{Hodge-Riemann bilinear relations}:

\begin{itemize}
\item $\langle F^{p}(M_{\dR}),F^{q}(M_{\dR})\rangle=0$ if $p+q>w$;
\item $\langle,\rangle$ defines non-degenerate pairings $M_{\sigma}^{pq}\otimes_{E\otimes\C}M_{\sigma}^{qp}\to A_{\sigma}\otimes\C$ whenever $p+q=w$. Moreover, these pairings are rational, that is, they descend to non-degenerate pairings $\gr^p(M_{\dR})\otimes_{E\otimes K}\gr^{q}(M_{\dR})\to A_{\dR}$. 
\end{itemize}

Let $M\in\mathcal{R}(K)_{E}$ be a regular, special realization endowed with a $A$-polarization. Fix once and for all an $E\otimes K$-basis $\{\zeta\}$ of $A_{\dR}=(E\otimes\overline K)^{\Gamma_{K}}$. For each $\sigma\in J_K$ and $\varphi\in J_{E}$, let $p_{1}(\sigma,\varphi)>\dots>p_{d}(\sigma,\varphi)$ be the Hodge numbers defined in Subsection~\ref{ssec:motives}. As in Remark~\ref{rem:c+reg}, we can choose an $E\otimes K$-basis $\{\omega_{1,\sigma},\dots,\omega_{d,\sigma}\}$ of $M_{\dR}$ such that $\{\omega_{1,\sigma},\dots,\omega_{d^\pm,\sigma}\}$ is an $E\otimes K$-basis of $F^{\pm}(M_{\dR})$. Moreover, letting $\omega_{i,\sigma}(\varphi)=\omega_{i,\sigma}\otimes_{E\otimes K,\varphi\otimes\sigma}1\in M_{\dR}\otimes_{E\otimes K,\varphi\otimes\sigma}\C$, we claim that we can choose the basis in such a way that $\{\omega_{1,\sigma}(\varphi),\dots,\omega_{i,\sigma}(\varphi)\}$ is a $\C$-basis of $F^{p_{i}(\sigma,\varphi)}(M_{\dR})\otimes_{E\otimes K,\varphi\otimes\sigma}\C$ for each $i=1,\dots,d$ and each $\varphi\in J_{E}$. To see this, we need some extra notation that will also be useful later. Write $E\otimes K\cong\prod_{\mu=1}^{m}K_{\mu}$, where the $K_{\mu}/K$ are finite extensions. Fixing the embedding $\sigma$, this decomposition induces a bijection between $J_{E}$ and $\coprod_{\mu=1}^{m}\Hom_{\sigma}(K_{\mu},\C)$, where we define the latter as the sets of embeddings of $K_{\mu}$ extending $\sigma$. It's easy to see that two embeddings $\varphi$ and $\varphi'$ give rise to the same index $\mu$ if and only if $\varphi'=h\circ\varphi$ for some $h\in\Aut(\C)$ such that $h\circ\sigma=\sigma$. Now, notice that $M_{\sigma}^{pq}(\varphi)\neq 0$ if and only if $\gr^{p}(M_{\dR})\otimes_{E\otimes K,\varphi\otimes\sigma}\C\neq 0$. The isomorphism 
\[\gr^{p}(M_{\dR})\otimes_{E\otimes K,\varphi'\otimes\sigma}\C\cong\gr^{p}(M_{\dR})\otimes_{E\otimes K,\varphi\otimes\sigma}\C\otimes_{\C,h}\C \]
implies that $p_{i}(\sigma,\varphi')=p_{i}(\sigma,\varphi)$ for all $i=1,\dots,d$. This shows that we can unambiguously define $p_{i}(\sigma,\mu)$ by declaring them to be $p_{i}(\sigma,\varphi)$, where $\varphi\in J_{E}$ is an embedding corresponding to the index $\mu$. Finally, there is a natural isomorphism $M_{\dR}\simeq\bigoplus_{\mu}M_{\dR}\otimes_{E\otimes K}K_{\mu}$, so we can construct a basis of $M_{\dR}$ from a family of bases of the spaces $M_{\dR}\otimes_{E\otimes K}K_{\mu}$. Once we fix $\mu$, we can choose the corresponding basis by taking the first $i$ elements in $F^{p_{i}(\sigma,\mu)}$, which proves our claim. 

For each $i$, $\sigma$ and $\varphi$, let 
\[ I_{\infty,\sigma,i}(\varphi):M_{\sigma}^{p_{i}(\sigma,\varphi),q_{i}(\sigma,\varphi)}(\varphi)\to\gr^{p_{i}(\sigma,\varphi)}(M_{\dR})\otimes_{E\otimes K,\varphi\otimes\sigma}\C\]
be the induced isomorphism obtained from $I_{\infty,\sigma}$. The image $\overline{\omega_{i,\sigma}(\varphi)}=\overline{\omega_{i,\sigma}}(\varphi)$ of $\omega_{i,\sigma}(\varphi)$ in the right hand side is a $\C$-basis, and we let 
\[ \Omega_{i,\sigma}(\varphi)=I_{\infty,\sigma,i}(\varphi)^{-1}(\overline{\omega_{i,\sigma}(\varphi)}) \]
be the corresponding $\C$-basis of $M_{\sigma}^{p_{i}(\sigma,\varphi),q_{i}(\sigma,\varphi)}(\varphi)$. Since $I_{\infty,\sigma}(\varphi)(\Omega_{i})-\omega_{i,\sigma}(\varphi)$ belongs to $F^{p_{i}(\sigma,\varphi)+1}$, we can write
\[ I_{\infty,\sigma}(\varphi)(\Omega_{i,\sigma}(\varphi))=\omega_{i,\sigma}(\varphi)+\sum_{j=1}^{i-1}r_{ji,\sigma}(\varphi)\omega_{j,\sigma}(\varphi) \]
for some $r_{ji,\sigma}(\varphi)\in\C$. We let $r_{ji,\sigma}(\varphi)=0$ if $j>i$ and $r_{ii,\sigma}(\varphi)=1$. Let $R_{\sigma}(\varphi)=(r_{ij,\sigma}(\varphi))_{i,j=1,\dots,d}$, so that $R_{\sigma}(\varphi)$ is the change of basis matrix from $\{I_{\infty,\sigma}(\varphi)(\Omega_{i,\sigma}(\varphi))\}_{i=1,\dots,d}$ to $\{\omega_{i,\sigma}(\varphi)\}_{i=1,\dots,d}$. Note that $R_{\sigma}(\varphi)$ is an upper triangular matrix with diagonal entries $1$. We let $r_{ij,\sigma}\in E\otimes\C$ be the elements whose $\varphi$-components are $r_{ij,\sigma}(\varphi)$ for every $\varphi$, and $R_{\sigma}$ be the corresponding upper triangular matrix in $\GL_{d}(E\otimes\C)$ with diagonal entries $1$. Under the isomorphism $M_{\sigma}\otimes\C\simeq\bigoplus_{\varphi}M_{\sigma}(\varphi)$, let $\Omega_{i,\sigma}\in M_{\sigma}\otimes\C$ be the element whose $\varphi$-component is $\Omega_{i,\sigma}(\varphi)$, so that we can write
\[ I_{\infty,\sigma}(\Omega_{i,\sigma})=\omega_{i,\sigma}+\sum_{j=1}^{i-1}r_{ji,\sigma}\omega_{j,\sigma} \]
and $R_{\sigma}$ is the change of basis matrix from $\{I_{\infty,\sigma}(\Omega_{i,\sigma})\}_{i=1,\dots,d}$ to $\{\omega_{i,\sigma}\otimes_{K,\sigma}1\}_{i=1,\dots,d}$ in $M_{\dR}\otimes_{K,\sigma}\C$, as in Remark~\ref{rem:c+reg}.

For each $i=1,\dots,d$, define the elements $\mu_{i,\sigma}\in E\otimes K$ by
\[ \langle\omega_{i,\sigma},\omega_{d+1-i,\sigma}\rangle_{\dR}=\mu_{i,\sigma}\zeta.\]
If $\varphi\in J_{E}$, then
\[ \langle\omega_{i,\sigma}(\varphi),\omega_{d+1-i,\sigma}(\varphi)\rangle_{\dR,\sigma,\varphi}=\mu_{i,\sigma}(\varphi)\zeta,\]
where $\mu_{i,\sigma}(\varphi)=(\varphi\otimes\sigma)(\mu_{i,\sigma})\in\C$. From the second Hodge-Riemann bilinear relation, it follows that $(\varphi\otimes\sigma)(\mu_{i,\sigma})\in\C^{\times}$. The following lemma implies that $\mu_{i,\sigma}\in(E\otimes K)^{\times}$. 

\begin{lemma} Let $E$ and $K$ be number fields, and fix $\sigma\in J_{K}$. Suppose that $x\in E\otimes K$ is an element such that $(\varphi\otimes\sigma)(x)\in\C^{\times}$ for every $\varphi\in J_{E}$. Then $x\in(E\otimes K)^{\times}$.
\begin{proof} Write $E\otimes K\cong\prod_{\mu}K_{\mu}$ as before, and let $x_{\mu}\in K_{\mu}$ be the $\mu$-component of $x$. We need to show that $x_{\mu}\neq 0$ for any $\mu$. For this, it's enough to see that $\sigma_{\mu}(x_{\mu})\in\C^{\times}$ for some embedding $\sigma_{\mu}$ of $K_{\mu}$. Since we can choose an arbitrary embedding, we can take one extending $\sigma$, so that the pair $(\mu,\sigma_{\mu})$ determines an embedding $\varphi\in J_{E}$ with the property that $(\varphi\otimes\sigma)(x)=\sigma_{\mu}(x_{\mu})$. The lemma follows.
\end{proof}
\end{lemma}

The first Hodge-Riemann bilinear equation implies that
\[ \langle\omega_{i,\sigma}(\varphi),\omega_{d+1-i,\sigma}(\varphi)\rangle_{\dR,\sigma,\varphi}=\langle I_{\infty,\sigma}(\varphi)(\Omega_{i,\sigma}(\varphi)),I_{\infty,\sigma}(\varphi)(\Omega_{d+1-i,\sigma}(\varphi))\rangle_{\dR,\sigma,\varphi}\]
for every $\varphi$, and thus
\begin{equation}\label{Omegaesmu} \langle\Omega_{i,\sigma},\Omega_{d+1-i,\sigma}\rangle_{\sigma}=\sigma(\mu_{i,\sigma})\delta_{\sigma}(A)^{-1}\in E\otimes\C.\end{equation}
(see Remark~\ref{rem:twist}). Since the realization is regular, there exists $\lambda_{i,\sigma}(\varphi)\in\C^{\times}$ for every $i$, $\sigma$ and $\varphi$, such that
\[ F_{\sigma,\C}(\varphi)(\Omega_{i,\sigma}(\varphi))=\lambda_{i,\sigma}(\varphi)\Omega_{d+1-i,\sigma}(\varphi).\]
Applying $F_{\sigma,\C}(\varphi)$ a second time, we get that $\lambda_{i,\sigma}(\varphi)\lambda_{d+1-i,\sigma}(\varphi)=1$. In particular, if $d$ is odd, $\lambda_{(d+1)/2,\sigma}(\varphi)=\pm1$, and the sign is independent of $\varphi$, because in fact it is the scalar $\varepsilon$ by which $F_{\sigma,\C}$ acts on $M_{\sigma}^{w/2,w/2}$. Thus, if we let $\lambda_{i,\sigma}\in(E\otimes\C)^{\times}$ be the elements whose $\varphi$-components are $\lambda_{i,\sigma}(\varphi)$, then $\lambda_{i,\sigma}\lambda_{d+1-i,\sigma}=1$, and $\lambda_{(d+1)/2,\sigma}=\varepsilon=\pm1$ if $d$ is odd.

For $i=1,\dots,d$, let $Q_{i,\sigma}\in E\otimes\C$ be defined by the formula
\[ Q_{i,\sigma}=\langle\Omega_{i,\sigma},F_{\sigma,\C}(\Omega_{i,\sigma})\rangle_{\sigma}.\] 
From the definition of the scalars $\lambda_{i,\sigma}$ and~(\ref{Omegaesmu}), it follows that 
\[ Q_{i,\sigma}=\lambda_{i,\sigma}\sigma(\mu_{i,\sigma})\delta_{\sigma}(A)^{-1}\]
and it's an element of $\in(E\otimes\C)^{\times}$. Moreover,
\[ Q_{i,\sigma}^{-1}F_{\sigma,\C}(\Omega_{i,\sigma})=\sigma(\mu_{i,\sigma}^{-1})\delta_{\sigma}(A)\Omega_{d+1-i,\sigma}.\]

\begin{lemma}\label{lemmadualityQ}
Suppose that $A=E(0)$ is trivial. Let $0\leq r<s\leq d$ be integers such that $d=r+s$. Then
\[ \prod_{i=1}^{r}Q_{i,\sigma}\sim_{E\otimes K,\sigma}\prod_{i=1}^{s}Q_{i,\sigma}.\]
\begin{proof} We write $\sim$ for $\sim_{E\otimes K,\sigma}$ throughout. Since $Q_{i,\sigma}\sim\lambda_{i,\sigma}$ and $\lambda_{i,\sigma}\lambda_{d+1-i,\sigma}=1$, it follows that $Q_{i,\sigma}\sim Q_{d+1-i,\sigma}^{-1}$. If $r+1\leq i\leq s$, then $r+1\leq d+1-i\leq s$ as well. Finally, $Q_{(d+1)/2,\sigma}\sim 1$ if $d$ is odd, because in this case $\lambda_{(d+1)/2,\sigma}=\pm1$. This is enough to prove the lemma.
\end{proof}
\end{lemma}

\begin{prop}\label{propc+c-} Let $M\in\mathcal{R}(K)_{E}$ be a regular, special realization of weight $w$ and rank $d$, endowed with an $A$-polarization, and let $\sigma\in J_{K}$. Then
\[ c^{+}_{\sigma}(M)c^{-}_{\sigma}(M)\sim_{E\otimes K,\sigma}(2\pi i)^{-dw}\delta_{\sigma}(M)^{-1}\delta_{\sigma}(A)^{d+\lfloor d/2\rfloor}\prod_{j=1}^{\lfloor d/2\rfloor}Q_{j,\sigma}.\]

\begin{proof} We pick up the notation from Remark~\ref{rem:c+reg}, so that $\delta_{\sigma}(M)\sim\det(\tilde P_{\sigma})^{-1}$. Now, recalling that $Q_{j,\sigma}=\lambda_{j,\sigma}\sigma(\mu_{j,\sigma})\delta_{\sigma}(A)^{-1}$, by applying $F_{\sigma,\C}$ to~(\ref{form:comblinealO}) we obtain that
\begin{equation}\label{relationQ} \tilde a_{i,d+1-j,\sigma}^{\pm}=\pm\delta_{\sigma}(A)^{-1}\sigma(\mu_{j,\sigma})Q_{j,\sigma}^{-1}\tilde a_{ij,\sigma}^{\pm}\quad i=1,\dots,d^{\pm},j=1\dots,d.\end{equation}
It follows that $\delta_{\sigma}(M)^{-1}\sim\delta_{\sigma}(A)^{-\lfloor d/2\rfloor}\prod_{j=1}^{\lfloor d/2\rfloor}Q_{j,\sigma}^{-1}\det(H_{\sigma})$, where
\[ H_{\sigma}= 
	\left(\begin{array}{cc} (\tilde a_{ij,\sigma}^{+})_{ \substack{1\leq i\leq d^{+}\\1\leq j\leq d^{+}} } &  (\tilde a_{i,d+1-j,\sigma}^{+})_{\substack{1\leq i\leq d^{+}\\d^{+}+1\leq j\leq d} }\\
	                                  (\tilde a_{ij,\sigma}^{-})_{ \substack{1\leq i\leq d^{-}\\1\leq j\leq d^{+}} }  & (-\tilde a_{i,d+1-j,\sigma}^{-})_{ \substack{1\leq i\leq d^{-}\\d^{+}+1\leq j\leq d} }
	\end{array}\right)
\]
if either $d$ is even or $d=2m-1$ is odd and $d^{-}=m-1$, and
\[ H_{\sigma}= 
	\left(\begin{array}{cc} (\tilde a_{ij,\sigma}^{+})_{ \substack{1\leq i\leq m-1\\1\leq j\leq m} } & (\tilde a_{i,d+1-j,\sigma}^{+})_{\substack{1\leq i\leq m-1\\m+1\leq j\leq d} }\\
	                                  (\tilde a_{ij,\sigma}^{-})_{ \substack{1\leq i\leq m\\1\leq j\leq m} }  & (-\tilde a_{i,d+1-j,\sigma}^{-})_{ \substack{1\leq i\leq m\\m+1\leq j\leq d} }
	\end{array}\right)
\]
if $d=2m-1$ is odd and $d^{-}=m$. Suppose first that $d=2m$ is even. Then, applying elementary column operations, we can take $H_{\sigma}$ to the matrix
\[ 	\left(\begin{array}{cc} (2\tilde a_{ij,\sigma}^{+})_{ \substack{1\leq i\leq d^{+}\\1\leq j\leq d^{+}} } &  (\tilde a_{i,d+1-j,\sigma}^{+})_{\substack{1\leq i\leq d^{+}\\d^{+}+1\leq j\leq d} }\\
	                                  (0)_{ \substack{1\leq i\leq d^{-}\\1\leq j\leq d^{+}} }  & (-\tilde a_{i,d+1-j,\sigma}^{-})_{ \substack{1\leq i\leq d^{-}\\d^{+}+1\leq j\leq d} }
	\end{array}\right).
\]
It follows that
\begin{equation}\label{exprdeltaQ} \delta_{\sigma}(M)^{-1}\sim\delta_{\sigma}(A)^{-\lfloor d/2\rfloor}\prod_{j=1}^{\lfloor d/2\rfloor}Q_{j,\sigma}^{-1}\cdot\det\left((\tilde a_{ij,\sigma}^{+})_{i,j=1}^{d^{+}}\right)\cdot\det\left((\tilde a_{ij,\sigma}^{-})_{i,j=1}^{d^{-}}\right).\end{equation}
Suppose now that $d=2m-1$ is odd and that $d^{-}=m-1$. Then we can take $H_{\sigma}$ to
\[ 	\left(\begin{array}{ccc} (2\tilde a_{ij,\sigma}^{+})_{ \substack{1\leq i\leq m\\1\leq j\leq m-1} } & (\tilde a_{i,m,\sigma}^{+})_{1\leq i\leq m} & (\tilde a_{i,m-j,\sigma}^{+})_{\substack{1\leq i\leq m\\1\leq j\leq m-1} }\\
	                                  (0)_{ \substack{1\leq i\leq m-1\\1\leq j\leq m-1} }  & (\tilde a_{i,m,\sigma}^{-})_{1\leq i\leq m-1} & (-\tilde a_{i,m-j,\sigma}^{-})_{ \substack{1\leq i\leq m-1\\1\leq j\leq m-1} }
	\end{array}\right).
\]
But we also have that $F_{\sigma,\C}(\Omega_{m,\sigma})=\lambda_{m,\sigma}\Omega_{m,\sigma}$, with $\lambda_{m,\sigma}=\pm1$. Since we are assuming that $d^{+}=m$, it follows that $\lambda_{m,\sigma}=\varepsilon=1$. This implies that $a_{i,m,\sigma}^{-}=0$ for every $i=1,\dots,m-1$, and thus we obtain~(\ref{exprdeltaQ}). The case where $d^{-}=m$ is analyzed in a similar way, and in all cases we obtain the formula~(\ref{exprdeltaQ}).

Now, it follows from~(\ref{form:widec+}) that $c^{\pm}_{\sigma}(M^{\vee})\sim\det\left((\tilde a_{ij,\sigma}^{\pm})_{i,j,=1}^{d^{\pm}}\right)$. Finally, the $A$-polarization defines an isomorphism $M^{\vee}\otimes_{E}A\cong M(w)$, and since
\[ c_{\sigma}^{+}(M^{\vee}\otimes_{E}A)c_{\sigma}^{-}(M^{\vee}\otimes_{E}A)\sim c_{\sigma}^{+}(M^{\vee})c_{\sigma}^{-}(M^{\vee})\delta_{\sigma}(A)^{d}\]
and
\[ c_{\sigma}^{+}(M(w))c_{\sigma}^{-}(M(w))\sim c_{\sigma}^{+}(M)c_{\sigma}^{-}(M)(2\pi i)^{wd},\]
it follows that
\[ \delta_{\sigma}^{-1}(M)\sim \prod_{j=1}^{[d/2]}Q_{j,\sigma}^{-1}c^{+}_{\sigma}(M)c^{-}_{\sigma}(M)(2\pi i)^{wd}\delta_{\sigma}(A)^{-d-\lfloor d/2\rfloor},\]
which is the expression we wanted to obtain.
\end{proof}
\end{prop}

\begin{rem}\label{remdelta} The fact that $M$ has an $A$-polarization implies that 
\begin{equation}\label{delta2} \delta_{\sigma}(M)^{2}\sim_{E\otimes K,\sigma} \delta_{\sigma}(A)^{d}(2\pi i)^{-wd}.\end{equation}
Assume, moreover, that $M$ is regular and special. Then this expression can also be written in terms of discriminants as in 1.4.12, \cite{harriscrelle}. One can even take the square root on both sides when the pairing $\langle,\rangle$ is alternated (and hence $d$ is even) to get $\delta_{\sigma}(M)\sim\delta_{\sigma}(A)^{d/2}(2\pi i)^{-wd/2}$.
\end{rem}

\subsection{The motives $RM(\chi)$}\label{motivesRM} Let $L/K$ be a totally imaginary quadratic extension of the totally real field $K$. Let $\chi$ be an algebraic Hecke character of $L$ with infinity type $(n_{\tau})_{\tau\in J_{L}}$. Let $RM(\chi)=R_{L/K}M(\chi)\in\mathcal{R}(K)_{\Q(\chi)}$. Then $RM(\chi)$ is of rank $2$ over $\Q(\chi)$, pure of weight $w(\chi)$. Note that given $\sigma\in J_K$ and $\rho\in J_{\Q(\chi)}$, $RM(\chi)_{\sigma}^{pq}(\rho)\neq 0$ if and only if $(p,q)=(n(\tau,\rho),n(\tau,\overline\rho))$ or $(p,q)=(n(\tau,\overline\rho),n(\tau,\rho))$. Moreover, $RM(\chi)$ is regular if and only if $n_{\tau}\neq n_{\overline\tau}$ for every $\tau\in J_{L}$, in which case we call $\chi$ \emph{critical}. The numbers $p_{i}^{\chi}(\sigma,\rho)$ ($i=1,2$) for the motive $RM(\chi)$ are given by $\{p_{1}^{\chi}(\sigma,\rho),p_{2}^{\chi}(\sigma,\rho)\}=\{n(\tau,\rho),n(\overline\tau,\rho)\}$. For each $\sigma\in J_{K}$ and $\rho\in J_{\Q(\chi)}$, define
\[ t_{\sigma,\rho}(\chi)=p_{1}^{\chi}(\sigma,\rho)-p_{2}^{\chi}(\sigma,\rho).\]

Let $\varepsilon_{L}:\A_{K}^{\times}/K^{\times}\to\{\pm1\}$ denote the quadratic character attached to $L/K$. Since $K$ is totally real, $\chi|_{\A_{K}^{\times}/K^{\times}}=\chi_{0}\|\cdot\|^{-w(\chi)}$, where $\chi_{0}$ is of finite order. It's easy to see, as in (1.6.2) of \cite{harriscrelle}, that $\Lambda^{2}_{\Q(\chi)}RM(\chi)\cong[\chi_{0}\varepsilon_{L}](-w(\chi))$. This map defines a morphism
\begin{equation}\label{almostpol} \langle,\rangle:RM(\chi)\otimes_{\Q(\chi)} RM(\chi)\to[\chi_{0}\varepsilon_{L}](-w(\chi)) \end{equation}
which is alternated and non-degenerate; in particular, it is an $A$-polarization of $RM(\chi)$, with $A=[\chi_{0}\varepsilon_{L}]$. For every $\sigma\in J_{K}$, we let $\delta_{\sigma}(\chi)=\delta_{\sigma}(RM(\chi))$. By Remark~\ref{remdelta}, $\delta_{\sigma}(\chi)\sim_{\Q(\chi)\otimes K,\sigma}(2\pi i)^{-w(\chi)}\delta_{\sigma}[\chi_{0}\varepsilon_{L}]$.

Assume from now on that $\chi$ is critical, and let $c^{\pm}_{\sigma}(\chi)=c^{\pm}_{\sigma}(RM(\chi))$. Thus, by Proposition~\ref{propc+c-},
\begin{equation}\label{propRM} c_{\sigma}^{+}(\chi)c_{\sigma}^{-}(\chi)\sim_{\Q(\chi)\otimes K,\sigma}(2\pi i)^{-2w(\chi)}\delta_{\sigma}(\chi)^{-1}\delta_{\sigma}[\chi_{0}\varepsilon_{L}]^{3}Q_{\sigma}(\chi)\sim\end{equation}
\[(2\pi i)^{-w(\chi)}\delta_{\sigma}[\chi_{0}\varepsilon_{L}]^{2}Q_{\sigma}(\chi).\]
The element $Q_{\sigma}(\chi)$ is defined as in Subsection~\ref{ssec:quadraticperiods}. To be more precise and set up some further notation, we recall its definition. We take a $\Q(\chi)\otimes K$-basis $\{\omega_{\sigma}(\chi),\omega_{\sigma}'(\chi)\}$ of $RM(\chi)_{\dR}$ with the property that $\{\omega_{\sigma}(\chi)\}$ is a basis of $F^{\pm}(RM(\chi)_{\dR})$. For every $\rho\in J_{\Q(\chi)}$, we let $\Omega_{\sigma}(\chi)(\rho)=I_{\infty,\sigma,1}(\rho)^{-1}(\overline{\omega_{\sigma}(\chi)(\rho)})$ and $\Omega'_{\sigma}(\chi)(\rho)=I_{\infty,\sigma,2}(\rho)^{-1}(\overline{\omega_{\sigma}'(\chi)(\rho)})$,
and we let $\Omega_{\sigma}(\chi),\Omega'_{\sigma}(\chi)\in RM(\chi)_{\sigma}\otimes\C$ be the elements whose $\rho$-components are the $\Omega_{\sigma}(\chi)(\rho)$ and $\Omega'_{\sigma}(\chi)(\rho)$ respectively. Then $Q_{\sigma}(\chi)=\langle\Omega_{\sigma}(\chi),F_{\sigma,\C}(\Omega_{\sigma}(\chi))\rangle_{\sigma}$. We also write
\[ I_{\infty,\sigma}(\Omega_{\sigma}'(\chi))=x_{\sigma}(\chi)\omega_{\sigma}(\chi)+\omega_{\sigma}'(\chi)\]
with $x_{\sigma}(\chi)\in E\otimes\C$.

Let $\chi$ be critical. For $\tau\in J_{L}$ and $\rho\in J_{\Q(\chi)}$, let $e_{\tau,\rho}=1$ if $n(\tau,\rho)>n(\overline\tau,\rho)$ and $e_{\tau,\rho}=-1$ if $n(\tau,\rho)<n(\overline\tau,\rho)$. Let $e_{\tau}=(e_{\tau,\rho})_{\rho\in J_{\Q(\chi)}}\in(\Q(\chi)\otimes\C)^{\times}$. Note that $e_{\tau}=-e_{\overline\tau}$, and in particular, $e_{\tau}\sim_{\Q(\chi)}e_{\overline\tau}$. See also H.4, \cite{blasiusperiods}. 

\begin{lemma} Let $\tau\in J_{L}$ and $\sigma=\tau|_{K}$. Then
\[ c_{\sigma}^{-}(\chi)\sim_{\Q(\chi)\otimes K,\sigma}e_{\tau}c_{\sigma}^{+}(\chi).\]
\begin{proof} 

Let $\{\gamma_{\tau}\}$ be a $\Q(\chi)$-basis of $M(\chi)_{\tau}$, and let $\gamma_{\overline\tau}=F_{\tau}(\gamma_{\tau})$, a basis of $M(\chi)_{\overline\tau}$. We can then form a basis $\{(\gamma_{\tau},0),(0,\gamma_{\overline\tau})\}$ of $RM(\chi)_{\sigma}=M(\chi)_{\tau}\oplus M(\chi)_{\overline\tau}$. Then $\{(\gamma_{\tau},\gamma_{\overline\tau})\}$ (resp. $\{(\gamma_{\tau},-\gamma_{\overline\tau})\}$) is a $\Q(\chi)$-basis of $RM(\chi)_{\sigma}^{+}$ (resp. $RM(\chi)_{\sigma}^{-}$). 
Write 
\[ I_{\infty,\sigma}((\gamma_{\tau},0)_{\C})=a_{\sigma}\omega_{\sigma}(\chi)+b_{\sigma}\omega'_{\sigma}(\chi),\]
\[ I_{\infty,\sigma}((0,\gamma_{\overline\tau})_{\C})=c_{\sigma}\omega_{\sigma}(\chi)+d_{\sigma}\omega'_{\sigma}(\chi),\]
with $a_{\sigma}$, $b_{\sigma}$, $c_{\sigma}$ and $d_{\sigma}$ in $\Q(\chi)\otimes\C$. It follows that $c_{\sigma}^{+}(\chi)\sim_{\Q(\chi)\otimes K,\sigma}b_{\sigma}+d_{\sigma}$ and $c_{\sigma}^{-}(\chi)\sim_{\Q(\chi)\otimes K,\sigma}b_{\sigma}-d_{\sigma}$.

Now fix $\rho\in J_{\Q(\chi)}$. Suppose that $e_{\tau,\rho}=-1$, so that $n(\tau,\rho)<n(\overline\tau,\rho)$, $p_{1}^{\chi}(\sigma,\rho)=n(\overline\tau,\rho)$ and $F^{\pm}(RM(\chi)_{\dR}\otimes_{\Q(\chi)\otimes K,\rho\otimes\sigma}\C=F^{n(\overline\tau,\rho)}\otimes_{\Q(\chi)\otimes K,\rho\otimes\sigma}\C$. Then $\gamma_{\overline\tau,\rho}\in M(\chi)_{\overline\tau}(\rho)$ contributes to the $(n(\overline\tau,\rho),n(\tau,\rho))$ Hodge component, so it is in $F^{p_{1}}$, and thus $d_{\sigma,\rho}=0$. Similarly, if $e_{\tau,\rho}=1$, then $b_{\sigma,\rho}=0$. This proves that $b_{\sigma}+d_{\sigma}=e_{\tau}(b_{\sigma}-d_{\sigma})$, which is what we wanted.
\end{proof}
\end{lemma}

For future reference, let $a_{\sigma}^{\pm}(\chi)\sim c_{\sigma}^{\pm}(RM(\chi)^{\vee})$. We deduce the following formulas (where $\tau$ is any element of $J_{L}$ extending $\sigma$):
\begin{equation}\label{formulaQsigmachic2}
Q_{\sigma}(\chi)\sim_{\Q(\chi)\otimes K,\sigma}(2\pi i)^{w(\chi)}\delta_{\sigma}[\chi_{0}\varepsilon_{L}]^{-2}e_{\tau}c_{\sigma}^{+}(\chi)^{2}.
\end{equation}
\begin{equation}\label{formulaasigmachisigno}
a_{\sigma}^{-}(\chi)\sim_{\Q(\chi)\otimes K,\sigma}e_{\tau}a_{\sigma}^{+}(\chi).
\end{equation}

The following lemma will be useful to ignore integer powers of $\delta_{\sigma}[\varepsilon_{L}]$ if one is willing to work over the Galois closure $L'$ of $L$.

\begin{lemma}\label{epsilonenL} For any $\sigma\in J_{K}$, $\delta_{\sigma}[\varepsilon_{L}]\in L'\subset\C$.
\begin{proof} By the calculations in Remark~\ref{rem:twist}, we can take $\delta_{\sigma}[\varepsilon_{L}]=\tilde\sigma(\zeta)^{-1}$. Here $\tilde\sigma$ is an extension of $\sigma$ to an embedding of $\overline K$ in $\C$, and $\zeta$ is a $K$-basis of $(\Q\otimes\overline K)^{\Gamma_{K}}$, where $\gamma\in\Gamma_{K}$ act by sending $q\otimes z$ ($q\in\Q$, $z\in\overline K$) to $\varepsilon_{L}(\gamma)q\otimes\gamma(z)$. Thus, we can take $\zeta=1\otimes\alpha$, where $\alpha\in L^{\times}$ satisfies $\iota(\alpha)=-\alpha$, which shows that $\tilde\sigma(\zeta)^{-1}$ is in $L'$. 
\end{proof}
\end{lemma}

\subsection{Periods of Hecke twists}\label{periodhecketwists} Let $L/K$ be a CM extension, $E'$ any number field, $M\in\mathcal{R}(K)_{E}$ and $N\in\mathcal{R}(L)_{E'}$. Denote by $RN=R_{L/K}N\in\mathcal{R}(K)_{E'}$ (restriction of scalars). 

\begin{lemma}\label{lemapp} Suppose that $M$ and $N$ are pure. Then $M\otimes RN$ has critical values if and only if for every $\sigma\in J_{K}$, $(M\otimes RN)_{\sigma}^{pp}=0$ for every $p\in\Z$.
\end{lemma}
\begin{proof} For simplicity, we will assume that $E\otimes E'$ is a field. Let $w$ and $w_{N}$ denote the weights of $M$ and $N$ respectively. Note that if $\sigma\in J_K$, then
\[ (M\otimes RN)_{\sigma}^{pq}=\bigoplus_{a+b=p}M_{\sigma}^{a,w-a}\otimes_{\C}RN_{\sigma}^{b,w_{N}-b}.\]
Suppose that, for some $p\in\Z$ and $\sigma\in J_K$, $(M\otimes RN)_{\sigma}^{pp}\neq 0$. Then there exists a pair $(a,b)$ with $a+b=p=w-a+w_{N}-b$ and $M_{\sigma}^{a,w-a}\otimes_{\C}RN_{\sigma}^{b,w_{N}-b}\neq 0$. Two possibilities arise. First, suppose that $w-a\neq a$. Then $w_{N}-b\neq b$, and there are two different summands
\[ (M\otimes RN)_{\sigma}^{pp}\supset\left(M_{\sigma}^{a,w-a}\otimes_{\C} RN_{\sigma}^{b,w_{N}-b}\right)\oplus\left(M_{\sigma}^{w-a,a}\otimes_{\C} RN_{\sigma}^{w_{N}-b,b}\right).\]
These two nonzero summands are interchanged by $F_{\sigma}$ acting on $(M\otimes RN)_{\sigma}^{pp}$, and thus $F_{\sigma}$ does not act on $(M\otimes RN)_{\sigma}^{pp}$ as a scalar. In particular, $M\otimes RN$ does not have critical values, because it's not special. Now suppose that $w-a=a$, so that $w_{N}-b=b$ as well. There is a nonzero $F_{\sigma}$-stable summand $M_{\sigma}^{a,a}\otimes_{\C}RN_{\sigma}^{b,b}$, but by definition of $RN$, this splits as
\[ \left(M_{\sigma}^{a,a}\otimes_{\C}N_{\tau}^{b,b}\right)\oplus\left(M_{\sigma}^{a,a}\otimes_{\C}N_{\overline\tau}^{b,b}\right),\]
where $\tau,\overline\tau$ are the two embeddings of $L$ over $\sigma$; also here, Frobenius interchanges these two nonzero summands, and again this implies that it does not act on $(M\otimes RN)_{\sigma}^{pp}$ as scalars.

For the converse, it's easily seen (using, for instance, (1.3.1) of \cite{deligne}) that $(M\otimes RN)(t)$ is critical, with $t=\lfloor\frac{w+w(\chi)+1}{2}\rfloor$ (the symbol $\lfloor,\rfloor$ denotes the floor function).
\end{proof}

Let $\chi$ be an algebraic Hecke character $\chi$ of $L$ of infinity type $(n_{\tau})_{\tau\in J_{L}}$. From now on, unless otherwise stated, we will assume for simplicity that $E\otimes\Q(\chi)$ is a field. 

\begin{prop} Suppose that $M\in\mathcal{R}(K)_{E}$ is regular. Then $M\otimes RM(\chi)$ has critical values if and only if, for every $\sigma\in J_{K}$ and every $\rho\in J_{\Q(\chi)}$, $t_{\sigma,\rho}(\chi)$ is not equal to any of the values $w-2p_{i}(\sigma,\varphi)$ (for any $i=1,\dots,d$ and any $\varphi\in J_{E}$).
\begin{proof}
For $\sigma\in J_{K}$, $\rho\in J_{\Q(\chi)}$ and $\varphi\in J_{E}$,
\begin{equation}\label{descompMRM} (M\otimes RM(\chi))_{\sigma}^{pq}(\varphi\otimes\rho)=\bigoplus_{i=1}^dM_{\sigma}^{p_i(\sigma,\varphi),q_i(\sigma,\varphi)}(\varphi)\otimes_{\C}RM(\chi)_{\sigma}^{p-p_i(\sigma,\varphi),q-q_i(\sigma,\varphi)}(\rho).\end{equation}
The proposition follows from this and Lemma \ref{lemapp}.
\end{proof}
\end{prop}

If $M$ is regular, $\sigma\in J_{K}$, $\varphi\in J_{E}$, and $i=0,\dots,d$, consider the interval $I_{i}(\sigma,\varphi)=(w-2p_{i}(\sigma,\varphi),w-2p_{i+1}(\sigma,\varphi))$, where for consistency of notation we take $p_{0}(\sigma,\varphi)=+\infty$ and $p_{d}(\sigma,\varphi)=-\infty$. By the last proposition, if $M$ is regular and $M\otimes RM(\chi)$ has critical values, then for each $\sigma\in J_K$, $\varphi\in J_E$ and $\rho\in J_{\Q(\chi)}$, there is a number $r=r_{\sigma,\varphi,\rho}(\chi)\in\{0,\dots,d\}$ such that $t_{\sigma,\rho}(\chi)\in I_{r}(\sigma,\varphi)$.

Write $E\otimes\Q(\chi)\otimes K\cong\prod_{\mu=1}^{m}K_{\mu}$, where the $K_{\mu}/K$ are finite extensions. By the same reasoning as in Subsection~\ref{ssec:quadraticperiods}, we can unambiguously define $p_{i}(\sigma,\mu)$, $p_{i}^{\chi}(\sigma,\mu)$ and $r_{\sigma,\mu}(\chi)$, by declaring them to be $p_{i}(\sigma,\varphi)$, $p_{i}^{\chi}(\sigma,\rho)$ and $r_{\sigma,\varphi,\rho}(\chi)$ respectively, where $(\varphi,\rho)$ is any pair of embeddings corresponding to the index $\mu$. This will be useful later to construct a suitable rational basis of $(M\otimes RM(\chi))_{\dR}$.

Retaining the hypothesis that $M$ and $\chi$ are regular, with $M\otimes RM(\chi)$ having critical values, we can determine the set of critical integers of $M\otimes RM(\chi)$ as follows. Fix $\sigma\in J_{K}$, $\varphi\in J_{E}$ and $\rho\in J_{\Q(\chi)}$. Then~(\ref{descompMRM}) shows that
\[ \left\{(p,q):h_{\sigma}^{pq}(\varphi\otimes\rho)\neq 0,\quad p<q\right\}=\]
\[ \left\{ (p_{i}(\sigma,\varphi)+p_{1}^{\chi}(\sigma,\rho),q_{i}(\sigma,\varphi)+p_{2}^{\chi}(\sigma,\rho)) \right\}_{i=r+1}^{d}\cup\]\[ \left\{ (p_{i}(\sigma,\varphi)+p_{2}^{\chi}(\sigma,\rho),q_{i}(\sigma,\varphi)+p_{1}^{\chi}(\sigma,\rho)) \right\}_{i=d+1-r}^{d}, \]
where $r=r_{\sigma,\varphi,\rho}(\chi)$. Note that the two sets in the union may have non-trivial intersection. However, the elements in the intersection contribute with $h_{\sigma}^{pq}(\varphi\otimes\rho)=2$, while the elements that are not in the intersection contribute with $h_{\sigma}^{pq}(\varphi\otimes\rho)=1$. It follows from the construction of $L_{\infty}(\varphi\otimes\rho,M\otimes RM(\chi),s)$ that this equals
\[ \prod_{\sigma\in J_{K}}\left( \prod_{i=r_{\sigma,\varphi,\rho}(\chi)+1}^{d}\Gamma_{\C}(s-p_{i}(\sigma,\varphi)-p_{1}^{\chi}(\sigma,\rho))\right.\]\[\left. \prod_{i=d+1-r_{\sigma,\varphi,\rho}(\chi)}^{d}\Gamma_{\C}(s-p_{i}(\sigma,\varphi)-p_{2}^{\chi}(\sigma,\rho)) \right). \]
Thus, letting 
\[ \upsilon^{(1)}_{\varphi,\rho}(\chi)=\max\{ p_{r_{\sigma,\varphi,\rho}(\chi)+1}(\sigma,\varphi)+p_{1}^{\chi}(\sigma,\rho),p_{d+1-r_{\sigma,\varphi,\rho}(\chi)}(\sigma,\varphi)+p_{2}^{\chi}(\sigma,\rho) \}_{\sigma\in J_{K}}, \]
the set of poles of $L_{\infty}(\varphi\otimes\rho,M\otimes RM(\chi),s)$ consists of the set of integers $m$ such that $m\leq\upsilon^{(1)}_{\varphi,\rho}(\chi)$
Similarly, letting
\[ \upsilon^{(2)}_{\varphi,\rho}(\chi)=\min\{ p_{r_{\sigma,\varphi,\rho}(\chi)}(\sigma,\varphi)+p_{1}^{\chi}(\sigma,\rho),p_{d-r_{\sigma,\varphi,\rho}(\chi)}(\sigma,\varphi)+p_{2}^{\chi}(\sigma,\rho) \}_{\sigma\in J_{K}}, \]
the set of poles of $L_{\infty}(\varphi\otimes\rho,(M\otimes RM(\chi))^{\vee},1-s)$ consists of the set of integers $m$ such that $m\geq\upsilon^{(2)}_{\varphi,\rho}(\chi)+1$. We conclude that the set of critical integers of $M\otimes RM(\chi)$ consists of the set of integers $m$ such that
\begin{equation}
\label{setofcriticalvalues} \upsilon_{\varphi,\rho}^{(1)}(\chi)<m\leq\upsilon_{\varphi,\rho}^{(2)}(\chi).
\end{equation}

The following theorem is the main result of this section, and it generalizes Proposition 1.7.6 of \cite{harriscrelle}. The notation $\lceil d/2\rceil$ refers to the ceiling function.

\begin{thm}\label{thmfact} Let $M$ be a regular, special realization over $K$ with coefficients in $E$, of rank $d$ and weight $w$, endowed with a polarization. Suppose that $\chi$ is a critical algebraic Hecke character of $L$ of weight $w(\chi)$. Assume that $M\otimes RM(\chi)$ has critical values. Let $\mathbf{r}_{\sigma}=(r_{\sigma,\varphi,\rho}(\chi))_{\varphi\otimes\rho}$ and $\mathbf{s}_{\sigma}=(d-r_{\sigma,\varphi,\rho}(\chi))_{\varphi\otimes\rho}$. Then 
\[ c^{+}_{\sigma}(M\otimes RM(\chi))\sim(2\pi i)^{-\lceil d/2\rceil w(\chi)}\delta_{\sigma}([\chi_{0}\varepsilon_{L}])^{\mathbf{r}_{\sigma}}\delta_{\sigma}(M)a_{\sigma}^{*}(\chi)Q_{\sigma}(\chi)^{\mathbf{r}_{\sigma}-\lceil d/2\rceil}\prod_{j=1}^{\mathbf{s}_{\sigma}}Q_{j,\sigma},\]
where $a_{\sigma}^{*}(\chi)=1$ if $d$ is even, and $a_{\sigma}^{*}(\chi)=a^{\pm}_{\sigma}(\chi)$ if $d$ is odd, with $\pm=-$ if $d^{+}>d^{-}$ and $\pm=+$ if $d^{-}>d^{+}$. In the formula, $\sim$ means $\sim_{(E\otimes\Q(\chi))\otimes K,\sigma}$.

\begin{rem} If $K=\Q$, then $\mathbf{r}_{1}=(r_{1,\varphi,\rho})$, and the integers $r_{1,\varphi,\rho}$ are actually independent of the embeddings. This extra simplicity comes from the fact that the nonzero Hodge components of $M$ and $RM(\chi)$ are free of rank $1$ over $E\otimes\C$ and $\Q(\chi)\otimes\C$ respectively in this case (see Proposition 2.5 of \cite{deligne}), which is not true in general if $[K:\Q]>1$.
\end{rem}

\begin{proof} 
Recall that we are assuming that $E\otimes\Q(\chi)$ is a field; the general case basically follows from this case. The polarization of $M$ and the $[\chi_{0}\varepsilon_{L}]$-polarization of $RM(\chi)$~(\ref{almostpol}), define an $A$-polarization
\[ M\otimes RM(\chi)\cong (M\otimes RM(\chi))^{\vee}\otimes_{E\otimes\Q(\chi)}[\chi_{0}\varepsilon_{L}](-w-w(\chi))\]
in $\mathcal{R}(K)_{E\otimes\Q(\chi)}$, with $A=[\chi_{0}\varepsilon_{L}]$. Note that $(\chi_{0}\varepsilon_{L})(c_{\sigma})=(-1)^{w(\chi)+1}$ for any complex conjugation $c_{\sigma}\in\Gamma_{K}$ attached to any $\sigma\in J_{K}$. Hence, the motive $[\chi_{0}\varepsilon_{L}]$ is special, and by Remark~\ref{rem:twist}, we can write
\begin{equation}\label{exprper} c_{\sigma}^{+}(M\otimes RM(\chi))\sim(2\pi i)^{(-w-w(\chi))d}\delta_{\sigma}[\chi_{0}\varepsilon_{L}]^{d}c_{\sigma}^{(-1)^{w+1}}((M\otimes RM(\chi))^{\vee}).\end{equation}
Here, and in the rest of the proof, we write $\sim$ for $\sim_{(E\otimes\Q(\chi))\otimes K,\sigma}$.

For $\sigma\in J_K$, $\varphi\in J_E$ and $\rho\in J_{\Q(\chi)}$, denote by $I_{\infty,\sigma}(\varphi\otimes\rho)$ the comparison isomorphism between $(M\otimes RM(\chi))_{\sigma}(\varphi\otimes\rho)$ and $(M\otimes RM(\chi))_{\dR}\otimes_{(E\otimes\Q(\chi))\otimes K,\varphi\otimes\rho\otimes\sigma}\C$. Note that $d^{\pm}(M\otimes RM(\chi))=d$, $F^+((M\otimes RM(\chi))_{\dR})=F^-((M\otimes RM(\chi)_{\dR})$, and

\begin{align}\label{filtr} I_{\infty,\sigma}(\varphi,\rho)^{-1}(F^{\pm}((M\otimes RM(\chi))_{\dR})\otimes_{(E\otimes\Q(\chi))\otimes K,\varphi\otimes\rho\otimes\sigma}\C)=\notag\\
\left(\bigoplus_{i=1}^{r_{\sigma,\varphi,\rho}(\chi)}M_{\sigma}^{p_{i}(\sigma,\varphi),q_{i}(\sigma,\varphi)}(\varphi)\otimes_{\C}RM(\chi)_{\sigma}^{p_{1}^{\chi}(\sigma,\rho),p_{2}^{\chi}(\sigma,\rho)}(\rho)\right)\oplus\notag\\ \left(\bigoplus_{i=1}^{d-r_{\sigma,\varphi,\rho}(\chi)}M_{\sigma}^{p_{i}(\sigma,\varphi),q_{i}(\sigma,\varphi)}(\varphi)\otimes_{\C}RM(\chi)_{\sigma}^{p_{2}^{\chi}(\sigma,\rho),p_{1}^{\chi}(\sigma,\rho)}(\rho)\right).\end{align}
Also,
\[ (M\otimes RM(\chi))_{\sigma}^{+}=\left(M_{\sigma}^{+}\otimes RM(\chi)_{\sigma}^{+}\right)\oplus\left(M_{\sigma}^{-}\otimes RM(\chi)_{\sigma}^{-}\right),\]
\[ (M\otimes RM(\chi))_{\sigma}^{-}=\left(M_{\sigma}^{-}\otimes RM(\chi)_{\sigma}^{+}\right)\oplus\left(M_{\sigma}^{+}\otimes RM(\chi)_{\sigma}^{-}\right).\]
Choose $E$-bases $\{e_{i,\sigma}\}_{i=1,\dots,d^+}$, $\{f_{i,\sigma}\}_{i=1,\dots,d^-}$ of $M_{\sigma}^+$ and $M_{\sigma}^-$ respectively, and choose $\Q(\chi)$-bases $\{e_{\sigma}(\chi)\}$ and $\{f_\sigma(\chi)\}$ of $RM(\chi)_{\sigma}^+$ and $RM(\chi)_{\sigma}^-$ respectively. Also, choose an $E\otimes K$-bases $\{\omega_{i,\sigma}\}_{i=1,\dots,d}$ of $M_{\dR}$ as in Subsection~\ref{ssec:quadraticperiods}. Take as well a $\Q(\chi)\otimes K$-basis $\{\omega_{\sigma}(\chi),\omega_{\sigma}'(\chi)\}$ as in Subsection~\ref{motivesRM}. In particular, $\{\omega_{\sigma}(\chi)(\rho)\}$ is a $\C$-basis of $RM(\chi)_{\sigma}^{p_{1}^{\chi}(\sigma,\rho),p_{2}^{\chi}(\sigma,\rho)}(\rho)$. For each $i=1,\dots,d$, we let 
\[ \omega_{i,\sigma}(\chi)=\omega_{i,\sigma}\otimes\omega_{\sigma}(\chi),\quad\omega_{i,\sigma}'(\chi)=\omega_{i,\sigma}\otimes\omega_{\sigma}'(\chi).\]
Let 
\[ \{e_{i,\sigma}\otimes e_{\sigma}(\chi)\}_{i=1,\dots,d^{+}}\cup\{f_{i,\sigma}\otimes f_{\sigma}(\chi)\}_{i=1,\dots,d^{-}}\]
and 
\[ \{e_{i,\sigma}\otimes f_{\sigma}(\chi)\}_{i=1,\dots,d^{+}}\cup\{f_{i,\sigma}\otimes e_{\sigma}(\chi)\}_{i=1,\dots,d^{-}}\]
be the $E\otimes\Q(\chi)$-bases of $(M\otimes RM(\chi))_{\sigma}^{+}$ and $(M\otimes RM(\chi))_{\sigma}^{-}$ respectively, obtained from the chosen bases for $M$ and $RM(\chi)$. Their union is a basis of $(M\otimes RM(\chi))_{\sigma}$. Consider also the $E\otimes\Q(\chi)\otimes K$-basis of $(M\otimes RM(\chi))_{\dR}$ given by $\{\omega_{i,\sigma}(\chi),\omega_{i,\sigma}'(\chi)\}_{i=1,\dots,d}$. We want to use Remark~\ref{rem:c+reg} to compute the Deligne periods of the dual of $M\otimes RM(\chi)$. For that we need an appropriate basis of its de Rham realization which contains a basis of the corresponding $\pm$-step. Using~(\ref{filtr}), we see that the basis $\{\omega_{i,\sigma}(\chi),\omega_{i,\sigma}'(\chi)\}_{i=1,\dots,d}$ does not exactly meet this in the first place, since the indexing for each embedding $\varphi\otimes\rho$ will depend on $r_{\sigma,\varphi,\rho}(\chi)$, which may change for different choices. However, we can form another basis by reordering this basis for each embedding, in the following way. Recall the notation $E\otimes\Q(\chi)\otimes K\cong\prod_{\mu}K_{\mu}$ introduced before, so that we can speak of the numbers $r_{\mu}=r_{\sigma,\mu}(\chi)$. We can write $(M\otimes RM(\chi))_{\dR}=\bigoplus_{\mu}(M\otimes RM(\chi))_{\dR,\mu}$. For each $\mu$, if $x\in(M\otimes RM(\chi))_{\dR}$, $x_{\mu}$ will denote the element $x\otimes_{E\otimes\Q(\chi)\otimes K}1_{K_{\mu}}$. Consider the ordered basis 
\[ \left\{\omega_{i,\sigma}(\chi)_{\mu}\right\}_{i=1}^{r_{\mu}}\cup\left\{\omega_{i,\sigma}'(\chi)_{\mu}\right\}_{i=1}^{d-r_{\mu}}\cup\left\{\omega_{i,\sigma}(\chi)_{\mu}\right\}_{i=r_{\mu}+1}^{d}\cup\left\{\omega_{i,\sigma}'(\chi)_{\mu}\right\}_{i=d+1-r_{\mu}}^{d}. \]
These bases define an $E\otimes\Q(\chi)\otimes K$-basis $\{\gamma_{1,\sigma},\dots,\gamma_{2d,\sigma}\}$ of $M\otimes RM(\chi)_{\dR}$ with the property that the first $d$ elements form a basis of $F^{\pm}$, because the decomposition~(\ref{filtr}) actually takes place over $K_{\mu}$ (where $\mu$ corresponds to $(\varphi,\rho)$). 

Let $\Omega_{i,\sigma}$, $\Omega_{\sigma}(\chi)$ and $\Omega_{\sigma}'(\chi)$ be the elements constructed in Subsection~\ref{ssec:quadraticperiods}. Let $\Omega_{i,\sigma}(\chi)=\Omega_{i,\sigma}\otimes\Omega_{\sigma}(\chi)$ and $\Omega_{i,\sigma}'(\chi)=\Omega_{i,\sigma}\otimes\Omega_{\sigma}'(\chi)$. Consider the ordered $\C$-basis of $(M\otimes RM(\chi))_{\sigma}(\varphi\otimes\rho)$ given by
\[ \left\{\Omega_{i,\sigma}(\chi)(\varphi\otimes\rho)\right\}_{i=1}^{r_{\varphi,\rho}}\cup\left\{\Omega_{i,\sigma}'(\chi)(\varphi\otimes\rho)\right\}_{i=1}^{d-r_{\varphi,\rho}}\cup\]
\[\left\{\Omega_{i,\sigma}(\chi)(\varphi\otimes\rho)\right\}_{i=r_{\varphi,\rho}+1}^{d}\cup\left\{\Omega_{i,\sigma}'(\chi)(\varphi\otimes\rho)\right\}_{i=d+1-r_{\varphi,\rho}}^{d}, \]
where $r_{\varphi,\rho}=r_{\sigma,\varphi,\rho}(\chi)$. Name the elements of this ordered basis as
\[ \{\Gamma_{1,\sigma}(\varphi\otimes\rho),\dots,\Gamma_{2d,\sigma}(\varphi\otimes\rho)\}.\]
Putting these bases together, we get an $E\otimes\Q(\chi)\otimes\C$-basis $\{\Gamma_{1,\sigma},\dots,\Gamma_{2d,\sigma}\}$ of $(M\otimes RM(\chi))_{\sigma}\otimes\C$. It's easy to see, given our choices, that $\{I_{\infty,\sigma}(\Gamma_{1,\sigma}),\dots,I_{\infty,\sigma}(\Gamma_{2d,\sigma})\}$ is obtained from $\{\gamma_{1,\sigma},\dots,\gamma_{2d,\sigma}\}$ by means of an upper triangular matrix with diagonal entries $1$, as in Remark~\ref{rem:c+reg}. Thus, to compute $c_{\sigma}^{\pm}((M\otimes RM(\chi))^{\vee})$, we can use the basis $\{\Gamma_{1,\sigma},\dots,\Gamma_{2d,\sigma}\}$. We write
\begin{align}\label{coefficientesO} \Omega_{j,\sigma}=\sum_{i=1}^{d^{+}}\tilde a_{ij,\sigma}^{+}e_{i,\sigma}+\sum_{i=1}^{d-}\tilde a_{ij,\sigma}^{-}f_{i,\sigma},\quad j=1,\dots,d,\\
\label{coefficienteschiO}\Omega_{\sigma}(\chi)=\tilde a_{\sigma}^+(\chi)e_{\sigma}(\chi)+\tilde a_{\sigma}^-(\chi)f_{\sigma}(\chi),\\ 
\notag \Omega_{\sigma}'(\chi)=\tilde b_{\sigma}^+(\chi)e_{\sigma}(\chi)+\tilde b_{\sigma}^-(\chi)f_{\sigma}(\chi),\end{align}
\begin{equation}\label{coefficientesG}\Gamma_{j,\sigma}=\sum_{i=1}^{d^{+}}b_{ij,\sigma}^{+}e_{i,\sigma}\otimes e_{\sigma}(\chi)+\sum_{i=1}^{d-}c_{ij,\sigma}^{+}f_{i,\sigma}\otimes f_{\sigma}(\chi)+\end{equation}\[\sum_{i=1}^{d^{+}}b_{ij,\sigma}^{-}e_{i,\sigma}\otimes f_{\sigma}(\chi)+\sum_{i=1}^{d^{-}}c_{ij,\sigma}^{-}f_{i,\sigma}\otimes e_{\sigma}(\chi).\]
Then
\[ c_{\sigma}^{\pm}((M\otimes RM(\chi))^{\vee})\sim\det T_{\sigma}^{\pm},\]
where
\[ T_{\sigma}^{\pm}=\left(\begin{array}{c}\left(b_{ij,\sigma}^{\pm}\right)_{\substack{i=1,\dots,d^{+}\\j=1,\dots,d}}\\ \left(c_{ij,\sigma}^{\pm}\right)_{\substack{i=1,\dots,d^{-}\\j=1,\dots,d}}\end{array}\right).\]
We now compute each $\varphi\otimes\rho$-component $\det(T_{\sigma}^{\pm})_{\varphi,\rho}=\det(T_{\sigma,\varphi,\rho}^{\pm})$. By the shape of the basis $\{\Gamma_{1,\sigma}(\varphi\otimes\rho),\dots,\Gamma_{2d,\sigma}(\varphi\otimes\rho)\}$, equations~(\ref{coefficientesO}),~(\ref{coefficienteschiO}) and~(\ref{coefficientesG}) imply that
\[ T_{\sigma,\varphi,\rho}^{\pm}=\left(\begin{array}{cc}    
                                                 \left(\tilde a_{\sigma,\rho}^{\pm}(\chi)\tilde a_{ij,\sigma,\varphi}^{+}\right)_{ \substack{1\leq i\leq d^{+}\\1\leq j\leq r} } &
                                                 \left(\tilde b_{\sigma,\rho}^{\pm}(\chi)\tilde a_{ij,\sigma,\varphi}^{+}\right)_{ \substack{1\leq i\leq d^{+}\\1\leq j\leq d-r} } \\
                                                 \left(\tilde a_{\sigma,\rho}^{\mp}(\chi)\tilde a_{ij,\sigma,\varphi}^{-}\right)_{ \substack{1\leq i\leq d^{-}\\1\leq j\leq r} } &
                                                 \left(\tilde b_{\sigma,\rho}^{\mp}(\chi)\tilde a_{ij,\sigma,\varphi}^{-}\right)_{ \substack{1\leq i\leq d^{-}\\1\leq j\leq d-r} } \end{array}\right),\]                                               
where $r=r_{\sigma,\varphi,\rho}(\chi)$. Note that $d-r<r$ because $p_{1}^{\chi}(\sigma,\rho)>p_{2}^{\chi}(\sigma,\rho)$.

Now, as in Subsection~\ref{ssec:quadraticperiods}, there are elements $\lambda_{i,\sigma}\in(E\otimes\C)^{\times}$, $\lambda_{\sigma}(\chi)\in(\Q(\chi)\otimes\C)^{\times}$, $\mu_{i,\sigma}\in(E\otimes K)^{\times}$ and $\mu_{\sigma}(\chi)\in(\Q(\chi)\otimes K)^{\times}$ such that 
\[ F_{\sigma,\C}(\Omega_{i,\sigma})=\lambda_{i,\sigma}\Omega_{d+1-i,\sigma},\quad F_{\sigma,\C}(\Omega_{\sigma}(\chi))=\lambda_{\sigma}(\chi)\Omega_{\sigma}'(\chi),\] 
\[ Q_{i,\sigma}=\lambda_{i,\sigma}\sigma(\mu_{i,\sigma}),\quad Q_{\sigma}(\chi)=\lambda_{\sigma}(\chi)\sigma(\mu_{\sigma}(\chi))\delta_{\sigma}[\chi_{0}\varepsilon_{L}]^{-1}.\]
This, together with~(\ref{coefficientesO}) and~(\ref{coefficienteschiO}), implies that
\[ \tilde a_{ij,\sigma}^{\pm}=\pm\lambda_{j,\sigma}\tilde a_{i,d+1-j,\sigma}^{\pm},\]
\[ \tilde a_{\sigma}^{\pm}(\chi)=\pm\lambda_{\sigma}(\chi)\tilde b_{\sigma}^{\pm}(\chi).\]
Thus,
\[ \det(T_{\sigma}^{\pm})=\sigma(\mu_{\sigma}(\chi))^{d-\mathbf{r}_{\sigma}}\delta_{\sigma}[\chi_{0}\varepsilon_{L}]^{\mathbf{r}_{\sigma}-d}Q_{\sigma}(\chi)^{\mathbf{r}_{\sigma}-d}(\tilde a_{\sigma}^{\pm}(\chi))^{d^{+}}(\tilde a_{\sigma}^{\mp}(\chi))^{d-}\det(V_{\sigma}^{\pm}), \]
where $V_{\sigma}^{\pm}\in\GL_{d}(E\otimes\Q(\chi)\otimes\C)$ is the matrix such that its $\varphi,\rho$-components are
\[ V_{\sigma,\varphi,\rho}^{\pm}=\left(\begin{array}{cc}    
                                                 \left(\tilde a_{ij,\sigma,\varphi}^{+}\right)_{\substack{1\leq i\leq d^{+}\\1\leq j\leq r}} &
                                                 \left(\pm\tilde a_{ij,\sigma,\varphi}^{+}\right)_{\substack{1\leq i\leq d^{+}\\1\leq j\leq d-r}} \\
                                                 \left(\tilde a_{ij,\sigma,\varphi}^{-}\right)_{\substack{1\leq i\leq d^{-}\\1\leq j\leq r}} &
                                                 \left(\mp\tilde a_{ij,\sigma,\varphi}^{-}\right)_{\substack{1\leq i\leq d^{-}\\1\leq j\leq d-r}} \end{array}\right).\]
Note that $\sigma(\mu_{\sigma}(\chi))^{d-\mathbf{r}_{\sigma}}$, as an element of $E\otimes\Q(\chi)\otimes\C$, belongs to $E\otimes\Q(\chi)\otimes K$ via $\sigma$. This follows from the decomposition $E\otimes\Q(\chi)\otimes K\cong\prod_{\mu}K_{\mu}$ and the fact that the $r_{\sigma,\varphi,\rho}(\chi)$ only depend on the index $\mu$. Using the same reasoning with $(-1)^{d-\mathbf{r}_{\sigma}}$, we get
\[ c_{\sigma}^{\pm}((M\otimes RM(\chi))^{\vee})\sim\delta_{\sigma}[\chi_{0}\varepsilon_{L}]^{\mathbf{r}_{\sigma}-d}Q_{\sigma}(\chi)^{\mathbf{r}_{\sigma}-d}(\tilde a_{\sigma}^{\pm}(\chi))^{d^{+}}(\tilde a_{\sigma}^{\mp}(\chi))^{d-}\det(V_{\sigma}^{+}).\]
Applying elementary column operations and using the relations~(\ref{relationQ}), together with the fact that the powers of $2$ and $-1$ that appear in the process belong to $E\otimes\Q(\chi)\otimes K$ by the reasoning above, we conclude that
\[ \det(V_{\sigma}^{+})\sim\det((\tilde a_{ij,\sigma,\varphi}^+)_{i,j=1,\dots,d^+})\det((\tilde a_{ij,\sigma,\varphi}^-)_{i,j=1,\dots,d^-})\prod_{j=d^{+}+1}^{\mathbf{r}_{\sigma}}Q_{j,\sigma} \] 
(for this, one needs to use at some point that when $d=2k-1$ is odd, $Q_{k,\sigma}\in(E\otimes K)^{\times}$). Using~(\ref{exprdeltaQ}) (note that $A$ is trivial in this case), we get
\[ \det(V_{\sigma}^{+})\sim\delta_{\sigma}(M)^{-1}\prod_{j=1}^{\lfloor d/2\rfloor}Q_{j,\sigma}\prod_{j=d^{+}+1}^{\mathbf{r}_{\sigma}}Q_{j,\sigma}.\]
Then it's easy to see, again using that $Q_{k,\sigma}\sim1$, that
\[ \det(V_{\sigma}^{+})\sim\delta_{\sigma}(M)^{-1}\prod_{j=1}^{\mathbf{r}_{\sigma}}Q_{j,\sigma}.\]
We can use Lemma~\ref{lemmadualityQ} (rather, it's generalization to powers indexed by embeddings of $E\otimes\Q(\chi)$; the formula in the lemma still holds because the relevant scalars in $E\otimes\Q(\chi)\otimes K$ are powers of the $\sigma(\mu_{j,\sigma})$, which we can index by the fields $K_{\mu}$ as above) to obtain
\[ \det(V_{\sigma}^{+})\sim\delta_{\sigma}(M)^{-1}\prod_{j=1}^{\mathbf{s}_{\sigma}}Q_{j,\sigma}.\]
Note that, by Remark~\ref{rem:c+reg}, $\tilde a_{\sigma}^{\pm}(\chi)\sim_{\Q(\chi)\otimes K,\sigma}a_{\sigma}^{\pm}(\chi)\sim_{\Q(\chi)\otimes K,\sigma}c_{\sigma}^{\pm}(RM(\chi)^{\vee})$. Also, by the fact that $RM(\chi)^{\vee}\cong RM(\chi)(w(\chi))\otimes_{\Q(\chi)}[\chi_{0}\varepsilon_{L}]$ and by~(\ref{propRM}), we can write
\[ c_{\sigma}^{+}(RM(\chi)^{\vee})c_{\sigma}^{-}(RM(\chi)^{\vee})\sim c_{\sigma}^{+}(\chi)c_{\sigma}^{-}(\chi)(2\pi i)^{2w(\chi)}\delta_{\sigma}[\chi_{0}\varepsilon_{L}]^{-2}\sim(2\pi i)^{w(\chi)}Q_{\sigma}(\chi).\]
We put together all these formulas with~(\ref{exprper}) to arrive at the main formulas in the statement of the theorem (we also use~(\ref{delta2})).
\end{proof}
\end{thm}

\begin{rem} There is an apparent difference between the formulas of the theorem and those of Proposition 1.7.6 of \cite{harriscrelle}. The main point is that we are also leaving the factor $\delta_{\sigma}(M)$ instead of replacing it with powers of $(2\pi i)$ and discriminant factors. We don't need to do that in this paper. 
\end{rem}

The following proposition follows from Theorem~\ref{thmfact} and its proof when $M=\Q(0)$. Alternatively, it's a simple consequence of the isomorphism $RM(\chi)\cong RM(\chi)^{\vee}\otimes_{\Q(\chi)}[\chi_{0}\varepsilon_{L}](-w(\chi))$.

\begin{prop}\label{propa+a-} Let $\chi$ be a critical algebraic Hecke character of $L$. Then
\[ c^{\pm}_{\sigma}(\chi)\sim_{\Q(\chi)\otimes K}(2\pi i)^{-w(\chi)}\delta[\chi_{0}\varepsilon_{L}]a_{\sigma}^{\mp}(\chi).\]
\end{prop}

\begin{rem}
	\label{expresiondeC-} Suppose that $M$ and $\chi$ are as in Theorem~\ref{thmfact}. Then we can also express the periods $c_{\sigma}^{-}(M\otimes RM(\chi))$ in a similar fashion. Looking at the proof of the theorem, along with~(\ref{formulaasigmachisigno}), the following is clear:
	\[ c_{\sigma}^{-}(M\otimes RM(\chi))\sim_{(E\otimes\Q(\chi))\otimes K,\sigma}e_{\sigma}'c_{\sigma}^{+}(M\otimes RM(\chi)), \]
	where $e_{\sigma}'=1$ if $d$ is even, and $e_{\sigma}'=e_{\tau}$ if $d$ is odd, for any $\tau\in J_{L}$ extending $\sigma$.
\end{rem}

\subsection{CM periods}\label{CMperiods}
In this subsection we recall the relationship between Deligne periods and CM periods for algebraic Hecke characters. The theory of CM periods as we will use it is explained in \cite{harrisunitary}, to which we refer for details. Let $L/K$ be a CM extension, and let $T^{L}=\Res_{L/\Q}\Gm{L}$. Then $(T^{L})_{\C}\simeq\prod_{\tau\in J_{L}}\Gm{\C}$. Suppose that $\eta\in X^{*}(T^{L})$ and $\chi\in X(\eta)$, and assume that $\chi$ is critical, so that $n_{\tau}\neq n_{\overline{\tau}}$ for all $\tau\in J_{L}$. Given any morphism $h:\mathbb{S}\to(T^{L})_{\R}$, the pair $(T^{L},h)$ is a Shimura datum. In \emph{op. cit.}, as in the Appendix of \cite{harriskudla}, a CM period $p(\chi;h)\in\C^{\times}$ is constructed, well defined modulo $(\Q(\chi)E_{h})^{\times}$, where $E_{h}$ is the reflex field of $(T^{L},h)$. For example, for any $\Psi\subset J_{L}$ such that $\Psi\cap\overline\Psi=\emptyset$ (for instance $\Psi=\{\tau\}$ for a single $\tau$), we can naturally construct a map $h_{\Psi}:\mathbb{S}\to T^{L}_{\R}$; we denote the corresponding periods by $p(\chi;\Psi)$, and the reflex field by $E_{\Psi}$. Concretely, $E_{\Psi}$ is the subfield of $\overline\Q$ fixed by the elements $\gamma\in\Gamma_{\Q}$ such that $\gamma\Psi=\Psi$. 

Attached to $\eta$ is a CM type $\Phi_{\eta}$, defined by the fact that $n_{\tau}>n_{\overline{\tau}}$ if and only if $\tau\in\Phi_{\eta}$ (note that what we call $n_{\tau}$ here is $-\lambda(\tau)$ in \cite{harrisunitary}). We also write $\Phi_{\chi}=\Phi_{\eta}$. The reflex field $E_{\Phi_{\eta}}$ of $\Phi_{\eta}$ is contained in $\Q(\chi)$ (in fact in $\Q(\eta)\subset\Q(\chi)$, where $\Q(\eta)$ is the field of definition of $\eta$). If $\gamma\in\Aut(\C)$, the characters $\eta^{\gamma}$ and $\chi^{\gamma}\in X(\eta^{\gamma})$ only depend on the restriction of $\gamma$ to $\Q(\chi)$, and hence we can look at the family $(\chi^{\rho},\eta^{\rho})_{\rho\in J_{\Q(\chi)}}$. If we look at embeddings $\rho$ of $\Q(\chi)E_{\Psi}$, then we can also define $\Psi^{\rho}=\tilde\rho\Psi$, for any extension $\tilde\rho$ of $\rho$ to $\C$. We write
\[ \mathbf{p}(\chi;\Psi)=\left(p(\chi^{\rho};\Psi^{\rho})\right)_{\rho\in J_{\Q(\chi)E_{\Psi}}}\in \Q(\chi)E_{\Psi}\otimes\C. \]
If $\Psi=\Phi_{\eta}$, then $\mathbf{p}(\chi;\Phi_{\eta})=\left(p(\chi^{\rho};\Phi_{\eta^{\rho}})\right)_{\rho\in J_{\Q(\chi)}}\in \Q(\chi)\otimes\C$ (note that $\Phi_{\eta^{\rho}}=\rho\Phi_{\eta})$. 

The following formula is due to Blasius. We use the statement given as Proposition 1.8.1 of \cite{harrisunitary} (corrected as in the Introduction to \cite{harriscrelle}, that changes $\chi$ for $\check{\chi}$), combined with Deligne's conjecture for the motive $M(\chi)$ (proved by Blasius in \cite{blasiusannals}), to get
\[ c^{+}(\chi)\sim_{\Q(\chi)}D_{K}^{1/2}\mathbf{p}(\check{\chi};\Phi_{\eta}).\]
Here $\check{\chi}=\chi^{\iota,-1}$ (not to be confused with the dual $\chi^{\vee}=\chi^{-1}$). In fact, Blasius's constructions should provide the following more precise statement, which we will assume: for every $\sigma\in J_{K}$,
\begin{equation}\label{formulaBlasius} c^{+}_{\sigma}(\chi)\sim_{\Q(\chi)\otimes K,\sigma}\mathbf{p}(\check{\chi};\tau),\end{equation}
where $\tau\in\Phi_{\chi}$.

Let $M$ and $\chi$ be as in Theorem~\ref{thmfact}. For each $\rho\in J_{\Q(\chi)}$ and $\varphi\in J_{E}$, let $c^{\pm}(M\otimes RM(\chi))_{\varphi,\rho}\in\C$ be the $\rho\otimes\varphi$-component of $c^{\pm}(\Res_{K/\Q}(M\otimes RM(\chi)))\in E\otimes\Q(\chi)\otimes\C$. We let $s_{\sigma,\varphi,\rho}=d-r_{\sigma,\varphi,\rho}$ and $\mathbf{s}_{\sigma}$ as in Theorem~\ref{thmfact}.  Define $Q^{\mathbf{s}}(M)\in E\otimes\Q(\chi)\otimes\C$ by
\[ {Q}^{\mathbf{s}}(M)=\prod_{\sigma\in J_{K}}\prod_{j=1}^{\mathbf{s}_{\sigma}}Q_{j,\sigma},\]
and let ${Q}^{\mathbf{s}}_{\varphi,\rho}(M)\in E\otimes\C$ be its $\varphi\otimes\rho$-component. Suppose that $m\in\Z$ is a critical integer for $M\otimes RM(\chi)$. For simplicity of notation, we let $r_{\sigma}=r_{\sigma,\varphi,\rho}$ and $s_{\sigma}=s_{\sigma,\varphi,\rho}$ in the following expressions. By~(\ref{form:yoshida1}),~(\ref{form:yoshida2}),~(\ref{formulaQsigmachic2}),~(\ref{formulaasigmachisigno}), Remark~\ref{expresiondeC-}, Proposition~\ref{propa+a-},~(\ref{formulaBlasius}) and Theorem~\ref{thmfact}, we can write
\[
c^{+}(M\otimes RM(\chi)(m))_{\varphi,\rho}\sim_{\varphi(E)\rho(\Q(\chi))K'}\]
\[ (2\pi i)^{emd-w(\chi)\sum_{\sigma}s_{\sigma}}\delta(M)_{\varphi}\left(\prod_{\sigma\in J_{K}}\delta_{\sigma}[\chi_{0}\varepsilon_{L}]_{\rho}^{s_{\sigma}}\left(p((\chi^{\rho}){\check{}};\rho\tau)\right)^{r_{\sigma}-s_{\sigma}}\right){Q}^{\mathbf{s}}_{\varphi,\rho}(M),
\]
where $K'$ is the Galois closure of $K$ in $\overline{\Q}$, we choose the embeddings $\tau$ in $\Phi_{\chi}$, and we write $\rho\tau$ for $\tilde\rho\tau$, with $\tilde\rho$ an extension of $\rho$. If we moreover assume that $\chi_{0}$ is trivial, then by Lemma~\ref{epsilonenL} we can write
\begin{equation}\label{formulacritica}
c^{+}(M\otimes RM(\chi)(m))_{\varphi,\rho}\sim_{\varphi(E)\rho(\Q(\chi))L'}\end{equation}
\[ (2\pi i)^{emd-w(\chi)\sum_{\sigma}s_{\sigma}}\delta(M)_{\varphi}\left(\prod_{\sigma\in J_{K}}\left(p((\chi^{\rho}){\check{}};\rho\tau)\right)^{r_{\sigma}-s_{\sigma}}\right)Q^{\mathbf{s}}_{\varphi,\rho}(M).
\]

\section{Hodge-de Rham structures for unitary groups and automorphic periods}\label{sec:automorphicmotives}
In this section we introduce the Hodge-de Rham structures attached to automorphic representations of unitary groups, which come from the cohomology of automorphic vector bundles and local systems on the corresponding Shimura varieties. On the first subsections we introduce the varieties in question, and setup the notation that we will use throughout the rest of the paper regarding weights and automorphic vector bundles. Many of the things that we say here are valid for more general Shimura varieties, but we restrict ourselves to the unitary group case to keep a reasonable length, and because it's ultimately the case that we will use in the rest of the paper. In the later subsections, we define automorphic quadratic periods for cohomological automorphic representations. For generalities regarding Shimura varieties, automorphic vector bundles, and conjugation, our main references are \cite{dmos} and \cite{milne}. 

\subsection{The groups}\label{sectiongroups} 
In this paper we will work with the following unitary groups. Let $K$ be a totally real field of degree $e=[K:\Q]$ and let $L/K$ be a totally imaginary quadratic extension. Let $V$ be a finite-dimensional $L$-vector space and $h:V\times V\to L$ a non-degenerate hermitian form relative to the non-trivial automorphism $\iota\in\Gal(L/K)$. Let $n=\dim_{L}V$. Let $GU_{*}$ denote the similitude unitary group of the pair $(V,h)$ over $K$. Thus, for a $K$-algebra $R$,
\[ GU_{*}(R)=\{g\in\Aut_{F\otimes_{K}R}(V\otimes_{K}R)\mid h_{R}(gu,gv)=\nu(g)h_{R}(u,v)\quad\forall u,v\in V\otimes_{k}R\},\]
where $\nu(g)\in R^{\times}$ and $h_{R}:V\otimes_{K}R\times V\otimes_{K}R\to F\otimes_{K}R$ is defined by $h_{R}(u\otimes a,v\otimes b)=h(u,v)\otimes ab$. The map $g\mapsto\nu(g)$ defines a morphism $\nu:GU_{*}\to\Gm{K}$, and its kernel is the unitary group $U_{*}$ over $K$. The center of $GU_{*}$ is $\Res_{L/K}\Gm{L}$. We let $GU_{0}=\Res_{K/\Q}GU_{*}$ and $U=\Res_{K/\Q}U_{*}$. Finally, we let $GU=GU(V)$ be the subgroup of $GU_{0}$ consisting of automorphisms $g\in GU_{0}(R)$ for which $\nu(g)\in R^{\times}\subset(K\otimes R)^{\times}$. All of these groups are reductive algebraic groups. Throughout the rest of this subsection, we let $G=GU$. The center $Z$ of $G$, which is connected, is the subtorus of $T^{L}=\Res_{L/\Q}\Gm{L}$ given by
\[ Z(R)=\{x\in(L\otimes_{\Q}R)^{\times}\mid N_{L\otimes_{\Q}R/K\otimes_{\Q}R}(x)\in R^{\times}\}.\]
Note that this is actually the center of the (abstract) group $G(R)$.

We will fix an $L$-basis $\beta=\{v_{1},\dots,v_{n}\}$ of $V$, orthogonal for $h$. For each $\tau\in J_{L}$, let $V_{\tau}=V\otimes_{L,\tau}\C$, and let $h_{\tau}:V_{\tau}\times V_{\tau}\to\C$ be the non-degenerate hermitian form (relative to $\C/\R$) defined by $h_{\tau}(u\otimes z,v\otimes w)=\tau(h(u,v))z\overline w$. We let $(r_{\tau},s_{\tau})$ denote the signature of $(V_{\tau},h_{\tau})$. Choose once and for all a CM type $\Phi$ for the extension $L/K$. Then there are isomorphisms $L\otimes_{\Q}\R\cong\prod_{\tau\in\Phi}\C$ and $V\otimes_{\Q}\R\cong\prod_{\tau\in\Phi}V_{\tau}$ which induce isomorphisms
\begin{equation}\label{isoGRm}
G_{\R}\cong\left(\prod_{\tau\in\Phi}GU(r_{\tau},s_{\tau})\right)',\quad G_{\C}\cong\prod_{\tau\in\Phi}\GL_{n,\C}\times\Gm{\C},\end{equation}
where the symbol $'$ means that we are looking at tuples where all the elements have the same multiplier $\nu$. The second isomorphism is defined over the Galois closure $L'$ of $L$ in $\overline\Q\subset\C$. Let $T\subset G$ be the subgroup of elements of $G$ which, considered as $L$-linear automorphisms of $V$, are diagonal with respect to $\beta$. Then $T$ is a maximal torus of $G$, and it maps to the subgroup of diagonal matrices in~(\ref{isoGRm}). Let $B\subset G_{\C}$ (or $G_{L'}$) be the Borel subgroup of $G_{\C}$ containing $T_{\C}$ that maps to $\prod_{\tau\in\Phi}B_{n,\C}\times\Gm{\C}$ under the second isomorphism of (\ref{isoGRm}), where $B_{n,\C}\subset\GL_{n,\C}$ is the group of upper triangular matrices.

\begin{rem}\label{accionenT} The group $\Aut(\C)$ acts on the complex points $G(\C)$ by functoriality. Let $g=((X_{\tau})_{\tau\in\Phi},\nu)\in G(\C)$. We can explicitly describe the action of an element $\gamma\in\Aut(\C)$, but we will only need the formula when $\gamma=c$. If $g\in G(\C)$, then $c(g)=((\overline\nu I_{r_{\tau},s_{\tau}}X_{\tau}^{*,-1}I_{r_{\tau},s_{\tau}})_{\tau\in\Phi},\overline{\nu})$.
\end{rem}

\subsection{The Shimura varieties}
Fix $(V,h,\Phi)$ as in the previous subsection. Given an orthogonal basis $\beta$ and the corresponding isomorphism~(\ref{isoGRm}), we define $x:\mathbb{S}\to G_{\R}$ as $x=(x_{\tau})_{\tau\in\Phi}$, where
\begin{equation}\label{defx} x_{\tau}(z)=\left(\begin{array}{cc}zI_{r_{\tau}} & 0 \\ 0 & \overline zI_{s_{\tau}}\end{array}\right)\end{equation}
for an $\R$-algebra $R$ and $z\in\mathbb{S}(R)$. We denote by $X$ the $G(\R)$-conjugacy class of $x$. The pair $(G,X)$ satisfies Deligne's axioms (2.1.1.1-3, \cite{deligneSh}) for a Shimura datum, unless $(V,h)$ is totally definite (i.e., unless $r_{\tau}s_{\tau}=0$ for all $\tau\in J_{L}$; in this case, we can still attach to $G$ a Shimura variety of dimension $0$, but we will omit this case from our discussion and assume henceforward that $(V,h)$ is not totally definite). The point $x$ of~(\ref{defx}) factors through $T_{\R}$, where $T$ is the maximal torus constructed in the previous subsection from the same orthogonal basis, and hence $(T,x)$ is a CM pair. The reflex field $E=E(G,X)$, the field of definition of the $G(\C)$-conjugacy class of the Hodge cocharacter $\mu_{x}$, is the field generated over $\Q$ by the set $\{\sum_{\tau\in\Phi}\tau(b)r_{\tau}+\overline\tau(b)s_{\tau}: b\in L\}$. In particular, it is contained in $L'$ and hence it's CM or totally real. The Shimura varieties over $E$ defined by the datum $(G,X)$ will be denoted by $S_{E}$. For any $F\supset E$, we let $S_{F}=S_{E}\times_{E}F$. For any compact open subgroup $U\subset G(\A_{f})$, which will always be assumed to be sufficiently small, we denote by $S_{U,E}$ the Shimura varieties with level $U$.

Let $\gamma\in\Aut(\C)$. We denote by $({}^{\gamma,x}G,{}^{\gamma,x}X)$ the conjugate Shimura datum, and by ${}^{\gamma,x}S$ the corresponding Shimura variety, so that $\gamma S_{E}\cong{}^{\gamma,x}S_{\gamma(E)}$ (Theorem II.4.2, \cite{milne}). In the particular case of $\gamma=c$ complex conjugation, we can naturally identify the pair $({}^{c,x}G,{}^{c,x}X)$ with $(G,\overline{X})$, where $\overline{X}=\{h\circ c:h\in X\}$ (see \cite{blasiusguerb}). We denote by $\overline S_{E}$ the Shimura variety attached to $(G,\overline X)$.

\begin{rem} We can also identify the pair $(G,\overline X)$ with the Shimura datum defined by the hermitian space $(V,-h)$ and the CM type $\Phi$, or by the hermitian space $(V,h)$ and the CM type $\overline\Phi=\{\overline\tau:\tau\in\Phi\}$.
\end{rem}

\subsection{Roots, weights and representations}\label{roots} 
Let $(G,X)$ be the pair attached to $(V,h,\Phi)$ as in the last subsections, and let $(T,x)$ be the CM pair defined in~(\ref{defx}). Let $K_{x}$ denote the centralizer of $x$ in $G_{\R}$, i.e., the scheme-theoretic centralizer of the (scheme-theoretic) image of $x$ in $G_{\R}$. Then $K_{x}\cong\left(\prod_{\tau\in\Phi}Z_{\tau}K_{\tau}\right)'$, where $Z_{\tau}=\mathbb{S}$ is the center of $GU(r_{\tau},s_{\tau})$ and $K_{\tau}=U(r_{\tau})\times U(s_{\tau})$ is embedded diagonally. Also, $K_{x,\C}\cong\left(\prod_{\tau\in\Phi}\GL_{r_{\tau},\C}\times\GL_{s_{\tau},\C}\right)\times\Gm{\C}$. Note that $T_{\R}\subset K_{x}$.

Let $R$ denote the set of roots of the pair $(G_{\C},T_{\C})$, and write $R=R_{c}\coprod R_{n}$, where $R_{c}$ denotes the roots which are also roots of $(K_{x,\C},T_{\C})$ (these are called compact roots). We let $\Lambda=X^*(T)$. Using (\ref{isoGRm}), we make once and for all the identification $\Lambda\cong\left(\prod_{\tau\in\Phi}\Z^{n}\right)\times\Z$, and we write elements $\mu\in\Lambda$ as tuples of integers $\mu=\left((a_{\tau,1},\dots,a_{\tau,n})_{\tau\in\Phi};a_{0}\right)$. Thus, $\mu$ corresponds to the character of $T_{\C}$ given by $((\diag(t_{\tau,1},\dots,t_{\tau,n}))_{\tau\in\Phi};t_{0})\mapsto t_{0}^{a_0}\prod_{\tau\in\Phi}\prod_{i=1}^nt_{\tau,i}^{a_{\tau,i}}$. For each $a\in\Z$, we let $\mu_{a}=\left((0,\dots,0)_{\tau\in\Phi};a\right)\in\Lambda$. For $1\leq i\leq n$ and $\tau\in\Phi$, we let $e_{i,\tau}$ denote the element of $\Lambda$ which hast the $i$-th standard vector of $\Z^{n}$ in coordinate $\tau$, and $0$ everywhere else, including the multiplier coordinate. We can write the set of roots $R$ as
$ R\cong\coprod_{\tau\in\Phi}R_{\tau}$, where $R_{\tau}=\{e_{i,\tau}-e_{j,\tau}|1\leq i\neq j\leq n\}$. We can write the set of compact roots $R_{c}$ as $R_{c}\cong\coprod_{\tau\in\Phi}R_{c,\tau}$, where $R_{c,\tau}=\{e_{i,\tau}-e_{j,\tau}|1\leq i\neq j\leq r_{\tau}\text{ or }r_{\tau}+1\leq i\neq j\leq n\}$.

Let $\mathfrak{g}=\Lie(G)$ and $\mathfrak{k}_x=\Lie(K_x)$. The map $\Ad\circ x:\mathbb{S}\to\GL_{\mathfrak{g}_{\R}}$ induces a Hodge decomposition on $\mathfrak{g}_{\R}$, which has Hodge type $\{(0,0),(-1,1),(1,-1)\}$, and $\mathfrak{g}_{\R}^{0,0}=\mathfrak{k}_{x,\C}$. Let $\mathfrak{p}_x^{\pm}=\mathfrak{g}_{\R}^{\mp1,\pm1}$. Identify $X^*(\mathbb{S}_{\C})\cong\Z\oplus\Z$. Note that if $\alpha\in R$, then $\alpha x_{\C}\in X^*(\mathbb{S}_{\C})$ equals either $(0,0)$, $(-1,1)$ or $(1,-1)$. We denote by $R_x^{0,0}$, $R_x^{-1,1}$ and $R_x^{1,-1}$ the corresponding subsets, so that $R=R_x^{0,0}\coprod R_x^{-1,1}\coprod R_x^{1,-1}$. Note that $R_x^{0,0}=R_{c}$, and if $\alpha=e_{i,\tau}-e_{j,\tau}\in R$, then $\alpha\in R_x^{1,-1}$ (resp. $R_{x}^{-1,1}$) if and only if $1\leq i\leq r_{\tau}$ and $r_{\tau}+1\leq j\leq n$ (resp. $r_{\tau}+1\leq i\leq n$ and $1\leq j\leq r_{\tau}$). Let $R_{c,\tau,x}^{+}\subset R_{c,\tau}$ consist of the elements $e_{i,\tau}-e_{j,\tau}$ for which $i<j$, and let $R_{c,x}^{+}=\coprod_{\tau\in\Phi}R_{c,\tau,x}^{+}$. This is a set of positive roots in $R_{c}$. Let $B_{c,x}$ denote the corresponding Borel subgroup of $K_{x,\C}$ containing $T_{\C}$; then $B_{c,x}\cong\prod_{\tau\in\Phi}\left(B_{r_{\tau},\C}\times B_{s_{\tau},\C}\right)\times\Gm{\C}$. Let $R_x^+=R_{c,x}^+\coprod R_x^{1,-1}$. This is a set of positive roots in $R$. Concretely, $R_{x}^{+}=\coprod_{\tau\in\Phi}R_{\tau,x}^{+}$ where $R_{\tau,x}^{+}$ consists of the elements $e_{i,\tau}-e_{j,\tau}$ in $R_{\tau}$ such that $i<j$. Let $B_x\subset G_{\C}$ be the corresponding Borel subgroup containing $T_{\C}$; then $B_{x}\cong\prod_{\tau\in\Phi}B_{n,\C}\times\Gm{\C}$, i.e., it is the group that we called $B$ before. We use the notation $B_{x}$ because we will soon work with $B_{\overline x}$ which is a different group. Note that $B_{c,x}=B_x\cap K_{x,\C}$. Let $P_{x}\subset G_{\C}$ be the subgroup fixing the filtration on the category of representations of $G_{\C}$ defined by the cocharacter $\mu_x:\Gm{\C}\to G_{\C}$ (see Proposition I.1.7 of \cite{milne}). This is a parabolic subgroup with contains $K_{x,\C}$ as a Levi component. The Lie algebra of $P_x$ is $\mathfrak{P}_x=\mathfrak{k}_{x,\C}\oplus\mathfrak{p}_x^{-}$, and $\Lie(R_uP_x)=\mathfrak{p}_x^-$. Note that we can identify $\mathfrak{p}_{x}^+$ (resp. $\mathfrak{p}_x^-$) with the holomorphic (resp. antiholomorphic) tangent space of $X$ at the point $x$, and $P_{x}$ with the subgroup of tuples in $\prod_{\tau}\GL_{n,\C}\times\Gm{\C}$ for which each $\tau$-component is bock lower triangular with respect to the partition $n=r_{\tau}+s_{\tau}$. The groups $P_{x}$ and $K_{x,\C}$ are defined over $L'\subset\C$, and the maximal torus $T_{L'}$ of $K_{x,L'}$ is split.

We denote by $\Lambda_x^+$ and $\Lambda_{c,x}^+$ the set of dominant weights for $R_x^+$ and $R_{c,x}^+$ respectively. If $\mu\in\Lambda_x^+$ (resp. $\lambda\in\Lambda_{c,x}^+$), we denote by $(W_{\mu},\rho_{\mu})$ (resp. $(V_{\lambda},r_{\lambda})$) the irreducible representation of $G_{\C}$ (resp. $K_{x,\C}$) with highest weight $\mu$ (resp. $\lambda$). We say that a representation $(W,\rho)$ of $G_{\C}$ is defined over a subfield $F\subset\C$ if there exists a representation $(W_{F},\rho_{F})$ of $G_{F}$ such that the extension of scalars $(W_{F},\rho_{F})\otimes_{F}\C$ is isomorphic to $(W,\rho)$ as representations of $G_{\C}$. Note that all the representations of $G_{\C}$ are defined over $L'$. The set $\Lambda_{x}^{+}$ (resp. $\Lambda_{c,x}^{+}$) consists of tuples $\mu=\left((a_{\tau,1},\dots,a_{\tau,n})_{\tau\in\Phi};a_{0}\right)$ as above with $a_{\tau,1}\geq\dots\geq a_{\tau,n}$ (resp. $a_{\tau,1}\geq\dots\geq a_{\tau,r_{\tau}}$ and $a_{\tau,r_{\tau+1}}\geq\dots\geq a_{\tau,n}$) for every $\tau\in\Phi$. For example, the adjoint action of $G_{\C}$ on $\mathfrak{g}_{\C}$ restricts to an action of $K_{x,\C}$ on $\mathfrak{P}_x$ and on $\mathfrak{p}_{x}^\pm$; note that we can identify $\mathfrak{p}_{x}^{+}$ with the dual of $\mathfrak{p}_x^-$ via the Killing form. The space $\Lambda^{d}\mathfrak{p}_{x}^{+}$ has dimension $1$, and its highest weight in $\Lambda_{c,x}^{+}$ is $\left((s_{\tau},\dots,s_{\tau},-r_{\tau},\dots,-r_{\tau})_{\tau\in\Phi};0\right)$, with $s_{\tau}$ appearing $r_{\tau}$ times. As another example, what we called $\mu_{a}=\left((0,\dots,0_{\tau\in\Phi};a\right)$ is the highest weight of the character of $G_{\C}$ given by the $a$-th power $\nu^{a}$ of the multiplier $\nu:G\to\Gm{\Q}$, which is obviously defined over $\Q$. For any $\mu\in\Lambda$, we let $\xi(\mu)=2a_{0}+\sum_{\tau,i}a_{\tau,i}$.

\begin{rem} Let $w_{X}$ denote the weight morphism of $(G,X)$. If $\mu\in\Lambda_{x}^{+}$ then $\rho_{\mu}\circ w_{X,\C}:\Gm{\C}\to\GL_{W_{\mu}}$ takes $t\in\C^{\times}$ to $t^{-\xi(\mu)}\id_{W_{\mu}}$, because the central character of $W_{\mu}$ is the restriction of $\mu$ to $Z_{\C}$. 
\end{rem}

\begin{rem} Our parametrization differs slightly from that of \cite{harriscrelle}. Namely, a highest weight parametrized in the form $\left((a_{\tau,1},\dots,a_{\tau,n})_{\tau\in\Phi};a_{0}\right)\in\Lambda_{x}^{+}$ corresponds to the representation that in the parametrization of \emph{op. cit.} has the same $a_{\tau,i}$'s and $c=\xi(\mu)$.
\end{rem}

We denote by $\W$ (resp. $\W_{c}$) the Weyl group of $(G_{\C},T_{\C})$ (resp. $(K_{x,\C},T_{\C})$), and by $\ell$ the length function on $\W$ with respect to $R_x^+$ (note that if $w\in\W_{c}$, then $\ell(w)$ is also the length with respect to $R_{c,x}^+$ because $w$ preserves $R_{x}^{\pm1,\mp1}$). Let $w_0$ (resp. $w_{0,c}$) be the longest element of $\W$ (resp. $\W_{c}$), so that $\ell(w_0)=|R_x^+|=\frac{n(n-1)e}{2}$ and $\ell(w_{0,c})=|R_{c,x}^+|=|R_{x}^{+}|-d$. For any $w\in\W$, we let $w^{\flat}=w_{0,c}ww_{0}$, and we let $w_{0}^{1}=1^{\flat}=w_{0,c}w_{0}$. Note that $(w^{\flat})^{\flat}=w$ for any $w$. For any integer $n$, we let $\mathfrak{S}_{n}$ denote the symmetric group on $\{1,\dots,n\}$, acting on $(a_{1},\dots,a_{n})\in\Z^{n}$ by $\sigma(a_{1},\dots,a_{n})=(a_{\sigma^{-1}(1)},\dots,a_{\sigma^{-1}(n)})$. We let $u_{n}\in\mathfrak{S}_{n}$ denote the order reversing permutation $i\mapsto n+1-i$. We can identify $\W$ (resp. $\W_{c}$) with $\prod_{\tau\in\Phi}\mathfrak{S}_{n}$ (resp. $\prod_{\tau\in\Phi}\mathfrak{S}_{r_{\tau}}\times\mathfrak{S}_{s_{\tau}}$), with the natural inclusion of the latter into the former. Then we can see $w_{0}$ (resp. $w_{0,c}$) as the tuple $(u_{n})_{\tau\in\Phi}$ (resp. $(u_{r_{\tau}},u_{s_{\tau}})_{\tau\in\Phi}$), and $w_{0}^{1}=(w_{0,\tau}^{1})_{\tau\in\Phi}$, where \[ w_{0,\tau}^{1}(i)=\left\{\begin{array}{ccc} i+r_{\tau} & \text{if } & 1\leq i\leq s_{\tau} \\ i-s_{\tau} & \text{if } & s_{\tau}+1\leq i\leq n.\end{array}\right.\]
If $\mu=\left((a_{\tau,1},\dots,a_{\tau,n})_{\tau\in\Phi};a_{0}\right)\in\Lambda_{x}^{+}$ (resp. $\lambda=\left((a_{\tau,1},\dots,a_{\tau,n})_{\tau\in\Phi};a_{0}\right)\in\Lambda_{c,x}^{+}$), then the dual representation $W_{\mu}^{\vee}$ of $G_{\C}$ (resp. $V_{\lambda}^{\vee}$ of $K_{x,\C}$) has highest weight 
\[ \mu^{\vee}=-w_{0}(\mu)=\left((-a_{\tau,n},\dots,-a_{\tau,1})_{\tau\in\Phi};-a_{0}\right)\]
(resp.
\[  \lambda^{\vee}=-w_{0,c}(\lambda)=\left((-a_{\tau,r_{\tau}},\dots,-a_{\tau,1},-a_{\tau,n},\dots,-a_{\tau,r_{\tau}+1})_{\tau\in\Phi};-a_{0}\right)\text{)}.\] 

We define $\W^1=\{w\in\W:w(R_x^+)\supset R_{c,x}^+\}$; this is a set of coset representatives of shortest length for $\W_{c}\backslash\W$. Concretely, $\W^{1}$ is the set of tuples $(w_{\tau})_{\tau\in\Phi}$ with $w_{\tau}\in\mathfrak{S}_{n}$ a permutation such that $w_{\tau}^{-1}(i)<w_{\tau}^{-1}(j)$ whenever $1\leq i<j\leq r_{\tau}$ or $r_{\tau}+1\leq i<j\leq n$. The element $w_0^{1}$ belongs to $\W^1$, has length $d$, and is the longest element of $\W^1$. More generally, if $w\in\W^{1}$ then $w^{\flat}\in\W^{1}$ and $\ell(w^{\flat})=d-\ell(w)$. If $\mu\in\Lambda_x^+$ and $w\in\W^1$, then $w*\mu=w(\mu+\rho_x)-\rho_x\in\Lambda_{c,x}^+$, where 
\[ \rho_x=\frac{1}{2}\sum_{\alpha\in R_x^+}\alpha=\left(\left(\frac{1}{2}(n-1),\frac{1}{2}(n-3),\dots,\frac{1}{2}(1-n)\right)_{\tau\in\Phi};0\right)\in\Lambda\otimes\R.\]

\begin{rem}
	\label{descrW1n-11} Suppose that $V$ has signature $(n-1,1)$ at some $\tau_{0}\in\Phi$, and signatures $(n,0)$ at all other places. Then we can write $\mathcal{W}^{1}=\{w_{1},\dots,w_{n}\}$, where $w_{i}=(w_{i,\tau})_{\tau\in\Phi}$, $w_{i,\tau}=1$ if $\tau\neq\tau_{0}$, and $w_{i,\tau_{0}}$ is the permutation that sends $i$ to $n$ and is order perserving on the other $n-1$ elements, so $w_{i,\tau_{0}}(x)=x$ for $1\leq x<i$, $w_{i,\tau_{0}}(i)=n$ and $w_{i,\tau_{0}}(x)=x-1$ for $i+1\leq x\leq n$. Note that $\ell(w_{j})=n-j$.
\end{rem}

There is a natural action of $\Aut(\C)$ on $\Lambda$. For $\mu\in\Lambda$, $\gamma\in\Aut(\C)$ and $t\in T(\C)$, $\mu^{\gamma}(t)=\gamma\mu(\gamma^{-1}(t))$. By choosing the Borel $B_x\subset G_{\C}$, there is a second action $\mu\mapsto\gamma(\mu)$ of $\gamma\in\Aut(\C)$ on $\Lambda$, the $*$-action, defined as $\gamma(\mu)(t)=\mu^{\gamma}(n^{-1}tn)$, where conjugation by $n\in G(\C)$ takes the Borel pair $(\gamma(B_{x}),\gamma(T_{\C}))$ to $(B_{x},T_{\C})$. This action preserves $R$, $R_x^+$ and $\Lambda_x^+$, and is trivial on $\Aut(\C/F)$ for any $F\subset\C$ such that $G_F$ is split (in particular, on $\Aut(\C/L')$). Concretely, let $\mu=\left((a_{\tau,1},\dots,a_{\tau,n})_{\tau\in\Phi};a_{0}\right)\in\Lambda$. We can explicitly compute the parameters $\mu^{\gamma}$ and $\gamma(\mu)$. We only write here the action for $\gamma=c$. By Remark~\ref{accionenT},
\[ \mu^{c}=\left((-a_{\tau,1},\dots,-a_{\tau,n})_{\tau\in\Phi};a_0+\sum_{i,\tau}a_{\tau,i}\right)\]
and
\[c(\mu)=\left((-a_{\tau,n},\dots,-a_{\tau,1})_{\tau\in\Phi};a_0+\sum_{i,\tau}a_{\tau,i}\right).\]
If $\mu\in\Lambda_{x}^{+}$, then $c(\mu)=\mu^{\vee}+\mu_{\xi(\mu)}$.

\begin{rem} The equation $c(\mu)=\mu$ means that $a_{\tau,i}=-a_{\tau,n+1-i}$ for every $\tau\in\Phi$ and $i=1,\dots,n$. In this case, $\xi(\mu)=2a_{0}$ and $W_{\mu}\cong W_{\mu}^{\vee}\otimes_{\C}W_{\mu_{2a_{0}}}$.
\end{rem}

\begin{rem}\label{remark tits} Suppose that $W_{\mu}$ is defined over $\Q$. Then $\gamma(\mu)=\mu$ for any $\gamma\in\Aut(\C)$. To show this, we can replace $\C$ by $\overline\Q$. Let $\rho:G\to\GL_{W_{\mu,\Q}}$ denote the descent of the representation to $\Q$. Since it's irreducible, it follows from Th\'eoreme 7.2 of \cite{tits} that $\rho$ is isomorphic to ${}^{\Q}\rho_{\lambda}$ for some $\lambda\in\Lambda_{x}^+$, where the notation is as in \emph{op. cit}. Moreover, in Lemme 7.4 of \emph{op. cit.}, we must have $d=1$ and $r=1$ because $\rho$ is absolutely irreducible, so $\lambda\in(\Lambda_{x}^+)^{\Gal(\overline\Q/\Q)}$. Furthermore, $\rho_{\overline\Q}$ must be the irreducible representation with highest weight $\lambda$, and thus $\mu=\lambda$, which is invariant under $\Gal(\overline\Q/\Q)$. In particular, if $W_{\mu}$ is defined over $\Q$, then $W_{\mu}\cong W_{\mu}^{\vee}\otimes W_{\mu_{2a_{0}}}$ as representations over $\Q$. 

Conversely, if $\gamma(\mu)=\mu$ for every $\gamma\in\Aut(\C)$, then $c(\mu)=\mu$. It's easy to see that $\mu\in Q\oplus\Lambda_{0}$, where $Q$ is the subgroup of $\Lambda$ generated by $R$ and $\Lambda_{0}$ is the subgroup orthogonal to the coroots. It follows from Th\'eor\`eme 3.3 of \cite{tits} that $W_{\mu}$ is defined over $\Q$. 
\end{rem}

Consider now the pair $(G,\overline X)$ and its special point $\overline x$, and consider the same context and notation in this case. Then $K_{\overline x}=K_x$. The Hodge decomposition of $\mathfrak{g}_{\C}$ for $\overline x$ is the complex conjugate of that of $x$, so that $\mathfrak{p}_{\overline x}^{\pm}=\mathfrak{p}_x^{\mp}$, and $R_{\overline x}^{\pm 1,\mp 1}=R_{x}^{\mp 1,\pm 1}$. We \emph{choose} $R_{c,\overline x}^+\subset R_{c}$ to be $R_{c,\overline x}^{+}=-(R_{c,x}^{+})=\coprod_{\tau\in\Phi}R_{c,\tau,\overline x}^{+}$, where $R_{c,\tau,\overline x}^{+}$ consists of the elements $e_{i,\tau}-e_{j,\tau}$ such that $1\leq j<i\leq r_{\tau}$ or $r_{\tau}+1\leq j<i\leq n$. Then let $R_{\overline x}^+=R_{c,\overline x}^+\coprod R_{\overline x}^{1,-1}=-(R_{x}^{+})=\coprod_{\tau\in\Phi}R_{\tau,\overline x}^{+}$, where $R_{\tau,\overline x}^{+}$ consists of the elements $e_{i,\tau}-e_{j,\tau}$ with $i>j$. This is a set of positive roots for $R$. By making reference to $R_{\overline x}^{+}$ or $R_{c,\overline x}^{+}$ instead of $R_{x}^{+}$ or $R_{c,x}^{+}$, we can similarly define, using $\overline{x}$, all the objects defined so far with a subindex $x$. Note that under the natural identification $cG_{\C}\simeq G_{\C}$, the groups $c(B_x)$, $c(B_{c,x})$ and $c(P_x)$ correspond respectively to $B_{\overline x}$, $B_{c,\overline x}$ and $P_{\overline x}$. Concretely, $B_{c,\overline x}\cong\prod_{\tau\in\Phi}(B_{r_{\tau},\C}^{-}\times B_{s_{\tau},\C}^{-})\times\Gm{\C}$, $B_{\overline x}\cong(\prod_{\tau\in\Phi}B_{n,\C}^{-})\times\Gm{\C}$, and $P_{\overline x}$ has block upper triangular matrices. Also, $\Lambda_{\overline x}^+=-(\Lambda_{x}^{+})$ consists of elements $\mu=\left((a_{\tau,1},\dots,a_{\tau,n})_{\tau\in\Phi};\mu_{0}\right)$ such that $a_{\tau,1}\leq\dots \leq a_{\tau,n}$ for every $\tau\in\Phi$. Similarly, $\Lambda_{c,\overline x}^+=-(\Lambda_{c,x}^{+})$ consists of those $\mu$ satisfying $a_{\tau,1}\leq\dots\leq a_{\tau,r_{\tau}}$ and $a_{\tau,r_{\tau}+1}\leq\dots\leq a_{\tau,n}$. Finally, $\rho_{\overline x}=\rho_x^{c}=-\rho_{x}$. This gives an action of $\W$ on $\Lambda\otimes\R$ taking $\Lambda_{\overline x}^+$ to $\Lambda_{c,\overline x}^+$; if needed, we will distinguish it from the other one by writing $w\overline{*}\mu=w(\mu+\rho_{\overline x})-\rho_{\overline x}$. This is equal to $-(w*(-\mu))$. If $c(\mu)=\mu$, then $(w*\mu)^{c}=w\overline{*}\mu^{c}$. 

Let $(W,\rho)$ be a representation of $G_{\C}$. We define the conjugate representation $(W^{c},\rho^{c})$ of $G_{\C}$ by taking $W^{c}=W$ as a real vector space, with complex conjugate $\C$-action, and by taking $\rho^{c}(g)=\rho(\overline g)$ for $g\in G(\C)$. If $(W,\mu)$ is irreducible and $\mu\in\Lambda_{x}^{+}$ is its $R_{x}^{+}$-highest weight, then $(W^{c},\rho^{c})$ is irreducible, and its $R_{\overline x}^{+}$-highest weight is $\mu^{c}$. Equivalently, $(W^{c},\rho^{c})$ has $R_{x}^{+}$-highest weight $w_{0}(\mu^{c})=-(\mu^{c})^{\vee}=(-\mu^{c})^{\vee}$. If $W_{\mu}$ is defined over $\Q$, then $c(\mu)=\mu$ and this implies $(-\mu^{c})^{\vee}=\mu$, so there exists an isomorphism, unique up to scalars, between $W^{c}$ and $W$. If $\lambda\in\Lambda_{c,x}^{+}$, we define in a similar way the conjugate representation $(V_{\lambda}^{c},r_{\lambda}^{c})$, which has $R_{c,\overline x}^{+}$-highest weight $\lambda^{c}$. We can do a similar construction starting with a representation $(W,\rho)$ of $P_{x}$ and obtain a representation $(W^{c},\rho^{c})$ of $P_{\overline x}$.

For future reference, we introduce the following operation on compact weights. For $\lambda\in\Lambda_{c,x}^{+}$, define $\lambda^{\flat}$ to be the $R_{c,x}^{+}$-highest weight of the representation $V_{\lambda}^{\flat}=V_{\lambda}^{\vee}\otimes_{\C}\Lambda^{d}(\mathfrak{p}_{x}^{+})^{\vee}\otimes V_{\mu_{\xi(\lambda)}}$. Concretely, if $\lambda=\left((a_{\tau,1},\dots,a_{\tau,n})_{\tau\in\Phi};a_{0}\right)$, then
\[ \lambda^{\flat}=\left((-a_{\tau,r_{\tau}}-s_{\tau},\dots,-a_{\tau,1}-s_{\tau},-a_{\tau,n}+r_{\tau},\dots,-a_{\tau,r_{\tau}+1}+r_{\tau})_{\tau\in\Phi};\xi(\lambda)-a_{0}\right). \]
Note that $V_{\lambda}^{\flat}\cong V_{\lambda}^{c}\otimes_{\C}\Lambda^{d}(\mathfrak{p}_{x}^{+})^{\vee}$. It's easily checked that if $c(\mu)=\mu$ and $w\in\mathcal{W}^{1}$, then $(w*\mu)^{\flat}=w^{\flat}*\mu$.

Suppose that $(W,\rho)$ is any representation of $P_{x}$. The map $\rho\circ\mu_{x}:\Gm{\C}\to\GL_{W}$ defines a grading $W=\bigoplus_{p}W^{p}$ with $W^{p}=\{w\in W:\rho(\mu_{x}(z))(w)=z^{-p}w\quad\forall z\in\C^{\times}\}$. For each $p\in\Z$, we define $F^{p}(W)=\bigoplus_{p'\geq p}W^{p'}$. Then $F^{p}(W)$ is $P_{x}$-stable, and it's called the Hodge filtration of the representation $(W,\rho)$, even though it's not necessarily attached to a Hodge structure. Suppose in particular that $(W,\rho)$ is an irreducible representation of $K_{x,\C}$, extended to $P_{x}$ in such a way that the unipotent radical $R_{u}P_{x}$ acts trivially on $W$. Then $W=W^{p}$ for a unique $p\in\Z$. If $\lambda\in\Lambda_{c,x}^{+}$ is the $R_{c,x}^{+}$-dominant weight of $(W,\rho)$, we usually write $(p_{\lambda},q_{\lambda})=(p,-\xi(\lambda)-p)$. Concretely, if $\lambda=\left((a_{\tau,1},\dots,a_{\tau,n})_{\tau\in\Phi};a_{0}\right)$, then
\[ \begin{array}{c} p_\lambda=-a_0-\sum_{\tau\in\Phi}\sum_{i=1}^{r_{\tau}}a_{\tau,i}, \\
 q_\lambda=-a_0-\sum_{\tau\in\Phi}\sum_{i=r_{\tau}+1}^na_{\tau,i}.\end{array}
\]

\begin{rem} If $\lambda\in\Lambda_{c,\overline x}^{+}$, we can also write $(p_{\lambda},q_{\lambda})$ for the pair associated to the filtration obtained from $\overline x$ instead of $x$. With this convention, if $\lambda\in\Lambda_{c,x}^{+}$, then $p_{\lambda^{c}}=p_{\lambda}$.
\end{rem}

\subsection{Automorphic vector bundles and their conjugates}\label{avbconjugates} We refer to \cite{milne} for the definition and main properties of automorphic vector bundles, and just set up the relevant notation here. Automorphic vector bundles on $S_{\C}$ or $S_{U,\C}$ are attached to homogeneous bundles on the compact dual Hermitian symmetric space of $(G,X)$. Once we fix $x\in X$, these are given by representations of $P_{x}$. An automorphic vector bundle associated with a representation $(W,\rho)$ of $P_{x}$ which extends to $G_{\C}$ is called \emph{flat}. It is endowed with a natural regular flat connection. The associated local system on $\C$-vector spaces is denoted by $\tilde W$, and it has a structure of a local system on $F$-vector spaces for any $F\subset\C$ over which $(W,\rho)$ is defined as a representation of $G$ (see Proposition II.3.3 of \cite{milne}). On the other hand, automorphic vector bundles attached to representations of $P_{x}$ which are trivial on the unipotent radical of $P_{x}$ (that is, representations of $K_{x,\C}$) are called \emph{fully decomposed}. Automorphic vector bundles have canonical models over finite extensions of $E$. For instance, a flat automorphic vector bundle has a canonical model over $EF$, where $F$ is as above. Any fully decomposed automorphic vector bundle $\mathcal{E}$ has a canonical model over $L'$. This follows from the fact that $K_{x,L'}$ is a split group. We stress that $L'$ is not necessarily optimal. Usually, the letter $\mathcal{E}$ will denote an automorphic vector bundle over $S_{\C}$, and $\mathcal{E}_{U}$ will denote its descent to $S_{U,\C}$. If it has a canonical model over $F\supset E$, we write $\mathcal{E}_{U,F}$ for the corresponding bundle over $F$. The Hodge filtration on representations of $P_{x}$ induces a filtration on the corresponding automorphic vector bundles, which is independent of $x$. 
 
Let $\gamma\in\Aut(\C)$, and fix a special point $x\in X$. We can then conjugate automorphic vector bundles, in the sense that if $\mathcal{E}$ is an automorphic vector bundle over $S_{\C}$, then we can construct another automorphic vector bundle ${}^{\gamma,x}\mathcal{E}$ over ${}^{\gamma,x}S_{\C}$ (see \cite{milne}, III.5). If $\mathcal{E}$ is associated with the representation $(W,\rho)$ of $P_x$, then ${}^{\gamma,x}\mathcal{E}$ is associated with a representation $(W^{\gamma,x},\rho^{\gamma,x})$ of $P_{{}^\gamma x}$, described in the following way. There is a $\gamma$-semilinear isomorphism $W\to W^{\gamma,x}$ that takes the $\rho$-action of $p\in P_x(\C)$ to the $\rho^{\gamma,x}$-action of ${}^{\gamma,x}\gamma(p)\in P_{{}^\gamma x}(\C)$. In particular, if $\gamma=c$ and we forget about $x$ in the notation, then we obtain the representation $(W^{c},\rho^{c})$ that we defined in Subsection~\ref{roots}. 

If $\lambda\in\Lambda_{c,x}^{+}$, then we denote by $\mathcal{E}_{\lambda}$ the automorphic vector bundle over $S_{\C}$ obtained from $V_{\lambda}$. If $\lambda\in\Lambda_{c,\overline x}^{+}$, then $\mathcal{E}_{\lambda}$ will be the automorphic vector bundle over $\overline S_{\C}$ obtained from $V_{\lambda}$. With these conventions, we can identify ${}^{c}\mathcal{E}_{\lambda}$ with $\mathcal{E}_{\lambda^{c}}$ (see Subsection~\ref{roots}).

Fix $U\subset G(\A_{f})$. There is a family of compactifications $S_{U,\C,\Sigma}$ of $S_{U,\C}$, depending on auxiliary data $\Sigma$, which are smooth, projective, and with the property that $Z_{U,\Sigma}=S_{U,\C,\Sigma}-S_{U,\C}$ is a divisor with normal crossings. For the main properties of these compactifications and their conjugates, see \cite{harrisdeltabar} and \cite{bhr}. They have canonical models $S_{U,E,\Sigma}$ over $E$. If $\mathcal{E}$ is an automorphic vector bundle over $S_{\C}$, there is a \emph{canonical} extension $\mathcal{E}_{U,\Sigma}^{\can}$ of $\mathcal{E}_U$ to a vector bundle over $S_{U,\C,\Sigma}$. Let $\mathcal{I}_{U,\Sigma}$ denote the sheaf of ideals defining $Z_{U,\Sigma}$, and define the subcanonical extension as $\mathcal{E}_{U,\Sigma}^{\sub}=\mathcal{E}_{U,\Sigma}^{\can}\otimes_{\Or_{S_{U,\C,\Sigma}}}\mathcal{I}_{U,\Sigma}$. If $\mathcal{E}$ has a canonical model over $F\supset E$, then the canonical and subcanonical extensions $\mathcal{E}_{U,\Sigma}^{?}$ also have canonical models $\mathcal{E}_{U,F,\Sigma}^{?}$ over $F$. Conjugation of automorphic vector bundles extends to these extensions (Proposition 1.4.3 of \cite{bhr}). Given $U$ and two toroidal data $\Sigma$ and $\Sigma'$, there are canonical isomorphisms $H^{i}(S_{U,F,\Sigma},\mathcal{E}^{?}_{U,F,\Sigma})\simeq H^{i}(S_{U,F,\Sigma'},\mathcal{E}^{?}_{U,F,\Sigma'})$, where $?=\can$ or $\sub$ (Proposition 2.4 of \cite{harrisdeltabar}). We define $H^{i}(S_{U,F},\mathcal{E}^{?})=\varinjlim_{\Sigma}H^{i}(S_{U,F,\Sigma},\mathcal{E}^{?}_{U,F,\Sigma})$. Let $H^{i}_!(S_{U,F},\mathcal{E})$ be the image of $H^{i}(S_{U,F},\mathcal{E}^{\sub})$ in $H^{i}(S_{U,F},\mathcal{E}^{\can})$ under the natural map induced from $\mathcal{E}_F^{\sub}\to\mathcal{E}_F^{\can}$. This space is called the \emph{interior cohomology} of $S_{U,F}$ with coefficients in $\mathcal{E}$. All these cohomologies spaces are finite-dimensional over $F$. We also define
\[ H^i(S_F,\mathcal{E}^?)=\varinjlim_{U}H^i(S_{U,F},\mathcal{E}^{?}) \]
and $H^i_!(S_F,\mathcal{E})$ to be the image of $H^i(S_F,\mathcal{E}^{\sub})$ in $H^i(S_F,\mathcal{E}^{\can})$. There is an action $G(\A_f)$ on $H^i_!(S_F,\mathcal{E})$, which makes it into a smooth admissible representation of $G(\A_f)$ (\cite{harrisdeltabar}, Proposition 2.6).

For any automorphic vector bundle $\mathcal{E}$ over $S_{U,\C}$, let $\mathcal{E}'=\Omega^d_{S_{U,\C}}\otimes\mathcal{E}^\vee$ be its Serre dual. Then there is a non-degenerate pairing (Serre duality)
\[ H^{d-i}_!(S_{U,\C},\mathcal{E}')\otimes_{\C}H^i_!(S_{U,\C},\mathcal{E})\to\C,\]
ratioanl over any $F$ over which $\mathcal{E}$ has a canonical model (Corollary 3.8.5, \cite{harrisdeltabar}). 

We can combine the isomorphism $cS_{U,F}\to\overline S_{U,c(F)}$ and its extension to toroidal compactifications with base change for cohomology, and we get a $c$-semilinear, $G(\A_{f})$-equivariant isomorphism 
\[ c_{\coh}:H^i_!(S_F,\mathcal{E})\to H^i_!(\overline S_{c(F)},{}^c\mathcal{E}).\]

\subsection{The Hodge-de Rham structures $M(W)$}\label{therealizations} Let $(W,\rho)$ be an absolutely irreducible representation of $G$, with $R_{x}^{+}$-highest weight $\mu=\left((a_{\tau,1},\dots,a_{\tau,n})_{\tau\in\Phi};a_{0}\right)$. Then $\rho\circ w_{X,\Q}(t)=t^{-\xi}$ for each $t\in\Q^\times$, where $\xi=\xi(\mu)=2a_{0}$. In Theorems 2.2.7 and 2.3.1 of \cite{harrismotives}, a construction is given of a Hodge-de Rham structure (rather, a $G(\A_{f})$-admissible Hodge-de Rham structure, with the obvious definition) $M(W)^{i}$ associated to $W$, using the $i$-th degree cohomology of the local system $\tilde W$ and its conjugates for the Betti realizations and the hypercohomology of a certain complex of automorphic vector bundles, called the Faltings (dual) BGG complex, for the de Rham realization. We refer to \cite{harrismotives} for the details of the construction.

\begin{rem} We are assuming that $(W,\rho)$ is a representation of $G$ as an algebraic group over $\Q$. This assumption is made for simplicity of notation. Otherwise, the field of definition of $W_{\mu}$ must be incorporated into the field of coefficients of the Hodge-de Rham structures. Also, following Subsection 2.2 of \cite{harriscrelle}, one can start with any irreducible representation of $G_{\C}$, which will be defined over some number field, and take the sum of its Galois conjugates to obtain a representation over $\Q$. To simplify notation, we will assume from now on that $(W,\rho)$ is defined over $\Q$. We can identify which representations are defined over $\Q$ using Remark \ref{remark tits}. 
\end{rem}

\begin{rem} The additional data of the $\ell$-adic sheaves associated to $W$ (Proposition II.3.3, \cite{milne}) makes $M(W)^{i}$ a realization (or a $G(\A_{f})$-admissible realization). We will not use them in this part of the paper. Ideally, $M(W)^{i}$ is a motive for absolute Hodge cycles, but this is not relevant for our calculations. 
\end{rem}

In any case, $M(W)^{i}$ is a pure Hodge-de Rham structure of weight $i-\xi$, over $L'$ (or over a field over which the Faltings complex has a canonical model; this complex is formed of direct sums of fully decomposable bundles, and so it always has a canonical model over $L'$), and it has coefficients in $\Q$. We write $M(W)=M(W)^{d}$. Letting $1:L'\hookrightarrow\C$ denote the given inclusion, the Hodge component of $M(W)^{i}_{1}$ of type $(p,i-\xi-p)$ is given by
\begin{equation}\label{hodgedecomp}
 (M(W)_{1}^{i})^{p,i-\xi-p}=\bigoplus_{\substack{w\in\W^1\\p_{w*\mu}=p}}H^{i-\ell(w)}_!(S_{\C},\mathcal{E}_{w*\mu}). \end{equation}
The summands with $\ell(w)=i$ form what is called the \emph{holomorphic part} of $M^i(W)_{1}\otimes\C$. Those with $\ell(w)=0$ form the \emph{anti-holomorphic part}. We usually refer to the summands in~(\ref{hodgedecomp}) as Weyl components. Any $w\in\W^{1}$ defines a Weyl component, contributing to the Hodge component of type $(p,q)$ with $p=p_{w*\mu}$. Note that $q=i+q_{w*\mu}\neq q_{w*\mu}$. On the case of interest in this paper, $i$ will be equal to $d$, in which case $q=d+q_{w*\mu}$ can also be written as $q=p_{w^\flat*\mu}$.

\subsection{Automorphic forms as cohomology classes} 
From now on, we will write $K_{x}$ for $K_{x}(\R)$ when it appears as an argument for relative Lie algebra cohomology. Thus, for example, we write $H^{q}(\mathfrak{P}_{x},K_{x};V)$ for a $(\mathfrak{P}_{x},K_{x}(\R))$-module $V$ when we mean $H^{q}(\mathfrak{P}_{x},K_{x}(\R);V)$. 

Let $(W,\rho)$ (resp. $(V,r)$) be a representation of $G$ over $\Q$ (resp. of $K_{x,\C}$ over $\C$). Let $\mathcal{A}$ (resp. $\mathcal{A}_0$, resp. $\mathcal{A}_{(2)}$) be the space of automorphic forms (resp. cuspidal, resp. square-integrable automorphic forms) on $G(\Q)\backslash G(\A)$ in the sense of \cite{bj}, with respect to some maximal compact subgroup $K_{\infty}\subset G(\R)$. There are canonical inclusions (\cite{schwermer} 4.2)
\begin{equation}\label{incCohomFlat} H^i(\mathfrak{g}_{\C},K_x;\mathcal{A}_0\otimes_{\C}W_{\C})\subset M(W)^{i}_{1}\otimes\C\subset H^i(\mathfrak{g}_{\C},K_x;\mathcal{A}_{(2)}\otimes_{\C}W_{\C}).\end{equation}
Let $c_B$ denote complex conjugation on the second factor of $H^{i}_{!}(W)\otimes\C$. There is a similar picture for the fully decomposed automorphic vector bundle $\mathcal{E}$ attached to $(V,r)$. This case is contained in \cite{harrisdeltabar}. There are canonical inclusions
\begin{equation}\label{incCohomFully}
 H^i(\mathfrak{P}_x,K_x;\mathcal{A}_0\otimes_{\C}V)\subset H^i_!(S_{\C},\mathcal{E})\subset H^i(\mathfrak{P}_x,K_x;\mathcal{A}_{(2)}\otimes_{\C}V).
\end{equation}
In some cases, these inclusions are isomorphisms. See for instance \cite{harrisdeltabar}, 5.3.2. The first inclusion in~(\ref{incCohomFully}) is an isomorphism whenever $i=0$ or $i=d$ (\cite{harrisdeltabar} 5.4.2). Moreover, for $i=0$ we can also write
	\begin{equation}\label{h0can} H^{0}(\mathfrak{P}_{x},K_{x};\mathcal{A}(G)\otimes_{\C}V)=H^{0}(S_{\C},\mathcal{E}^{\can}).\end{equation}

\subsection{Complex conjugation} Keep the assumptions of Subsection \ref{therealizations}, and assume moreover that $i=d$. We are interested in studying how $c_{B}$ acts on a coherent cohomology class represented in terms of automorphic forms. It's illustrating to start with a general $\lambda\in\Lambda_{c,x}^{+}$, so we look at the spaces $H^{q}_{!}(S_{\C},\mathcal{E}_{\lambda})$. By~(\ref{incCohomFully}), we can view it sitting between $H^{q}(\mathfrak{P}_{x},K_{x};\mathcal{A}_{0}\otimes_{\C}V_{\lambda})$ and $H^{q}(\mathfrak{P}_{x},K_{x};\mathcal{A}_{(2)}\otimes_{\C}V_{\lambda})$. We can decompose $\mathcal{A}_{0}$ and $\mathcal{A}_{(2)}$ into a countable direct sum of irreducible unitary $(\mathfrak{g}_{\C},K_{x})$-modules $\pi$. By Proposition 4.5 of \cite{harrisdeltabar} (see also \emph{op. cit.}, Lemma 5.2.3), each such $\pi$ contributes to cohomology if and only if $\chi_{\pi}(C)=\langle\lambda+\rho_x,\lambda+\rho_x\rangle-\langle\rho_x,\rho_x\rangle$, where in our case $\langle,\rangle$ is the standard inner product on vectors, $C$ is the Casimir element of $\mathfrak{g}_{\C}$ and $\chi_{\pi}$ is the infinitesimal character of $\pi$. In this case, all the corresponding cochains will be closed. If we let $\mathcal{A}_{?,\lambda}$ denote the sum of the $\pi$ that contribute, we obtain a natural identification (where $?=0$ or $(2)$)
\[ H^{q}(\mathfrak{P}_{x},K_{x};\mathcal{A}_{?}\otimes_{\C}V_{\lambda})=(\mathcal{A}_{?,\lambda}\otimes_{\C}\Lambda^{q}(\mathfrak{p}_{x}^{+})\otimes_{\C}V_{\lambda})^{K_{x}(\R)}.\]
Now, let $c_{\mathcal{A}}$ denote complex conjugation of functions on $\mathcal{A}_{?}$. It's easy to see that $c_{\mathcal{A}}$ preserves $\mathcal{A}_{?,\lambda}$ in a $K_{x}(\C)$-equivariant way. Similarly, we can conjugate elements of $\Lambda^{q}(\mathfrak{p}_{x}^{+})\otimes_{\C}V_{\lambda}$ and land in $\Lambda^{q}(\mathfrak{p}_{x}^{-})\otimes_{\C}V_{\lambda}^{c}$ in a $K_{x}(\R)$-equivariant way. Moreover, we can identify $V_{\lambda}^{c}$ with $V_{\lambda^{\flat}}\otimes_{\C}\Lambda^{d}(\mathfrak{p}_{x}^{+})$, and using the Killing form, this becomes $V_{\lambda^{\flat}}\otimes_{\C}\Lambda^{d}(\mathfrak{p}_{x}^{-})^{\vee}$. Thus, there is a $K_{x}(\R)$-equivariant, $c$-semilinear isomorphism from $\mathcal{A}_{?,\lambda}\otimes_{\C}\Lambda^{q}(\mathfrak{p}_{x}^{+})\otimes_{\C}V_{\lambda}$ to $\mathcal{A}_{?,\lambda}\otimes_{\C}\Lambda^{q}(\mathfrak{p}_{x}^{-})\otimes_{\C}\Lambda^{d}(\mathfrak{p}_{x}^{-})^{\vee}\otimes_{\C}V_{\lambda^{\flat}}$. We can see the latter space, after contracting, as $\Hom_{\C}(\Lambda^{d-q}(\mathfrak{p}_{x}^{-}),\mathcal{A}_{?,\lambda}\otimes_{\C}V_{\lambda^{\flat}})$. 

\begin{lemma}\label{lemmalambda} Suppose that $\sum_{\tau,i}\lambda_{\tau,i}=0$. Then $\mathcal{A}_{?,\lambda}=\mathcal{A}_{?,\lambda^{\flat}}$.
\begin{proof} An elementary computation shows that $\langle\lambda+\rho_x,\lambda+\rho_x\rangle=\langle\lambda^{\flat}+\rho_x,\lambda^{\flat}+\rho_x\rangle$ if the multiplier factors of $\lambda$ and $\lambda^{\flat}$ coincide, which is precisely the condition that $\sum_{\tau,i}\lambda_{\tau,i}=0$.
\end{proof}
\end{lemma}

\begin{rem}\label{sumaaj=0} For any $w\in\W^1$, $\lambda=w*\mu$ satisfies the condition of the lemma because $W$ is defined over $\Q$.
\end{rem}

Thus, under the conditions of the lemma, we get, for any $q=0,\dots,d$, $c$-semilinear isomorphisms $H^{q}(\mathfrak{P}_{x},K_{x};\mathcal{A}_{?}\otimes_{\C}V_{\lambda})\to H^{d-q}(\mathfrak{P}_{x},K_{x};\mathcal{A}_{?}\otimes V_{\lambda^{\flat}})$. The element $x(i)\in K_{x}(\R)$ acts on $\mathcal{A}_{?}$ by right translation, preserving $\mathcal{A}_{?,\lambda}$, and on $V_{\lambda}$ by $r_{\lambda}(x(i))$. Composing the previous map with $x(i)$ gives $c$-semilinear isomorphisms
\begin{equation}\label{primerisocB} 
c_{\aut}:H^{q}(\mathfrak{P}_{x},K_{x};\mathcal{A}_{?}\otimes_{\C}V_{\lambda})\to H^{d-q}(\mathfrak{P}_{x},K_{x};\mathcal{A}_{?}\otimes V_{\lambda^{\flat}}).
\end{equation}
The following proposition is proved in \cite{harriscrelle} (Corollary 2.5.9), and allows to express the conjugation $c_{B}$ in terms of automorphic forms. The statement in \emph{op. cit.} only covers the case $w=w_{0}^{1}$, or cohomology in degrees $0$ and $d$, but the proof extends to any $q=0,\dots,d$ (see \cite{harriscrelle}, Remark 2.5.13, 2.).

\begin{prop}\label{coro259} Let $W$ be a representation of $G$, absolutely irreducible with highest weight $\mu\in\Lambda_{x}^{+}$. Then $c_{B}:H^{d}_{!}(W)\otimes\C\to H^{d}_{!}(W)\otimes\C$ takes the Weyl component corresponding to $w$ to the Weyl component corresponding to $w^{\flat}$, and the following diagram commutes up to $F^\times$-multiples, where $F\supset E$ is a field over which $\mathcal{E}_{w*\mu}$ and $\mathcal{E}_{w^{\flat}*\mu}$ have a canonical model:

\begin{tikzcd}
H^{d-\ell(w)}(\mathfrak{P}_{x},K_{x};\mathcal{A}_{0}\otimes_{\C}V_{w*\mu})\arrow{r}{c_{\aut}}\arrow{d}{} & H^{d-\ell(w^{\flat})}(\mathfrak{P}_{x},K_{x};\mathcal{A}_{0}\otimes_{\C}V_{w^{\flat}*\mu})\arrow{d}{} \\
H^{d-\ell(w)}_{!}(S_{\C},\mathcal{E}_{w*\mu})\arrow{r}{c_{B}}\arrow{d}{} & H^{d-\ell(w^{\flat})}_{!}(S_{\C},\mathcal{E}_{w^{\flat}*\mu})\arrow{d}{} \\
H^{d-\ell(w)}(\mathfrak{P}_{x},K_{x};\mathcal{A}_{(2)}\otimes_{\C}V_{w*\mu})\arrow{r}{c_{\aut}} & H^{d-\ell(w^{\flat})}(\mathfrak{P}_{x},K_{x};\mathcal{A}_{(2)}\otimes_{\C}V_{w^{\flat}*\mu}).
\end{tikzcd}

\noindent In this diagram, the vertical arrows are the inclusions from~(\ref{incCohomFully}).

\end{prop}

We introduce now another operation $\Fr$ in terms of automorphic forms, as follows. As in~(\ref{incCohomFully}), the space $H^{d-\ell(w)}_!(\overline S_{\C},\mathcal{E}_{(w*\mu)^c})$ sits between \[H^{d-\ell(w)}(\mathfrak{P}_{\overline x},K_x;\mathcal{A}_{0}\otimes_{\C}V_{(w*\mu)^{c}})=(\mathcal{A}_{0,(w*\mu)^{c}}\otimes_{\C}\Lambda^{d-\ell(w)}(\mathfrak{p}_{x}^{-})\otimes_{\C}V_{(w*\mu)^{c}})^{K_{x}(\R)}\] and \[H^{d-\ell(w)}(\mathfrak{P}_{\overline x},K_x;\mathcal{A}_{(2)}\otimes_{\C}V_{(w*\mu)^{c}})=(\mathcal{A}_{(2),(w*\mu)^{c}}\otimes_{\C}\Lambda^{d-\ell(w)}(\mathfrak{p}_{x}^{-})\otimes_{\C}V_{(w*\mu)^{c}})^{K_{x}(\R)}.\] An easy computation shows that $\mathcal{A}_{?,(w*\mu)^{c}}=\mathcal{A}_{?,(w*\mu)}$, which by Lemma~\ref{lemmalambda} equals $\mathcal{A}_{?,(w*\mu)^{\flat}}$. On the other hand, there is a natural identification of $K_{x}(\R)$-modules between $V_{(w*\mu)^{c}}$ and $\Lambda^{d}(\mathfrak{p}_{x}^{+})\otimes_{\C}V_{(w*\mu)^{\flat}}$ (see Subsection~\ref{roots}). Finally, contracting via the Killing form, we obtain $\C$-linear isomorphisms
\[ H^{d-\ell(w)}(\mathfrak{P}_{\overline x},K_{x};\mathcal{A}_{?}\otimes_{\C}V_{(w*\mu)^{c}})\to H^{d-\ell(w^{\flat})}(\mathfrak{P}_{x},K_{x};\mathcal{A}_{?}\otimes_{\C}V_{(w*\mu)^{\flat}}), \]
and an isomorphism in coherent cohomology sitting in between them, which we denote by  
\begin{equation}\label{deffrobenius} \Fr:H^{d-\ell(w)}_{!}(\overline{S}_{\C},\mathcal{E}_{(w*\mu)^{c}})\to H^{d-\ell(w^{\flat})}_{!}(S_{\C},\mathcal{E}_{w^{\flat}*\mu}).\end{equation}
Recall that we also have a $c$-semilinear isomorphism 
\[c_{\coh}=H^{q}_{!}(S_{\C},\mathcal{E}_{w*\mu})\to H^{q}_{!}(\overline S_{\C},\mathcal{E}_{(w*\mu)^{c}})\] 
In the next Proposition, $c_{B}$ really denotes the inverse of the map in Proposition~\ref{coro259}, although in terms of Betti conjugation it's obviously equal to its own inverse, so we use the same notation. Alternatively, we use Proposition~\ref{coro259} with $w$ and $w^{\flat}$ interchanged.

\begin{prop}\label{Frisrational} Under the assumptions of Proposition~\ref{coro259}, the following $\C$-linear automorphism
\[ c_{B}\circ\Fr\circ c_{\coh}:H^{d-\ell(w)}_{!}(S_{\C},\mathcal{E}_{w*\mu})\to H^{d-\ell(w)}_{!}(S_{\C},\mathcal{E}_{w*\mu}) \]
is rational over $F$, where $F\supset E$ is any field over which $\mathcal{E}_{w*\mu}$ has a canonical model.
\begin{proof} This is proved exactly as in \cite{harriscrelle}, Lemma 2.5.12.
\end{proof}
\end{prop}

\subsection{Cohomological automorphic representations}\label{cohomologicalreps}
Let $\Aut_G$ denote the set of automorphic representations of $G$, $\Aut_G^0$ the set of cuspidal automorphic representations, and $\Temp_G\subset\Aut_G^0$ the set of essentially tempered automorphic representations (an automorphic representation $\pi$ is called essentially tempered if its twist by a character of $G(\Q)\backslash G(\A)$ is tempered at all places). From now on we will only work with essentially tempered representations, although many of the things we say are true in more general cases. We say that $\pi=\pi_{\infty}\otimes\pi_f\in\Temp_G$ is \emph{cohomological of type} $\mu\in R_{x}^+$ if $H^*(\mathfrak{g}_{\C},K_x;\pi_{\infty}\otimes_{\C}W_{\mu})\neq 0$. We assume almost all the time that $W_{\mu}$ is defined over $\Q$, so as to place ourselves in the setting of the last subsections. We write $W=W_{\mu}$ if $\mu$ is understood. The condition of $\pi\in\Temp_{G}$ being cohomological is equivalent to the statement that $\pi_{\infty}$ is a discrete series representation that belongs to the $L$-packet whose infinitesimal character is that of $W^{\vee}$. Moreover, it follows from Theorem II.5.4 of \cite{borelwallach} (see also VI, 1.4, (1) therein) that $\pi$ actually contributes to middle-dimensional cohomology only, that is, $H^{*}(\mathfrak{g}_{\C},K_{x};\pi_{\infty}\otimes_{\C}W_{\C})=H^{d}(\mathfrak{g}_{\C},K_{x};\pi_{\infty}\otimes_{\C}W_{\C})\neq 0$ (we thank the referee for the precise reference). We let $\Coh_{G,\mu}$ the set of $\pi\in\Temp_{G}$ which are cohomological of type $\mu$. 

Assume from now on that $\pi\in\Coh_{G,\mu}$. As in \cite{harriscrelle}, 2.7, we can define the standard $L$-function of $\pi$ as $L^{S}(s,\pi,\St)=L^{S}(s,BC(\pi_{0}),\St)$. Here $\St$ refers to the standard representation of the $L$-group of $\GL_{n}$ over $L$, $\pi_{0}$ is an irreducible constituent of the restriction of $\pi$ to $U(\A)$, and $BC(\pi_{0})$ is the base change of $\pi_{0}$ to an irreducible admissible representation of $\GL_{n}(\A_{L}^{S})$, for a big enough finite set of places $S$ of $L$. The base change is defined locally at archimedean places, at split places, and at places of $K$ where the local unitary group $U_{v}$ and $\pi_{0,v}$ are unramified. Under our assumptions, it is known that $BC(\pi_{0})$ is the restriction to $\GL_{n}(\A_{L}^{S})$ of an automorphic representation $\Pi$ of $\GL_{n}(\A_{L})$, so we can actually define $L(s,\pi,\St)$ at all places as $L(s,\Pi,\St)$. We define the motivic normalization by
\[ L^{\mot,S}(s,\pi,\St)=L^{S}\left(s-\frac{n-1}{2},\pi,\St\right).\]

The representation $\pi_{f}$ is defined over a number field $E(\pi)$. We denote the corresponding model by $\pi_{f,0}$ (note that $E(\pi)$ may not always be taken to be the fixed field under the stabilizer of $\pi_{f}$). We can take $E(\pi)$ to be a CM field (see \cite{bhr}, Theorem 4.4.1, and \cite{harriscrelle}, 2.6). To make things simpler, we take $E(\pi)$ to contain $L'$; this is still a CM field. We let $J_{\pi}=J_{E(\pi)}$, and for each $\sigma\in J_{\pi}$, we let $\pi_{f}^{\sigma}=\pi_{f,0}\otimes_{E(\pi),\sigma}\C$. Throughout, we make the assumption that $\pi_{f}^{\sigma}$ is essentially tempered for any $\sigma$. For $?=\tau:L'\hookrightarrow\C$ or $?=\dR$, define
\[ M(\pi,W)_{?}=\Hom_{\Q[G(\A_f)]}(\pi_{f,0},M(W)_{?}).\]
These data define a Hodge-de Rham structure $M(\pi,W)$ over $L'$ with coefficients in $E(\pi)$, which is pure of weight $d-\xi$. We let $M^{+}(\pi,W)=R_{L'/(L')^+}M(\pi,W)$, a Hodge-de Rham structure over $(L')^+$, the maximal totally real subfield of $L'$, with coefficients in $E(\pi)$. Just as in the case of $M(W)$, we can endow $M(\pi,W)$ and $M^{+}(\pi,W)$ with the additional structure of a realization using $\ell$-adic cohomology. It is well known that $M(\pi,W)$ is non-trivial if $\pi\in\Coh_{G,\mu}$. When $W$ is understood from the context, we will denote these realizations by $M(\pi)$ and $M^{+}(\pi)$ respectively.

\subsection{Polarizations}\label{ssec:polariz} Recall that there is an isomorphism $W\cong W^{\vee}\otimes W_{\mu_{\xi}}$ as representations of $G$ (see Subsection~\ref{roots}), with $\xi=2a_{0}$. As in \cite{harriscrelle}, 2.6.8, this induces a perfect $G(\A_{f})$-equivariant pairing of Hodge-de Rham structures
\begin{equation}\label{motpairing} \langle,\rangle:M(W)\otimes M(W)\to\Q(\xi-d). \end{equation}
\begin{rem}
	The $G(\A_{f})$-action on $\Q(\pi)(\xi-d)$ is given by the character $\|\nu\|^{-\xi}:G(\A)\to\Q^{\times}$, where $\|\cdot\|$ is the adelic norm on $\A^{\times}$.
\end{rem}

Let $w\in\mathcal{W}^{1}$. It's easy to see that the Serre dual $(\mathcal{E}_{w*\mu})'$ is isomorphic to $\mathcal{E}_{w^{\flat}*\mu^{\vee}}$, and we can write Serre duality
\begin{equation}
\label{serreduality}
H^{2d-i-\ell(w^{\flat})}_{!}(S_{\C},(\mathcal{E}_{w*\mu})')\times H^{i-\ell(w)}_{!}(S_{\C},\mathcal{E}_{w*\mu})\to\C
\end{equation}
(note that $2d-i-\ell(w^{\flat})=d-(i-\ell(w))$). Consider the de Rham pairing $\langle,\rangle_{\dR}:M(W)_{\dR}\otimes_{L'}M(W)_{\dR}\to L'$. For $\tau\in J_{L'}$, there is a complex pairing $\langle,\rangle_{\dR,\tau}:M(W)_{\dR,\tau}\otimes_{\C}M(W)_{\dR,\tau}\to\C$, where $M(W)_{\dR,\tau}=M(W)_{\dR}\otimes_{L',\tau}\C$. We will only be explicit for $\tau=1$, and usually ignore the subscript $1$. Then, in terms of Weyl components, the pairing $\langle,\rangle_{\dR,1}$ of~(\ref{motpairing}) is, up to the $\nu^{\xi}$-twist, Serre duality~(\ref{serreduality}) with $i=d$. Thus, we can write 
\begin{equation}\label{pairingWeyl}  \langle,\rangle_{\dR}:H^{d-\ell(w^{\flat})}_{!}(S_{\C},\mathcal{E}_{w^{\flat}*\mu})\otimes_{\C}H^{d-\ell(w)}_{!}(S_{\C},\mathcal{E}_{w*\mu})\to\C. \end{equation} This pairing is actually $L'$-rational, in the sense that it descends to a pairing
\[ \langle,\rangle_{\dR}:H^{d-\ell(w^{\flat})}_{!}(S_{L'},\mathcal{E}_{w^{\flat}*\mu})\otimes_{L'}H^{d-\ell(w)}_{!}(S_{L'},\mathcal{E}_{w*\mu})\to L'.\]
In particular, taking $w=1$, we obtain a pairing
\[ \langle,\rangle_{\dR}:H^{0}_{!}(S_{L'},\mathcal{E}_{w_{0}^{1}*\mu})\otimes_{L'}H^{d}_{!}(S_{L'},\mathcal{E}_{\mu})\to L'. \]

We can write these pairings in terms of vector-valued functions. Namely, the pairings~(\ref{pairingWeyl}) sit (with respect to the inclusions~(\ref{incCohomFully})) between pairings
\begin{equation}
\label{pairingRelLie}
\langle,\rangle_{\dR}:H^{d-\ell(w^{\flat})}(\mathfrak{P}_{x},K_{x};\mathcal{A}_{?}\otimes_{\C}V_{w^{\flat}*\mu})\times H^{d-\ell(w)}(\mathfrak{P}_{x},K_{x};\mathcal{A}_{?}\otimes_{\C}V_{w*\mu})\to\C 
\end{equation}
for $?=0$ or $(2)$. Now, as $K_{x,\C}$-representations, $V_{w^{\flat}*\mu}\cong V_{w*\mu}^{\vee}\otimes_{\C}\Lambda^{d}(\mathfrak{p}_{x}^{+})\otimes_{\C}V_{\mu_{\xi}}$. If $f$ (resp. $f'$) is an element in the first (resp. second) factor of~(\ref{pairingRelLie}), then we can see $f$ (resp. $f'$) as an element in $(\Lambda^{d-\ell(w^{\flat})}(\mathfrak{p}_{x}^{+})\otimes_{\C}\mathcal{A}_{?,w^{\flat}*\mu}\otimes_{\C}V_{w^{\flat}*\mu})^{K_{x}(\R)}$ (resp. in $(\Lambda^{d-\ell(w)}(\mathfrak{p}_{x}^{+})\otimes_{\C}\mathcal{A}_{?,w*\mu}\otimes_{\C}V_{w*\mu})^{K_{x}(\R)}$). After choosing an $L'$-rational isomorphism $\Lambda^{2d}(\mathfrak{p}_{x})\cong L'$, contraction of coefficients and multiplication of automorphic forms, an operation we denote by $g\mapsto[f(g),f'(g)]$, defines an element in $\mathcal{A}_{?}\otimes_{\C}V_{\mu_{\xi}}$. Keeping in mind our choices for Haar measures, we can write~(\ref{pairingRelLie}) as
\[ \langle f,f'\rangle_{\dR}=\int_{G(\Q)Z_{G}(\A)\backslash G(\A)}[f(g),f'(g)]\|\nu(g)\|^{\xi}dg.\]

\begin{rem}\label{scalarforms} We can also write these pairings in terms of scalar-valued automorphic forms, as in \cite{harriscrelle}, 2.6.11.  To be more precise, let $f,f'$ be as above. We can now see them as elements $f\in\Hom_{K_{x}(\R)}(V_{w^{\flat}*\mu}^{\vee}\otimes_{\C}\Lambda^{d-\ell(w^{\flat})}(\mathfrak{p}_{x}^{-}),C^{\infty}(G(\Q)\backslash G(\A)))$ and $f'\in\Hom_{K_{x}(\R)}(V_{w*\mu}^{\vee}\otimes_{\C}\Lambda^{d-\ell(w)}(\mathfrak{p}_{x}^{-}),C^{\infty}(G(\Q)\backslash G(\A)))$. We assume for simplicity that $V_{w*\mu}\otimes\Lambda^{d-\ell(w)}(\mathfrak{p}_{x}^{+})$ is irreducible. This is true, for example, if $w=1$ or $w=w_{0}^{1}$, which are the cases that will concern us in this paper. The general case follows by decomposing this space into irreducible components. It's easy to see that
\[ V_{w^{\flat}*\mu}^{\vee}\otimes_{\C}\Lambda^{d-\ell(w^{\flat})}(\mathfrak{p}_{x}^{-})\simeq V_{w*\mu}\otimes_{\C}\Lambda^{d-\ell(w)}(\mathfrak{p}_{x}^{+})\otimes_{\C}V_{\mu_{\xi}}.  \]
Let ${L}\subset V_{w*\mu}^{\vee}\otimes_{\C}\Lambda^{d-\ell(w)}(\mathfrak{p}_{x}^{-})$ (resp. ${L}'\subset V_{w^{\flat}*\mu}^{\vee}\otimes_{\C}\Lambda^{d-\ell(w^{\flat})}(\mathfrak{p}_{x}^{-})$) be the lowest (resp. highest) weight spaces. Then $f$ and $f'$ define smooth functions on $G(\A)$ by restricting to ${L}$ and ${L}'$ and evaluating at a basis element. Then (see \cite{harriscrelle}, Proposition 2.6.12)
\begin{equation}
\label{scalarpairing}
\langle f,f'\rangle_{\dR}=\int_{G(\Q)Z_{G}(\A)\backslash G(\A)}f(g)f'(g)\|\nu(g)\|^{\xi}dg.
\end{equation}
\end{rem}

Assume from now on that $\pi$ satisfies $\pi^{\vee}\cong\pi\otimes\|\nu\|^{\xi}$. Then $\langle,\rangle$ defines polarizations with the same notation
\[ \langle,\rangle:M(\pi,W)\otimes_{E(\pi)}M(\pi,W)\to E(\pi)(\xi-d)\]
and
\[ \langle,\rangle:M^{+}(\pi,W)\otimes_{E(\pi)}M^{+}(\pi,W)\to E(\pi)(\xi-d).\]
over $L'$ and $L'^{+}$ respectively. 

\begin{rem}\label{antiholomorphic} Suppose that $\beta:\pi_{f,0,\C}\to H^{0}_{!}(S_{\C},\mathcal{E}_{w_{0}^{1}*\mu})$ is a $\Q[G(\A_{f})]$-equivariant map. We can see the target as $\Hom_{K_{x}(\R)}(V_{w_{0}^{1}*\mu}^{\vee},\mathcal{A}_{0,w_{0}^{1}*\mu})$. For any $v\in\pi_{f,0}$, the image via $\beta(v)$ of a lowest weight vector of $V_{w_{0}^{1}*\mu}^{\vee}$ in $\mathcal{A}_{0}$ corresponds to a holomorphic form $f$ in $\pi$. Similarly, if $\beta:\pi_{f,0,\C}\to H^{d}_{!}(S_{\C},\mathcal{E}_{\mu})$, then $\beta$ gives rise to an antiholomorphic form in $\pi$.
\end{rem}

\subsection{Automorphic quadratic periods}\label{autquad}
Let $w\in\mathcal{W}^{1}$, and let $\beta$ be an element of $\Hom_{\Q[G(\A_{f})]}(\pi_{f,0},H^{d-\ell(w)}_{!}(S_{L'},\mathcal{E}_{w*\mu}))$. We define $R\beta$ to be the pair $(\beta,c_{\coh}(\beta))$. Define $\Fr R\beta=(\Fr c_{\coh}(\beta),\Fr\beta)$, where $\Fr$ is the map defined in (\ref{deffrobenius}). We define the automorphic quadratic period $Q(\pi;\beta)$ by
\[ Q(\pi;\beta)=\langle R\beta,\Fr R\beta\rangle\in E(\pi)\otimes\C, \]
where the pairing $\langle,\rangle$ is defined in Subsection \ref{ssec:polariz}. As in \cite{harriscrelle}, (2.8.7.4), we see that 
\[ Q(\pi;\beta)\sim_{E(\pi)\otimes L'}\langle\beta,\Fr c_{\coh}(\beta)\rangle.\]
Using Proposition \ref{Frisrational}, we can replace this by $\langle\beta,c_{B}(\beta)\rangle$.

\begin{rem}
We can express $Q(\pi;\beta)$ in terms of integrals of the form~(\ref{scalarpairing}). Namely, suppose that $\beta$ gives rise to an automorphic form $f$ on $G(\A)$, as in Remark~\ref{scalarforms}. Using Remark~\ref{sumaaj=0}, it's not hard to see $\lambda(x(i))=\pm1$, where $\lambda$ is the highest weight of $V_{w*\mu}\otimes_{\C}\Lambda^{d-\ell(w)}(\mathfrak{p}_{x}^{+})$. Determining the sign should not be complicated but we don't care about it in this paper. In any case, the map $c_{\aut}$ of~(\ref{primerisocB}) takes $f$ to $\pm\overline{f}$. Moreover, by Proposition \ref{coro259}, this differs from $c_{B}$ by a rational multiple in $L'$, and thus
\begin{equation}\label{Lema288}
Q(\pi;\beta)\sim_{E(\pi)\otimes L'}\int_{G(\Q)Z(\A)\backslash G(\A)}f(g)\overline{f}(g)\|\nu(g)\|^{\xi}dg.
\end{equation}

\end{rem}

\begin{rem}\label{rempipidual} When studying $L$-functions in the next section, we will need to pair forms in $\pi$ with forms in $\pi^{\vee}$, rather than with $\pi\cong\pi^{\vee}\otimes\|\nu\|^{-\xi}$. To remedy this, we introduce a new operation on automorphic forms and cohomology classes. Fix $\lambda\in\Lambda_{c,x}^{+}$, with $\xi=\xi(\lambda)$. Let $\hat\lambda=\lambda-\mu_{\xi}$. Define a map $\mathcal{A}_{?}\otimes_{\C}V_{\lambda}\to\mathcal{A}_{?}\otimes_{\C}V_{\lambda}\otimes_{\C}V_{\mu_{-\xi}}$ by taking $f\otimes v\mapsto\hat f\otimes v\otimes 1$, where $\hat f(g)=(2\pi i)^{\xi}\|\nu(g)\|^{\xi}f(g)$ for $f\in\mathcal{A}$. It's easy to see that this map is $(\mathfrak{P}_{x},K_{x}(\R))$-linear, and hence it thus induces a map in $(\mathfrak{P}_{x},K_{x})$-cohomology. There is a map
\[ H^{q}_{!}(S_{\C},\mathcal{E}_{\lambda})\to H^{q}_{!}(S_{\C},\mathcal{E}_{\hat\lambda})\]
\[\alpha\mapsto\hat\alpha\]
sitting in between as in~(\ref{incCohomFully}), which is rational over $L'$ (see \cite{harriscrelle}, 2.8.9.2). Now let $\beta:\Hom_{\Q[G(\A_{f})]}(\pi_{f,0},H^{d-\ell(w)}_{!}(S_{L'},\mathcal{E}_{w*\mu}))$. By Proposition \ref{coro259}, $c_{B}(\beta)\in\Hom_{\Q[G(\A_{f})]}(\pi_{f,0},H^{d-\ell(w^{\flat})}_{!}(S_{L'},\mathcal{E}_{w^{\flat}*\mu}))$. Note that $w^{\flat}*\mu-\mu_{\xi}=w^{\flat}*\mu^{\vee}$. Define $c_{B}(\beta)^{\vee}$ to be the composition
\[ \pi_{f,0}^{\vee}\to\pi_{f,0}^{\vee}\otimes\|\nu\|^{-\xi}\cong\pi_{f,0}\xrightarrow{c_{B}(\beta)}H^{d-\ell(w^{\flat})}_{!}(S_{\C},\mathcal{E}_{w^{\flat}*\mu})\xrightarrow{\alpha\mapsto\hat\alpha}H^{d-\ell(w^{\flat})}_{!}(S_{\C},\mathcal{E}_{w*\mu^{\vee}}). \]
Then we define $Q(\pi,\pi^{\vee};\beta)=\langle\beta,c_{B}(\beta)^{\vee}\rangle$, where this pairing comes from Serre duality between $H^{d-\ell(w)}_{!}(S_{\C},\mathcal{E}_{w*\mu})$ and $H^{d-\ell(w^{\flat})}_{!}(S_{\C},\mathcal{E}_{w^{\flat}*\mu^{\vee}})$ (see~(\ref{serreduality})). As in \cite{harriscrelle}, 2.8.9.5, it follows from~(\ref{Lema288}) that
\[ Q(\pi,\pi^{\vee};\beta)\sim_{E(\pi)\otimes L'}\int_{G(\Q)Z(\A)\backslash G(\A)}f(g)\hat{\overline{f}}(g)dg,\]
where $f$ lifts $\beta$ as above. Thus 
\begin{equation}
\label{formulaQpipidual}
Q(\pi,\pi^{\vee};\beta)\sim_{E(\pi)\otimes L'}(2\pi i)^{\xi}Q(\pi;\beta).
\end{equation}
\end{rem}

As in \cite{harriscrelle}, 2.9, we can extend the definition of automorphic quadratic periods to the case where we twist by an algebraic Hecke character $\psi$. The main reason that we switched from $\chi$ to $\psi$ in the notation for algebraic Hecke characters is that in Section~\ref{sec:relations} we will combine the results of Sections~\ref{sec:factorization} and \ref{sec:doubling}, but the algebraic Hecke characters will be slightly different (although closely related; see Section~\ref{sec:relations} for details). Recall that we have fixed an orthogonal basis of $V$, giving rise to a particular map $x:\mathbb{S}\to G_{\R}$ (see~(\ref{defx})). This depends on $\Phi$, but it will be fixed throughout. Note that $G\subset\Res_{L/\Q}GL_{V}$, and thus there is a map $\det:G\to T^{L}=\Res_{L/\Q}\Gm{L}$, and we can consider the composition $\det\circ x:\mathbb{S}\to(T^{L})_{\R}$. Using $\Phi$ to identify $(T^{L})_{\R}$ with $\prod_{\tau\in\Phi}\mathbb{S}$, the map $\det\circ x$ sends an element $z$ to the tuple $(z^{r_{\tau}}\overline{z}^{s_{\tau}})_{\tau\in\Phi}$, and thus $\det\circ x=\prod_{\tau\in\Phi}h_{\{\tau\}}^{r_{\tau}}h_{\{\overline\tau\}}^{s_{\tau}}$ with the notation of Subsection~\ref{CMperiods} (using sets $\Psi$ consisting of a single element). Just as in $\S1$ of \cite{harrisunitary}, we can normalize the CM periods $p(\psi;\Psi)$ in such a way that
\[ \left(\frac{p(\psi;\det\circ x)}{\prod_{\tau\in\Phi}p(\psi;\{\tau\})^{r_{\tau}}p(\psi;\{\overline\tau\})^{s_{\tau}}}\right)^{\gamma}=\frac{p(\psi^{\gamma};\det\circ x)}{\prod_{\tau\in\Phi}p(\psi^{\gamma};\{\tau\})^{r_{\tau}}p(\psi^{\gamma};\{\overline{\tau}\})^{s_{\tau}}} \]
for all $\gamma\in\Aut(\C)$. In particular, 
\[
\frac{p(\psi;\det\circ x)}{\prod_{\tau\in\Phi}p(\psi;\{\tau\})^{r_{\tau}}(\psi;\{\overline\tau\})^{s_{\tau}}}\in\Q(\psi).
\]
Moreover, Lemma 1.6 of \emph{op. cit.} implies that $p(\psi;\{\overline{\tau}\})\sim p({\psi^{\iota}};\{\tau\})$, and using 1.4 (c) of \emph{op. cit.}, it follows that 
\begin{equation}\label{factpchi} p(\psi;\det\circ x)\sim_{E(\psi)}\prod_{\tau\in\Phi}p(\psi^{r_{\tau}}({\psi}^{\iota})^{s_{\tau}};\{\tau\}). \end{equation}
Here we are taking $E(\psi)$ to be a sufficiently big field containing $\Q(\psi)$ and the relevant reflex fields (in practice, we will take $E(\psi)=L'\Q(\psi)$).
\begin{rem} After conjugating $\psi$ and $\eta$, as explained in Subsection~\ref{CMperiods}, as well as the Shimura datum $(T^{L},\det\circ x)$), we can consider the tuple $\mathbf{p}(\psi;\det\circ x)=(p(\psi^{\rho};(\det\circ x))^{\rho})_{\rho\in J_{E(\psi)}}$ as an element in $(E(\psi)\otimes\C)^\times$, and~(\ref{factpchi}) remains valid. 
	\end{rem}

Let $S(\det\circ x)$ be the Shimura variety attached to the pair $(T^{L},\det\circ x)$. The map $\det:G\to T^{L}$ sends $X$ to $\{\det\circ x\}$, and thus defines a map $\det:S_{\C}\to S(\det\circ x)_{\C}$, which is rational over $E(T^{L},\det\circ x)\supset E$. Let $W_{\eta}$ be the $1$-dimensional $\Q(\eta)$-vector space on which $\eta$ acts. This defines the local system $\tilde{W_{\eta}}$ on $\Q(\eta)$-vector spaces over $S(\det\circ x)(\C)$ that is used in the construction of the Hodge-de Rham structure $M(W_{\eta})$. We can pull it back to a local system $\det^{*}(\tilde{W_{\eta}})$ on $\Q(\eta)$-vector spaces over $S(\C)$. In terms of parameters, note that $\eta\circ\det:T_{\C}\to\Gm{\C}$, in our parametrization of $X^*(T)$, equals \[\mu(\eta)=\left((n_{\tau}-n_{\overline{\tau}},\dots,n_{\tau}-n_{\overline{\tau}})_{\tau\in\Phi};n\sum_{\tau\in\Phi}n_{\overline{\tau}}\right).\]
Now, let the notation and assumptions be as in Subsection~\ref{cohomologicalreps}. The local system $\tilde W\otimes\det^{*}(\tilde W_{\eta})$ on $\Q(\eta)$-vector spaces is attached to the absolutely irreducible representation of $G_{\Q(\eta)}$ whose highest weight is $\mu+\mu(\eta)$. Suppose that $\psi\in X(\eta^{-1})$, and let $\pi\otimes\psi=\pi\otimes\psi\circ\det$. We can consider the Hodge-de Rham structure $M(\pi\otimes\psi,W_{\mu+\mu(\eta)})$, which is defined over $L'$ and has coefficients in $E(\pi,\psi)=E(\pi)\otimes E(\psi)$. Note that $W_{\mu+\mu(\eta)}$ is only defined over $\Q(\eta)$, and rarely over $\Q$, but we can still construct the Hodge-de Rham structures as in Theorems 2.2.7 and 2.3.1 of \cite{harrismotives}.

Let $\beta$ be an element as in Subsection \ref{autquad} for $M(\pi,W)$. Suppose that 
\[s_{\psi}:E(\psi)\to H^{0}_{!}(S(\det\circ x)_{\C},\mathcal{E}_{\eta})\]
is a nonzero $\Q[T^{L}(\A_{f})]$-linear map. Here $\mathcal{E}_{\eta}$ refers to the automorphic vector bundle attached to the character $\eta$. Assume that $s_{\psi}$ is rational for the de Rham structure, that is, rational with respect to the canonical model of $\mathcal{E}_{\eta}$ over $L'$. Such nonzero ${s}_{\psi}$ is unique up to an $(E(\psi)\otimes L')^{\times}$-multiple (using the given embedding of $L'$ to view $E(\psi)\otimes L'\subset E(\psi)\otimes\C$). Pulling ${s}_{\psi}$ back to $S_{\C}$ via $\det$, and combining it with $\beta$, we form
\begin{equation}\label{defbetachi} \beta(\psi)=\beta\otimes{s}_{\psi}:\pi_{f,0}\otimes E(\psi)\to H^{d-\ell(w)}_{!}(S_{L'},\mathcal{E}_{w*(\mu+\mu(\eta))}).\end{equation}
We can similarly take a nonzero map ${s}_{\psi^{-1}}$ to define $\beta(\psi^{-1})$. As in Subsection \ref{autquad}, we define
\[ Q(\pi;\psi;\beta)=\langle R(\beta(\psi)),\Fr R(\beta(\psi^{-1}))\rangle\in E(\pi,\psi)\otimes\C, \]
where the pairing comes from the polarization on $M(\pi\otimes\psi,W_{\mu+\mu(\eta)})$. The calculations in \cite{harriscrelle}, 2.9, show that 
\begin{equation}\label{formulaQpichibeta}
 Q(\pi;\psi;\beta)\sim_{E(\pi,\psi)\otimes L'}Q(\pi;\beta)\mathbf{p}(\psi;\det\circ x)^{-1}\mathbf{p}(\psi^{-1};\det\circ\overline x)^{-1}.
\end{equation}

\begin{rem}
	As in Remark~\ref{rempipidual}, we modify these quadratic periods, and define $Q(\pi,\pi^{\vee};\psi;\beta)$ analogously.
\end{rem}

\section{Critical values of $L$-functions}\label{sec:doubling}

In this section we prove our main result (Theorem~\ref{maintheorem}) concerning the critical values of $L^{S,\mot}(s,\pi\otimes\psi,\St)$ for a cohomological automorphic representation $\pi$ of a unitary group $G$ and an algebraic Hecke character $\psi$ of $L$. In the first subsection, we introduce the double hermitian space in order to set up the doubling method. In the following subsections, we define a family of differential operators, Eisenstein series and Piatetski-Shapiro-Rallis zeta integrals, invoking Li's formula relating them to the standard $L$-function. In Subsection \ref{ssec:mainthm}, we analyze the rationality properties of these zeta integrals by carefully choosing the sections defining the Eisenstein series, and relate the remaining factors of Li's formula to automorphic quadratic periods and CM periods. Then we put everything together to prove Theorem~\ref{maintheorem}. 

\subsection{The double hermitian space}
Let $L/K$ be a CM extension, and let $(V,h_{V})$ be any hermitian vector space over $L/K$ of dimension $n$, as in Subsection~\ref{sectiongroups}. Define $-V$ to be the space $V$ with hermitian form $-h_{V}$, and let $2V=V\oplus(-V)$, equipped with its natural hermitian form $h_{2V}=h_{V}\oplus h_{-V}$. We fix an orthogonal basis $\beta=\{v_{1},\dots,v_{n}\}$ for $(V,h_{V})$ (and a CM type $\Phi$), after which we can choose $x:\mathbb{S}\to G_{\R}$ and $x^{-}:\mathbb{S}\to G_{\R}$ as before, where $G=GU(V)=GU(-V)$. We let $G^{(2)}=GU(2V)$. This is a unitary group with signatures $(n,n)$ at all places. Choosing the orthogonal basis $\beta^{(2)}=\{(v_{1},0),\dots,(v_{n},0),(0,v_{1}),\dots,(0,v_{n})\}$ of $2V$, we construct $x^{(2)}:\mathbb{S}\to G^{(2)}_{\R}$ in a similar fashion as $x$. Let $G^{\sharp}\subset G\times G$ be subgroup of pairs with the same similitude factor, and let $x^{\sharp}:\mathbb{S}\to G^{\sharp}_{\R}$ be the map $(x,x^{-})$. We denote by $(G,X)$, $(G^{-},X^{-})=(G,X^{-})=(G,\overline{X})$, $(G^{(2)},X^{(2)})$ and $(G^{\sharp},X^{\sharp})$ the corresponding Shimura data, and by $S$, $S^{-}$, $S^{(2)}$ and $S^{\sharp}$ the corresponding Shimura varieties. All the reflex fields are contained in the Galois closure $L'\subset\overline\Q$ of $L$ in $\overline\Q$, and the reflex field of $(G^{(2)},X^{(2)})$ is $\Q$. Notice that there is a natural embedding
\[ i:(G^{\sharp},X^{\sharp})\subset(G^{(2)},X^{(2)}).\] 
We denote by the same letter the embedding of Shimura varieties $i:S^{\sharp}\to S^{(2)}$.
Our choices give the following identifications:
\[
K_{x,\C}\cong\left(\prod_{\tau\in\Phi}\GL_{r_{\tau},\C}\times\GL_{s_{\tau},\C}\right)\times\Gm{\C}, \]
\[ K_{x^{-},\C}\cong\left(\prod_{\tau\in\Phi}\GL_{s_{\tau},\C}\times\GL_{r_{\tau},\C}\right)\times\Gm{\C}, \]
\[ K_{x^{\sharp},\C}\cong\left(\prod_{\tau\in\Phi}\GL_{r_{\tau},\C}\times\GL_{s_{\tau},\C}\times\GL_{s_{\tau},\C}\times\GL_{r_{\tau},\C}\right)\times\Gm{\C},\]
\[K_{x^{(2)},\C}\cong\left(\prod_{\tau\in\Phi}\GL_{n,\C}\times\GL_{n,\C}\right)\times\Gm{\C}.\]
The inclusion $K_{x^{\sharp},\C}\subset K_{x^{(2)},\C}$ coming from the embedding of Shimura data is explicitly given, at each place $\tau$, by taking a tuple of matrices $(A,D,A^{-},D^{-})$ (with $A,D^{-}\in\GL_{r_{\tau}}$ and $D,A^{-}\in\GL_{s_{\tau}}$) to the pair 
\[\left(\left(\begin{array}{cc}A & 0 \\ 0 & A^{-}\end{array}\right),\left(\begin{array}{cc}D & 0 \\ 0 & D^{-}\end{array}\right)\right)\in\GL_{n}\times\GL_{n}.\]

For future reference, we write the highest weight for the action of $K_{x^{\sharp},\C}$ on the $1$-dimensional space $\Lambda^{2d}(\mathfrak{p}_{x^{\sharp}}^{+})$. In the usual parametrization, where we split $2n=r_{\tau}+s_{\tau}+s_{\tau}+r_{\tau}$, it is given by
\[ \left((s_{\tau},\dots,s_{\tau};-r_{\tau},\dots,-r_{\tau};r_{\tau},\dots,r_{\tau};-s_{\tau},\dots,-s_{\tau})_{\tau\in\Phi};0\right).\]

Suppose that $(V,r)$ (resp. $({V}^{-},r^{-})$) is an irreducible representation of $K_{x,\C}$ (resp. $K_{x^{-},\C}$). Thus, ${V}$ is given as a tuple $({V}_{\tau})_{\tau\in\Phi}$, with ${V}_{\tau}$ an irreducible representation of $\GL_{r_{\tau},\C}\times\GL_{s_{\tau},\C}$, and a characer $\chi_{{V}}$ of $\Gm{\C}$. We similarly define ${V}_{\tau}^{-}$ and $\chi_{{V}^{-}}$. We define an irreducible representation $({V},{V}^{-})^{\sharp}$ of $K_{x^{\sharp},\C}$ by taking the family $({V}_{\tau},{V}_{\tau}^{-})_{\tau\in\Phi}$ and the character $\chi_{{V}}\chi_{{V}^{-}}$. If ${V}$ (resp. ${V}^{-}$) has highest weight $\lambda=\left((a_{\tau,1},\dots,a_{\tau,n})_{\tau\in\Phi};a_{0}\right)\in\Lambda_{c,x}^{+}$ (resp. $\lambda^{-}=\left((a_{\tau,1}^{-},\dots,a_{\tau,n}^{-})_{\tau\in\Phi};a_{0}^{-}\right)\in\Lambda_{c,x^{-}}^{+})$, then, in the usual parametrization, the highest weight of $({V},{V}^{-})^{\sharp}$ is
\[ (\lambda,\lambda^{-})^{\sharp}=\left((a_{\tau,1},\dots,a_{\tau,n},a_{\tau,1}^{-},\dots,a_{\tau,n}^{-})_{\tau\in\Phi}               ;a_{0}+a_{0}^{-}\right). \] 
We denote by $\mathcal{E}_{(\lambda,\lambda^{-})^{\sharp}}$ the corresponding automorphic vector bundle over $S^{\sharp}_{\C}$. We can also see it as the pullback via the natural map $S^{\sharp}_{\C}\hookrightarrow S_{\C}\times\overline{S}_{\C}$ of the external tensor product of $\mathcal{E}_{\lambda}$ and $\mathcal{E}_{\lambda^{-}}$. We let $({V}^{*},r^{*})$ be the representation of $K_{x^{-},\C}$ obtained by taking the dual $({V}^{\vee},r^{\vee})$ of $({V},r)$, seen as a representation of $K_{x^{-},\C}$ via swapping the $\GL$-factors at each place in the identifications above. For any $\ell\in\Z$, define $\lambda^{\sharp}(\ell)=(\lambda,\lambda^{*})^{\sharp}\otimes\nu^{\ell}$, so that
\[ \lambda^{\sharp}(\ell)=\left((a_{\tau,1},\dots,a_{\tau,n},-a_{\tau,n},\dots,-a_{\tau,1})_{\tau\in\Phi};\ell\right).\]

\subsection{A family of differential operators}

Fix an integer $m$, and define
\[ \Lambda_{m}=\left((-m,\dots,-m;m,\dots,m)_{\tau\in\Phi};0\right)\in\Lambda_{c,x^{(2)}}^{+}, \] 
where there are $n$ entries of the form $-m$ and $n$ of the form $m$. This is the highest weight of the character $K_{x^{(2)},\C}\to\Gm{\C}$ defined by 
\[ \left((g_{\tau},g_{\tau}')_{\tau\in\Phi};z\right)\to\prod_{\tau\in\Phi}\det(g_{\tau})^{-m}\det(g_{\tau}')^{m}\]
(for $g_{\tau},g_{\tau}'\in\GL_{n}$). We write $\mathbb{V}_{m}$ for $\C$, endowed with this action. In particular, the corresponding automorphic vector bundle over $S^{(2)}_{\C}$ is a line bundle which we denote by $\mathcal{E}_{m}=\mathcal{E}_{\Lambda_{m}}$. It has a canonical model over $L'$. The pullback $i^{*}\mathcal{E}_{m}$ is the automorphic vector bundle over $S^{\sharp}_{\C}$, whose associated $K_{x^{\sharp},\C}$-representation is the $1$-dimensional representation with highest weight
\[ i^{*}\Lambda_{m}=\left((-m,\dots,-m;m,\dots,m;-m,\dots,-m;m,\dots,m)_{\tau\in\Phi};0\right). \] 
Suppose that $({V},r)$ is an irreducible representation of $K_{x,\C}$, and look at the representation ${V}_{i^{*}\Lambda_{-m}}\otimes_{\C}{V}_{\lambda^{\sharp}(\ell)}$, for another integer $\ell\in\Z$. In the terminology of \cite{harrisvb2}, taking into account the different sign conventions, this representation is positive if and only if  $a_{\tau,1}+m\leq 0$ and $-a_{\tau,n}+m\leq 0$ for every $\tau\in\Phi$. The following Proposition follows from 7.11.11 of \emph{op. cit.}.

\begin{prop} Let the notation and assumptions be as above, and let $\eta$ be an algebraic character of $T^{L}=\Res_{L/\Q}\Gm{L}$, identified with the tuple of integers $(m_{\tau})_{\tau\in J_{L}}$. Let $\mu=\left((a_{\tau,1},\dots,a_{\tau,n})_{\tau\in\Phi};a_{0}\right)\in\Lambda_{x}^{+}$, and let $\Lambda=\Lambda(\mu;\eta^{-1})=w_{0}^{1}*(\mu-\mu(\eta))\in\Lambda_{c,x}^{+}$. Then, for each $m$ such that
\begin{equation}\label{ineqfort}
\frac{n}{2}\leq m\leq\min\{-a_{\tau,s_{\tau}+1}+s_{\tau}+m_{\tau}-m_{\overline\tau},a_{\tau,s_{\tau}}+r_{\tau}+m_{\overline\tau}-m_{\tau}\}_{\tau\in\Phi},\end{equation}
there exists a nonzero, $L'$-rational, differential operator
\[ \Delta_{m}(\Lambda;\ell):\mathcal{E}_{m}|_{S^{\sharp}_{\C}}\to\mathcal{E}_{\Lambda^{\sharp}(\ell)}. \] 
\end{prop}

\begin{rem} These differential operators extend to toroidal compactifications and canonical extensions, and the functorial properties of the latter (\cite{harrisdeltabar}) give rise to $L'$-rational maps
\[ \Delta_{m}(\Lambda;\ell):H^{0}(S^{(2)}_{L'},\mathcal{E}_{m}^{\can})\to H^{0}(S^{\sharp}_{L'},\mathcal{E}_{\Lambda^{\sharp}(\ell)}^{\can}). \]
Using~(\ref{h0can}), we get a differential operator 
\[ \Delta_{m}(\Lambda;\ell):H^{0}(\mathfrak{P}_{x^{(2)}},K_{x^{(2)}};\mathcal{A}(G^{(2)})\otimes_{\C}\mathbb{V}_{m})\to H^{0}(\mathfrak{P}_{x^{\sharp}},K_{x^{\sharp}};\mathcal{A}(G^{\sharp})\otimes_{\C}{V}_{\Lambda^{\sharp}(\ell)}).\]
As in Section 7 of \cite{harrisvb2}, $\Delta_{m}(\Lambda;\ell)$ is induced by an element $\partial_{m}(\Lambda;\ell)\in U(\mathfrak{p}_{x^{(2)}}^{+})\otimes_{\C}\Hom(\mathbb{V}_{m},{V}_{\Lambda^{\sharp}(\ell)})$. In this description, the universal enveloping algebra $U(\mathfrak{p}_{x^{(2)}}^{+})$ acts by differentiating an element of $\mathcal{A}(G^{(2)})$ and restricting to $G^{\sharp}(\A)$.
\end{rem}

\subsection{Eisenstein series} 
For this subsection, we follow \cite{cohomologicalII} (see also \cite{harriscrelle}, \cite{harrisunitary} and \cite{siegelweil}). We will use Siegel-type Eisenstein series, and for this it is more convenient to identify the group $G^{(2)}$ with the similitude unitary group over $\Q$ attached to $L/K$ and the skew-hermitian matrix $S_{n}\in\GL_{2n}(\Q)\subset\GL_{2n}(L)$ given by
\[ S_{n}=\left(\begin{array}{cc}0 & I_{n} \\ -I_{n} & 0\end{array}\right).\]
(In the definition of hermitian spaces and their associated groups, we might equally work with skew-hermitian forms instead of hermitian forms). Let $GU(S_{n})$ denote the similitude unitary group attached to the matrix $S_{n}$ over $L/K$. Thus, for a $\Q$-algebra $R$, $GU(S_{n})(R)$ consists of matrices $X\in\GL_{2n}(L\otimes R)$ such that $X^{*}S_{n}X=\nu(X)S_{n}$ for some scalar $\nu(X)\in R^{\times}$. To identify $G^{(2)}$ with $GU(S_{n})$, let $\alpha\in L$ be a nonzero totally imaginary element, that is, such that $\iota(\alpha)=-\alpha$. Fix an $L$-basis $\beta=\{v_{1},\dots,v_{n}\}$ of $V$, and let $Q\in\GL_{n}(L)$ be the matrix of $h_{V}$ in this basis, so that $Q_{ij}=h_{V}(v_{i},v_{j})$. Consider the $L$-basis $\beta^{(2)}=\{(v_{1},0),\dots,(v_{n},0),(0,v_{1}),\dots,(0,v_{n})\}$ of $2V$, and let $Q^{(2)}\in\GL_{2n}(L)$ be the matrix of $h_{2V}$ with respect to $\beta^{(2)}$. Thus,
\[ Q^{(2)}=\left(\begin{array}{cc}Q & 0 \\ 0 & -Q\end{array}\right).\]
Let $\eta\in\GL_{2n}(L)$ be the matrix
\[ \eta=\left(\begin{array}{cc}I_{n} & I_{n} \\ \frac{\alpha}{2}Q & -\frac{\alpha}{2}Q\end{array}\right).\]
Then the map sending $g\in G^{(2)}(R)$ to $\eta[g]_{\beta^{(2)}}\eta^{-1}$ (where $[g]_{\beta^{(2)}}$ is the matrix of $g$ with respect to $\beta^{(2)}$) gives an isomorphism $G^{(2)}\cong GU(S_{n})$. We can also write this isomorphism in terms of another convenient basis of $2V$. Namely, let $\gamma=\{v'_{1},\dots,v'_{2n}\}$, where $v'_{i}=(v_{i},v_{i})$ for $i=1,\dots,n$ and $v'_{i}=(v_{i},-v_{i})$ for $i=n+1,\dots,2n$. Then we can write the isomorphism $g\mapsto\eta[g]_{\beta^{(2)}}\eta^{-1}$ as $g\mapsto\eta'[g]_{\gamma}\eta^{'-1}$, where
\[ \eta'=\left(\begin{array}{cc}I_{n} & 0 \\ 0 & \frac{\alpha}{2}I_{n}\end{array}\right).\]

We will identify $G^{(2)}$ with $GU(S_{n})$ using these choices. Given $g\in G^{(2)}(R)$ (for some $\Q$-algebra $R$), we write
\[ g=\left(\begin{array}{cc}A & B \\ C & D\end{array}\right) \]
for the description of $g$ as a $2n\times2n$-matrix in $GU(S_{n})(R)$, and call the entries $A(g)=A$, $B(g)=B$, $C(g)=C$, $D(g)=D$, the \emph{Siegel coordinates} of $g$. We let $GP\subset G^{(2)}$ be the subgroup described by the condition $C(g)=0$. This is a rational, maximal parabolic subgroup of $G^{(2)}$.

 For an algebraic Hecke character $\alpha$ of $L$, consider the induced representation $I_{n}(s,\alpha)$ for $s\in\C$, as in \cite{cohomologicalII}, (1.2). We will only treat the case where $\alpha$ is trivial. The case of more general $\alpha$ (or more precisely, $\alpha$ with infinity type $(-\kappa,0)$ at each $\tau\in\Phi$) should follow exactly like our case, but for simplicity of notation we will only concentrate on trivial $\alpha$. We thus denote $I(s)=I_{n}(s,1)$. Concretely,
 \[ I(s)=\{f:G^{(2)}(\A)\to\C:f(pg)=\delta_{GP,\A}(p,s)f(g),g\in G^{(2)}(\A),p\in GP(\A)\},\]
 where $\delta_{GP,\A}(p,s)=\|N_{L/K}\det A(p)\|_{\A_{K}}^{\frac{n}{2}+s}\|\nu(p)\|_{\A_{K}}^{-\frac{n^{2}}{2}-ns}$. We similarly define the local inductions $I(s)_{v}$ and finite and archimedean inductions $I(s)_{f}$ and $I(s)_{\infty}$, and we require the functions to be $K_{x^{(2)}}(\R)$-finite. A section of $I(s)$ is a function $\varphi(\cdot,\cdot)$ that to each $s\in\C$ assigns an element $\varphi(\cdot,s)\in I(s)$, with a certain continuity property (see \cite{harriscrelle}, (3.1.8)). We define local sections similarly. The first variable is usually denoted by $g$, meaning that $\varphi(g,s)\in I(s)$ is a function of $g\in G^{(2)}(\A)$. For $\operatorname{Re}(s)\gg0$, the Eisenstein series
\[ E_{\varphi,s}(g)=\sum_{\gamma\in GP(\Q)\backslash G^{(2)}(\Q)}\varphi(\gamma g,s) \]
converges absolutely to an automorphic form on $G^{(2)}(\A)$, and this extends meromorphically to a function of $s\in\C$. For normalization purposes, we also include a possible shift in the construction. Namely, suppose that $\varphi(g,s)$ is a section but with $\varphi(g,s)\in I(s+a)$ for some fixed $a\in\Z$. Then we can still form the Eisenstein series $E_{\varphi,s}(g)$ using the same formula. Of course, if we define $\varphi'(g,s)=\varphi(g,s-a)$, this is a true section and the corresponding Eisenstein series satisfy $E_{\varphi,s}(g)=E_{\varphi',s+a}(g)$.

For an integer $m$ such that $2m\geq n$, consider the elements 
\[ \mathbb{J}_{m}\left(g,s+m-\frac{n}{2}\right)\in I\left(s+m-\frac{n}{2}\right)_{\infty}\]
defined in \cite{cohomologicalII}, (1.2.7) (taking $\kappa=0$; note that there is a misprint: there should be an extra factor $\nu(g)^{ns}$ multiplying the expression, so that it transforms correctly under $\delta$).  These are constructed from certain elements $\mathbb{J}_{m}\in I(m-\frac{n}{2})_{\infty}$ which generate irreducible, unitarizable $(\mathfrak{g}^{(2)}_{\C},K_{x^{(2)}}(\R))$-modules $\mathbb{D}_{m}\subset I(m-\frac{n}{2})_{\infty}$ (denoted by $\mathbb{D}(m,0)$ in \emph{op. cit.}). This module has the property that its generated by its $\mathfrak{p}_{x^{(2)}}^{-}$-torsion, and its lowest $K_{x^{(2)}}$-type is given by the character $\Lambda_{-m}$. Then
\[ H^{0}(\mathfrak{P}_{x^{(2)}},K_{x^{(2)}};\mathcal{A}(G^{(2)})\otimes_{\C}\mathbb{V}_{m})\cong\Hom_{\mathfrak{g}_{\C}^{(2)},K_{x^{(2)}}(\R)}(\mathbb{D}_{m},\mathcal{A}(G^{(2)})).\]
Combining this with~(\ref{h0can}), we obtain an isomorphism
\begin{equation}\label{isoDmodule} \Hom_{\mathfrak{g}_{\C}^{(2)},K_{x^{(2)}}(\R)}(\mathbb{D}_{m},\mathcal{A}(G^{(2)}))\cong H^{0}(S^{(2)}_{\C},\mathcal{E}_{m}^{\can}).\end{equation}
Suppose that $\varphi_{f}(\cdot,s)\in I(s)_{f}$ is a section at the finite id\`eles. Define
\begin{equation}\label{defphi} \varphi(g,s)=\mathbb{J}_{m}\left(g,s+m-\frac{n}{2}\right)\otimes\varphi_{f}\left(g,s+m-\frac{n}{2}\right)\in I\left(s+m-\frac{n}{2}\right).\end{equation}

If $m>n$, then $E_{\varphi,s}(g)$ has no pole at $s=0$ (\cite{shimuraeisenstein}), so we can speak of the automorphic form $E_{\varphi}=E_{\varphi,0}$ on $G^{(2)}(\A)$. Using (\ref{isoDmodule}), we can define an element also denoted by $E_{\varphi}\in H^{0}(S^{(2)}_{\C},\mathcal{E}_{m}^{\can})$, which is the map $\mathbb{D}_{m}\to\mathcal{A}(G^{(2)})$ sending $\mathbb{J}_{m}$ to $E_{\varphi}$. The rationality properties of $E_{\varphi}$ are dictated by the rationality properties of its constant term $\varphi$, which we can identify with a holomorphic automorphic form on the point boundary stratum attached to $(G^{(2)},X^{(2)})$ and the parabolic $GP$. We summarize these rationality properties in the next proposition, which follows from the results and techniques of \cite{harriseisenstein}, \cite{harrisvb2} and \cite{harrisgarrett}.

\begin{prop}\label{proprationaleis} Let $m>n$, and let $F\subset\overline\Q$ be a number field containing $L'$. Then $E_{\varphi}\in H^{0}(S^{(2)}_{\C},\mathcal{E}_{m}^{\can})$ is rational over $F$ if and only if $\varphi_{f}\left(\cdot,m-\frac{n}{2}\right)$ takes values in $(2\pi i)^{enm}F\Q^{\ab}$ and satisfies the following reciprocity law: if $g\in G^{(2)}(\A_{f})$ and $a\in\A_{f}^{\times}$ is an element such that $\art_{\Q}(a)\in\Gal(\Q^{\ab}/F\cap\Q^{\ab})$, then
\[
\left((2\pi i)^{-enm}\varphi_{f}\left(g,m-\frac{n}{2}\right)\right)^{\art_{\Q}(a)}=(2\pi i)^{-enm}\varphi_{f}\left(g,m-\frac{n}{2}\right).
\]
\begin{proof}

First of all, we need to identify the Shimura datum $(G_{P},F_{P})$, in the notation of \cite{harrisvb2}, Section 5. After going through the definitions, we find that, in Siegel coordinates, $G_{P}=G_{h}A_{P}\subset G^{(2)}$, where for a $\Q$-algebra $R$ 
\[ G_{h}(R)=\left\{aI_{2n}\in G(R),\quad a\in(L\otimes R)^{\times}, a\overline a\in R^{\times}\right\}  \]
and
\[ A_{P}(R)=\left\{\left(\begin{array}{cc}aI_{n} & 0 \\ 0 & dI_{n}\end{array}\right)\in G(R),\quad a,d\in R^{\times}\right\}.\] 
The boundary component $F_{P}$ in our case is a single point, corresponding to the map $h_{P}:\mathbb{S}\to G_{P,\R}$ described by
\[ h_{P}(z)=\left(\begin{array}{cc}z\overline z I_{n} & 0 \\ 0 & I_{n}\end{array}\right).\]
The constant term of $E_{\varphi}$, identified with a holomorphic automorphic form on $\Sh(G_{P},h_{P})$, is a section of a certain automorphic line bundle $\mathcal{E}_{m,P}$ that we can obtain as follows. There is a character $\tilde\lambda$ of $P$, whose restriction $\lambda$ to $G_{P}$ gives rise to this line bundle, and whose restriction to $P(\R)^{+}$ (the elements with positive multiplier) is the character $\delta_{GP,\R}(\cdot,m-\frac{n}{2})^{-1}$. Putting all together, we obtain that the character $\tilde\lambda:{P}\to\Gm{\Q}$ is given in Siegel coordinates by
\[ \tilde\lambda(p)=(N_{L/\Q}\det(A(p)))^{-m}\nu(p)^{enm} \]
(recall that $e=[K:\Q]$), and $\mathcal{E}_{m,P}$ is given by its restriction $\lambda:G_{P}\to\Gm{\Q}$. Note that the restriction of $\mathcal{E}_{m,P}$ to $\Sh(\Gm{\Q},N)$ (where $N$ is the norm) is the Tate automorphic vector bundle $\Q(-enm)$, and that $\Sh(G_{P},h_{P})$ is covered by the translates of $\Sh(\Gm{\Q},N)$. Here we are including $\Gm{\Q}$ in $G_{P}$ by taking $a$ to the diagonal matrix with $A(g)=aI_{n}$, $D(g)=I_{n}$. The rest of the proof follows from the arguments in \S3, (A.2.4) and (A.2.5) of \cite{harrisgarrett}. 
\end{proof}
\end{prop}

Now, let $m>n$ be an integer satisfying~(\ref{ineqfort}), and $\varphi$ defined by~(\ref{defphi}) with $\varphi_{f}$ satisfying the hypotheses of Proposition~\ref{proprationaleis}. Also, let $\ell$ be an arbitrary integer. Using the usual action of $U(\mathfrak{p}_{x^{(2)}}^{+})$ on smooth functions on $G^{(2)}(\R)$, we can define a section $\Delta_{m}(\Lambda;\ell)\varphi$ by means of the element $\partial_{m}(\Lambda;\ell)$. Then 
\[ \Delta_{m}(\Lambda;\ell)E_{\varphi}\in H^{0}(S_{\C}^{\sharp},\mathcal{E}_{{\Lambda^{\sharp}(\ell)}}^{\can})\cong H^{0}(\mathfrak{P}_{x^{\sharp}},K_{x^{\sharp}};\mathcal{A}^{\sharp}\otimes_{\C}V_{\Lambda^{\sharp}(\ell)})\]
is represented as a scalar automorphic form, as in Remark~\ref{scalarforms}, by the restriction to $G^{\sharp}(\A)$ of $E_{\Delta_{m}(\Lambda;\ell)\varphi}$. We can similarly define $\Delta_{m}(\Lambda;\ell)\mathbb{J}_{m}$.

\subsection{Piatetski-Shapiro-Rallis zeta integrals}
Let $\pi$ be a cuspidal automorphic representation of $G(\A)$. We always consider $\pi$ (rather, the space on which $\pi$ acts) to be a specific irreducible subspace of $\mathcal{A}_{0}(G)$. Similarly for its dual and their twists. Let 
\[ (\cdot,\cdot)_{G}:\pi\otimes\pi^{\vee}\to\C \]
be the pairing given as
\[ (f,f')_{G}=\int_{Z(A)G(\Q)\backslash G(\A)}f(g)f'(g)dg.\]

Suppose that $\varphi(g,s)$ is a section, $f\in\pi$ and $f'\in\pi^{\vee}$. The (modified) Piatetski-Shapiro-Rallis zeta integral is defined to be
\[ Z(s,f,f',\varphi)=\int_{Z^{\sharp}(\A)G^{\sharp}(\Q)\backslash G^{\sharp}(\A)}E_{\varphi,s}(i(g,g'))f(g)f'(g')dgdg',\]
where $Z^{\sharp}$ is the center of $G^{\sharp}$.

Decompose $\pi=\otimes'_{v}\pi_{v}$, $\pi^{\vee}=\otimes'_{v}\pi^{\vee}_{v}$, where $v$ runs through places of $\Q$, and suppose that $f$ and $f'$ are factorizable relative to these decompositions. Thus, $f=\otimes'_{v}f_{v}$ and $f'=\otimes_{v}'f'_{v}$. At almost all places, $\pi_{v}$ is unramified, and $f'_{v}$ and $f'_{v}$ are normalized spherical vectors with $(f_{v},f'_{v})=1$. Suppose as well that the section $\varphi(g,s)$ is factorizable as $\prod_{v}'\varphi_{v}$. Define the local integrals
\[ Z_{v}(s,f,f',\varphi)=\int_{U_{v}}\varphi_{v}(i(h_{v},1),s)c_{f,f',v}(h_{v})dh_{v},\]
where $U_{v}$ is the local unitary group for $V$. Here 
\[ c_{f,f',v}(h_{v})=(f_{v},f'_{v})^{-1}(\pi_{v}(h_{v})f_{v},f'_{v}) \]
is a normalized matrix coefficient for $\pi_{v}$. Let $S$ be a big enough set of places (namely, we can take $S$ to contain the set consisting of the archimedean places, the set of places at which $G_{v}$ is not quasi-split, and the set of places where $\pi_{v}$ is ramified or $f_{v}$ or $f'_{v}$ is not a standard spherical vector). Define
\[ d^{S}(s)=\prod_{j=0}^{n-1}L^{S}(2s+n-j,\varepsilon_{L}^{j}).\]
Then Li proved in \cite{li} that 
\begin{equation}\label{PSR} d^{S}\left(s-\frac{n}{2}\right)Z\left(s-\frac{n}{2},f,f',\varphi\right)=(f,f')_{G}\prod_{v\in S}Z_{v}\left(s-\frac{n}{2},f,f',\varphi\right)L^{\mot,S}\left(s,\pi,\St\right). \end{equation}

Writing $S=S_{\infty}\cup S_{f}$, we let $Z_{f}(s,f,f',\varphi)$ be the product of the local integrals for $v\in S_{f}$, and $Z_{\infty}(s,f,f',\varphi)$ be the local archimedean factor. 

\begin{lemma}\label{lemmadn} With the previous definitions, 
\[ d^{S}\left(m-\frac{n}{2}\right)\sim_{\Q}(D_{K}^{1/2})^{\lfloor\frac{n+1}{2}\rfloor}\delta([\varepsilon_{L}])^{\lfloor\frac{n}{2}\rfloor} (2\pi i)^{e\left(2mn-n(n-1)/2\right)}.\]
\begin{proof} Each factor $L^{S}(2m-j,\varepsilon_{L}^{j})$ differs from the full $L$-function $L(2m-j,\varepsilon_{L}^{j})$ by a multiple in $\Q^{\times}$. Consider first the case when $j$ is even, so that $L(2m-j,\varepsilon_{L}^{j})=\zeta_{K}(2m-j)$. Since $2m-j$ is a positive even integer, we can write
\[ \zeta_{K}(2m-j)\in (D_{K}^{1/2}(2\pi i)^{e(2m-j)})\Q^{\times} \]
using the Klingen-Siegel Theorem. The factors with $j$ odd correspond to special values of $L(s,\varepsilon_{L})$ at odd positive integers. These integers are critical for $\Res_{K/\Q}[\varepsilon_{L}]$. Deligne's conjecture (a theorem in this case) says then that 
\[ L(2m-j,\varepsilon_{L})\sim_{\Q}c^{-}([\varepsilon_{L}])(2\pi i)^{e(2m-j)}\]
for $0\leq j\leq n-1$ odd. Note that $c^{-}([\varepsilon_{L}])\sim_{\Q}\delta([\varepsilon_{L}])$ by (\ref{form:yoshida1}), (\ref{form:yoshida2}) and Remark \ref{rem:twist} (the motive $[\varepsilon_{L}]$ is special since $\varepsilon_{L}(c_{\sigma})=-1$ for any $\sigma\in J_{K}$). The final formula follows by combining all these factors.
\end{proof}

\end{lemma}

\subsection{The main theorem}\label{ssec:mainthm}
In this subsection we derive our main theorem for critical values of the standard $L$-function of an automorphic representation $\pi$, in the range of absolute convergence. We recollect the notation that we used in the previous subsections. From now on, suppose that $\pi\in\Coh_{G,\mu}$, with $W_{\mu}$ defined over $\Q$, and suppose that $\pi^{\vee}\simeq\pi\otimes\|\nu\|^{\xi}$, where $\xi=\xi(\mu)$. Recall that we are fixing a CM type $\Phi$ for $L/K$. We let ${M}(\pi)\in\mathcal{R}(L')_{E(\pi)}$ be the Hodge-de Rham structure $M(\pi,W)$. We fix an algebraic Hecke character $\psi$ of $L$, of infinity type $(m_{\tau})_{\tau\in J_{L}}$. We denote by $\eta$ the corresponding algebraic character of $T^{L}=\Res_{L/\Q}\Gm{L}$. We let $\Lambda=\Lambda(\mu;\eta^{-1})$ and $\ell=n\sum_{\tau\in\Phi}(m_{\tau}-m_{\overline\tau})$. We fix an integer $m>n$ satisfying the inequalities~(\ref{ineqfort}), and we let $\Delta_{m}=\Delta_{m}(\Lambda;\ell)$. Recall that we are taking $E(\psi)=L'\Q(\psi)$. We let $E(\pi,\psi)=E(\pi)\otimes E(\psi)$.

We assume that $\pi$ contains antiholomorphic automorphic forms. As in Remark~\ref{antiholomorphic}, we can find a nonzero, $L'$-rational, $G(\A_{f})$-equivariant map $\beta:\pi_{f,0}\to H^{d}_{!}(S_{\C},\mathcal{E}_{\mu})$. Note that $\mu^{c}=\mu^{\vee}+\mu_{\xi}$. Let $\beta':\pi_{f,0}^{\vee}\to H^{d}_{!}(\overline{S}_{\C},\mathcal{E}_{\mu^{\vee}})$ be the $G(\A_{f})$-map defined as the composition
\[ \pi_{f,0}^{\vee}\to\pi_{f,0}^{\vee}\otimes\|\nu\|^{-\xi}\cong\pi_{f,0}\xrightarrow{c_{\coh}(\beta)}H^{d}_{!}(\overline{S}_{\C},\mathcal{E}_{\mu^{c}})\xrightarrow{\alpha\mapsto\hat\alpha}H^{d}_{!}(\overline{S}_{\C},\mathcal{E}_{\mu^{\vee}}),\]
where the map $\alpha\mapsto\hat\alpha$ is the one we defined in Subsection~\ref{rempipidual}. Let $\gamma=\beta(\psi)$ and $\gamma'=\beta'(\psi^{-1})$ be as in~(\ref{defbetachi}). Cup product and pullback via $S^{\sharp}_{\C}\hookrightarrow S_{\C}\times\overline{S}_{\C}$ give rise to an element which we denote by $(\gamma,\gamma')^{\sharp}$. This element contributes to $H^{2d}_{!}(S^{\sharp}_{\C},\mathcal{E}_{\left((\mu+\mu(\eta),\mu^{\vee}-\mu(\eta)\right)^{\sharp}})$. Note that we can identify 
\[ \mathcal{E}_{\left((\mu+\mu(\eta),\mu^{\vee}-\mu(\eta)\right)^{\sharp}}=\mathcal{E}_{\Lambda^{\sharp}(\ell)}'=\mathcal{E}_{\Lambda^{\sharp}(\ell)}^{\vee}\otimes\mathbb{K}^{\sharp}\]
as automorphic vector bundles, where $\mathbb{K}^{\sharp}$ is the canonical bundle on $S^{\sharp}_{\C}$. 
As in Subsection~\ref{autquad}, the elements $\gamma$ and $\gamma'$ give rise to automorphic forms $f$ in $\pi\otimes\psi$ and $f'\in\pi^{\vee}\otimes\psi^{-1}$.

Let 
\[ \tilde\varphi_{m,\infty}(\cdot,s)=(\Delta_{m}\mathbb{J}_{m})\left(\cdot,s+m-\frac{n}{2}\right)\in I\left(s+m-\frac{n}{2}\right)_{\infty}.\]
The following result is proved by Garrett in \cite{garrett}; see \cite{siegelweil}, (4.4) for details.

\begin{lemma}\label{archzeta} The archimedean local factor $Z_{\infty}\left(m-\frac{n}{2},f,f',\tilde\varphi_{m,\infty}\right)$ belongs to $(L')^{\times}$.
\end{lemma}

\begin{lemma}\label{finzeta} There exists a finite section $\varphi_{f}(\cdot,s)\in I(s)_{f}$ with $\varphi_{f}(\cdot,m-\frac{n}{2})$ taking values in $\Q$ such that 
\[ Z_{f}\left(m-\frac{n}{2},f,f',\varphi_{f}\right)\in(L')^{\times}.\]
\begin{proof} The existence of $\varphi_{f}$ with $\varphi_{f}(\cdot,m-\frac{n}{2})\in\Q$ and $Z_{f}\left(m-\frac{n}{2},f,f',\varphi_{f}\right)\neq0$ follows from results of Li (\cite{li}), as in Lemma 3.5.7 of \cite{harriscrelle} and the discussion after it. The fact that it belongs to $L'$ follows from our choices and Lemma 6.2.7 of \cite{harrisunitary}.
\end{proof}
\end{lemma}

From now on, take $\varphi_{f}$ as in Lemma \ref{finzeta}, and define
\[ \varphi_{m}(\cdot,s)=\mathbb{J}_{m}\left(\cdot,s+m-\frac{n}{2}\right)\otimes(2\pi i)^{enm}\varphi_{f}\left(\cdot,s+m-\frac{n}{2}\right)\in I\left(s+m-\frac{n}{2}\right) \]
and
\[ \tilde\varphi_{m}(\cdot,s)=\tilde\varphi_{m,\infty}(\cdot,s)\otimes(2\pi i)^{enm}\varphi_{f}\left(\cdot,s+m-\frac{n}{2}\right)\in I\left(s+m-\frac{n}{2}\right),\]
so that $\Delta_{m}\varphi_{m}=\tilde\varphi_{m}$. We let $E_{m}=E_{\varphi_{m}}$ and $\tilde E_{m}=E_{\tilde\varphi_{m}}$. 

\begin{lemma}\label{formulazetaL} Let the notation and assumptions be as above. Then
\[ Z\left(m-\frac{n}{2},f,f',\tilde\varphi_{m}\right)\in L'.\]
\begin{proof}
It follows from Proposition~\ref{proprationaleis} that $E_{m}$ and $\tilde E_{m}$ are rational over $L'$. The Eisenstein series $\tilde E_{m}$, when restricted to $G^{\sharp}(\A)$, equals $\Delta_{m}E_{m}$ (as scalars automorphic forms). Let 
\[ \mathcal{L}_{m}:H^{2d}(S_{L'}^{\sharp},\mathcal{E}^{'}_{\Lambda^{\sharp}(\ell)})\to L' \]
be the $L'$-linear map defined by pairing with $\Delta_{m}E_{m}$ via Serre duality. Then, as in Remark~\ref{scalarforms}, we can write
\[ \mathcal{L}_{m}\left((\gamma,\gamma')^{\sharp}\right)=Z\left(m-\frac{n}{2},f,f',\tilde\varphi_{m}\right).\]
Since $\gamma$ and $\gamma'$ are rational over $L'$, this shows that the zeta function belongs to $L'$. 
\end{proof}
\end{lemma}

From now on, we assume the following hypothesis.

\begin{hyp}\label{hypomult} If $\sigma\in J_{\pi}$ then $\dim_{\C[G(\A_{f})]}(\pi_{f}^{\sigma},H^{d}_{!}(S_{\C},\mathcal{E}_{\mu}))\leq 1$.
\end{hyp}
Using the map $c_{B}$, we can see that $\dim_{\C[G(\A_{f})]}(\pi_{f}^{\sigma},H^{0}_{!}(S_{\C},\mathcal{E}_{w_{0}^{1}*\mu}))\leq 1$ as well. The statement of Hypothesis~\ref{hypomult} is part of Arthur's multiplicity conjectures for unitary groups (in which case the analogous statement should hold for the Weyl components corresponding to any $w\in\mathcal{W}^{1}$, but we are only considering the cases $w=1$ and $w=w_{0}^{1}$ here). It is reasonable to expect a proof of these multiplicity conjectures to appear soon, and we refer the reader to \cite{kmsw} and their forthcoming sequels for more details. Under this condition, we can unambiguously define
\[ Q^{\hol}(\pi)=Q(\pi;\alpha) \]
and
\[ Q^{\ahol}(\pi)=Q(\pi;\beta),\]
where $\beta$ is as above and $\alpha$ is a holomorphic vector (rational over $L'$ for the de Rham structure, as usual).

\begin{lemma}\label{lemaqhol} Suppose that $\pi$ satisfies Hypothesis~\ref{hypomult}. Then
\[ Q^{\hol}(\pi)Q^{\ahol}(\pi)\sim_{E(\pi)\otimes L'}1.\]
\begin{proof} This follows from the reasoning in Subsection~\ref{ssec:quadraticperiods}.
\end{proof}
\end{lemma}

We are now ready to state our main theorem, which follows from all the results so far. We include all the relevant hypothesis. 

\begin{thm}\label{maintheorem} Let $\pi\in\Coh_{G,\mu}$, with $W_{\mu}$ defined over $\Q$. Assume that $\pi^{\vee}\cong\pi\otimes\|\nu\|^{\xi}$, that it satisfies Hypothesis~\ref{hypomult} and that it contributes to antiholomorphic cohomology. Let $\psi$ be an algebraic Hecke character of $L$, with infinity type $(m_{\tau})_{\tau\in J_{L}}$, and let $m>n$ be an integer satisfying~(\ref{ineqfort}). Then
\[ L^{S,\mot}(m,\pi\otimes\psi,\St)\sim_{E(\pi,\psi)\otimes L'}\]
\[ (2\pi i)^{e\left(mn-n(n-1)/2\right)-\xi}(D_{K}^{1/2})^{\lfloor\frac{n+1}{2}\rfloor}\delta([\varepsilon_{L}])^{\lfloor\frac{n}{2}\rfloor}Q^{\hol}(\pi)\mathbf{p}(\psi;\det\circ x)\mathbf{p}(\psi^{-1};\det\circ\overline x).\]
\begin{proof} Recall that~(\ref{PSR}) for $\tilde\varphi_{m}$ says that
\begin{equation}\label{formproof} d^{S}\left(m-\frac{n}{2}\right)Z\left(m-\frac{n}{2},f,f',\tilde\varphi_{m}\right)=\end{equation}
\[(f,f')_{G}Z_{f}\left(m-\frac{n}{2},f,f',\tilde\varphi_{m}\right)Z_{\infty}(m-\frac{n}{2},f,f',\tilde\varphi_{m})L^{\mot,S}\left(m,\pi,\St\right).\] Note that $Z_{f}\left(m-\frac{n}{2},f,f',\tilde\varphi_{m}\right)=(2\pi i)^{emn}Z_{f}\left(m-\frac{n}{2},f,f',\varphi_{f}\right)\in(2\pi i)^{enm}L'$ by Lemma \ref{finzeta}. Also, the zeta integrals $Z_{\infty}(m-\frac{n}{2},f,f',\tilde\varphi_{m})$ and $Z\left(m-\frac{n}{2},f,f',\tilde\varphi_{m}\right)$ belong to $L'$ by Lemmas \ref{archzeta} and \ref{formulazetaL}. The factor $(f,f')_{G}$ can be written, by~(\ref{formulaQpipidual}) (rather, its twisted version), as
\[ (f,f')_{G}\sim_{E(\pi,\psi)\otimes L'}(2\pi i)^{\xi}Q(\pi;\psi;\beta). \]
It follows from this and (\ref{formproof}) that
\[ L^{S,\mot}(m,\pi\otimes\psi,\St)\sim_{E(\pi,\psi)\otimes L'}(2\pi i)^{-enm-\xi}d^{S}\left(m-\frac{n}{2}\right)Q(\pi;\psi;\beta)^{-1}.\]
This formula, together with~(\ref{formulaQpichibeta}) and Lemma~\ref{lemmadn}, implies that
\[ L^{\mot,S}\left(m,\pi\otimes\psi,\St\right)\sim_{E(\pi,\psi)\otimes L'} \]
\[ Q(\pi;\beta)^{-1}(2\pi i)^{e\left(mn-n(n-1)/2\right)-\xi}(D_{K}^{1/2})^{\lfloor\frac{n+1}{2}\rfloor}\delta([\varepsilon_{L}])^{\lfloor\frac{n}{2}\rfloor}\mathbf{p}(\psi;\det\circ x)\mathbf{p}(\psi^{-1};\det\circ\overline x).\]
Finally, the theorem follows by combining this with Lemma \ref{lemaqhol}.
\end{proof}
\end{thm}

\begin{rem} The right hand side of this formula contains $Q^{\hol}(\pi)$, which ultimately depends on the choice of CM type $\Phi$. This choice is also reflected on the CM periods, because the map $x$ implicitly depends on $\Phi$.
\end{rem}

\begin{rem}
	Since we are working over $L'$, we can use Lemma~\ref{epsilonenL} and~(\ref{form:yoshida2}) and rewrite the statement of Theorem~\ref{maintheorem} as
	\begin{equation}\label{formulanueva} L^{S,\mot}(m,\pi\otimes\psi,\St)\sim_{E(\pi,\psi)\otimes L'}\end{equation}\[(2\pi i)^{e\left(mn-n(n-1)/2\right)-\xi}D_{K}^{n/2}Q^{\hol}(\pi)\mathbf{p}(\psi;\det\circ x)\mathbf{p}(\psi^{-1};\det\circ\overline x).\]
	Thus, if $n$ is even we can ignore the discriminant factor. 
\end{rem}

\begin{rem} The formula in Theorem~\ref{maintheorem} proves a version of Conjecture 5.1.1 of the recent Jie Lin's thesis at Paris (\cite{thesislin}). The discriminant factor does not appear in her formula, as she assumes that it belongs to the coefficient field. Galois equivariance of the formula will appear in a forthcoming joint work with Lin. 
\end{rem}

\section{Period relations}\label{sec:relations}
In this section, we analyze Deligne's conjecture for certain motives of the form $M\otimes RM(\chi)$ coming from automorphic forms, comparing the results of Sections~\ref{sec:factorization} and \ref{sec:doubling}. From this comparison, we deduce period relations, and we explain how they are also reflected by Tate's conjecture. We start by recalling some results regarding base change and descent for automorphic representations of unitary groups and $\GL_{n}$. This will be the link in translating the results of Section~\ref{sec:doubling} into motivic terms. Some of the objects in this section are conjectural and used solely for heuristic purposes, so this section is hypothetical in nature. Nevertheless, we do write down concretely the period relations that we obtain in the end in terms of automorphic forms on unitary groups.

\subsection{Transfer and descent}\label{subsecdescent} We fix as before a CM extension $L/K$. We suppose from now on that $L=KE$ for some quadratic imaginary field $E/\Q$ (with a fixed inclusion $E\subset\overline\Q$). Suppose that $G$ is a unitary group attached to an $n$-dimensional hermitian space over $L/K$, and $\Pi$ is a cuspidal, cohomological, conjugate self-dual representation of $\GL_{n}(\A_{L})$. Then we expect the existence of a descent $\pi$ to $G$ (actually $\Pi$ should descend to an $L$-packet, but for our purposes we just choose one member of the corresponding $L$-packet). This has been proved in a significant number of cases (\cite{labesse}; see also \cite{mok} and \cite{kmsw}). We will furthermore assume that $\pi\in\Coh_{G,\mu}$, contributes to antiholomorphic cohomology, and satisfies $\pi^{\vee}\cong\pi\otimes\|\nu\|^{\xi}$. The parameter $\mu=((a_{\tau,1},\dots,a_{\tau,n})_{\tau\in\Phi};a_{0})$ is obtained from the weight of $\Pi$. We will give more details later in a special case. 

Suppose that $G'$ is another unitary group attached to an $n$-dimensional hermitian spaces over $L/K$. We say that two representations $\pi$ and $\pi^{'}$ of $G$ and $G'$ respectively are \emph{nearly equivalent} if $\pi_{v}\simeq\pi_{v}'$ for almost all places $v$ (this only makes sense for the places where the local groups are isomorphic, which is the case for all $v$ except a finite number of them). If this is the case, the partial $L$-functions $L^{S}(s,\pi,\St)$ and $L^{S}(s,\pi',\St)$ agree. The general Langlands philosophy predicts that, given $\pi$, there exists a represetantion $\pi'$ (more precisely, an $L$-packet) which is nearly equivalent to $\pi$. This usually involves a two step process: first, transfer $\pi$ to a cohomological automorphic representation $\Pi$ of $\GL_{n}(\A_{L})$ (this involves also an algebraic Hecke character $\phi_{\pi}$ of $E$; see the appendix to \cite{shin} by Shin), and then descend $\Pi$ to $G'$ as in the previous paragraph.

\subsection{Motivic interpretation}\label{motivicinterp}
Suppose that $\pi\in\Coh_{G,\mu}$. Let $\phi_{\pi}$ be the Hecke character of $E$ mentioned above, obtained by base change from $\pi$. Concretely, $\phi_{\pi}$ is given by the restriction to ${\A_{E}^{\times}}\subset Z(\A)$ of the conjugate $\chi_{\pi}^{c}$ of the central character of $\pi$. Its relation to the base change $\Pi$ is that the central character of $\Pi$ satisfies $\chi_{\Pi}|_{\A_{E}^{\times}}=\phi^{c}_{\pi}/\phi_{\pi}$. The character $\phi_{\pi}$ is an algebraic Hecke character of infinity type $(a_{0},a_{0})$. Then, there is a motive $M(\phi_{\pi})_{L'}=M(\phi_{\pi})\times_{E}L'\in\mathcal{M}(L')_{\Q(\phi_{\pi})}$ of rank $1$ and weight $\xi$ attached to $\phi_{\pi}$. 

Assume that $\Pi$ is cuspidal. We expect the existence of a motive $M(\Pi)$ over $L$, of rank $n$ and weight $n-1$, with coefficients in a number field $E(\Pi)$, such that its $\lambda$-adic representations give rise to the Galois representations attached to $\Pi$ (whose construction is due to many people, see \cite{chl}, \cite{shingalois}, \cite{ch}, \cite{sorensen}). The normalizations are such that, outside a finite set of places $S$, 
\[ L_{v}\left(s-\frac{n-1}{2},\Pi^{\sigma}\right)=L_{v}(\sigma,M(\Pi),s).\]
We do not need to know that $M(\Pi)$ actually exists. Its existence will be used for heuristic purposes, and we will just deduce some consequences of it in terms of period relations. Take $E(\Pi)=E(\pi)$ big enough, containing $\Q(\phi)$. Fixing the CM type $\Phi$, we let $(r_{\tau},s_{\tau})$ be the signature of $G$ at each place $\tau$, and we let $t_{\mathbf{s}}=\frac{1}{2}\sum_{\tau\in\Phi}s_{\tau}(s_{\tau}-1)$. We expect the existence of an isomorphism of motives over $L'$
\[ M(\pi)\otimes_{E(\pi)}M(\phi_{\pi})_{L'}\cong\left(\bigotimes_{\tau\in\Phi}\Lambda^{s_{\tau}}\left(M(\Pi)\times_{L,\tau}L'\right)\right)(t_{\mathbf{s}}).\]
(see \cite{hls}, 3.2.2-3.2.4). The corresponding Galois representations have been shown to be equal in many cases, and Tate's conjecture would imply this expected isomorphism. In particular, if $G$ has signature $(n,0)$ at all places $\tau\neq\tau_{0}$ and signature $(n-1,1)$ at $\tau_{0}$, then
\begin{equation}\label{expectedTateGLn} M(\pi)\otimes_{E(\pi)}M(\phi_{\pi})_{L'}\cong M(\Pi)\times_{L,\tau_{0}}L'.\end{equation}
Suppose that for each $\tau\in\Phi$, there is a group $G_{\tau}$, whose signature at $\tau$ is $(n-1,1)$ and whose signature at all other places is $(n,0)$, and suppose that there is an automorphic representation $\pi_{\tau}$ of $G_{\tau}(\A)$ which is nearly equivalent to $\pi$. Then, by the formulas above, we expect the existence of an isomorphism
\begin{equation}\label{expectedTateU} 
M(\pi)\otimes_{E}M(\phi)_{L'}\cong\left(\bigotimes_{\tau\in\Phi}\Lambda^{s_{\tau}}\left(M(\pi_{\tau})\otimes_{E}M(\phi)_{L'}\right)\right)(t_{\mathbf{s}}),
\end{equation}
where $\phi=\phi_{\pi}=\phi_{\pi_{\tau}}$ and $E=E(\pi)=E(\pi_{\tau})$ for all $\tau$.

In this paper we will mostly be concerned with the case of $\Pi$ self-dual. In this case, the motive $M(\Pi)$ should be the base change from $K$ to $L$ of a polarized regular motive $M(\Pi_{0})$ over $K$, where $\Pi_{0}$ is a cuspidal, self-dual, cohomological automorphic representation of $\GL_{n}(\A_{K})$ whose base change to $\GL_{n}(\A_{L})$ is $\Pi$. Under these conditions, $M(\Pi)^{\vee}=M(\Pi)(n-1)$. 

\begin{rem} If $\Pi$ is the base change of $\pi$ with $\pi^{\vee}\cong\pi\otimes\|\nu\|^{\xi}$, as we assume, then $\Pi$ is self-dual. Indeed, this follows from the uniqueness of the weak base change from $G(\A)$ to $\GL_{1}(\A_{E})\times\GL_{n}(\A_{L})$, since the $\GL_{n}(\A_{L})$-component is the weak base change of an irreducible constituent of the restriction to the unitary group, and $\pi$ and $\pi\otimes\|\nu\|^{\xi}$ have the same such restriction. 
\end{rem}

\begin{rem}\label{remexistGr} Suppose that $(r_{\tau})_{\tau\in\Phi}$ is a tuple of integers with $r_{\tau}\in\{0,\dots,n\}$ for each $\tau$. Then, by the classification of hermitian spaces, there exists a hermitian space $V$ of dimension $n$ such that the associated unitary group has signature $(r_{\tau},n-r_{\tau})$ at each $\tau\in\Phi$. Moreover, if $n$ is odd, or if $n$ is even and $ne-2\sum_{\tau\in\Phi}r_{\tau}\equiv0(4)$, we can impose the condition that the local unitary group at finite places $v$ of $K$ is quasi-split for all $v$. In the other cases, we can impose the same condition for all finite places except a single $v$, inert over $\Q$ and split in $L$. 
\end{rem}

\subsection{Deligne's conjecture} Let $\Pi_{0}$ be a cuspidal automorphic representation of $\GL_{n}(\A_{K})$, self-dual, and cohomological of type $(a_{\sigma,1},\dots,a_{\sigma,n})_{\sigma\in J_{K}}$ in the usual parametrization. For any $\tau\in J_{L}$, we let $a_{\tau,i}=a_{\sigma,i}$ if $\tau$ restricts to $\sigma$. We let $E=E(\Pi_{0})$, and we take the freedom to make $E$ bigger if necessary. We let $M=M(\Pi_{0})$ be the conjectural motive attached to $\Pi_{0}$ as in the previous subsection. We can recover the Hodge types $p_{i}(\sigma,1)$ (with $1:E\hookrightarrow\C$ the given embedding) from the infinity type of $\Pi_{0,\infty}$. The formula is 
\begin{equation}\label{eqnpi} p_{i}(\sigma,1)=a_{\sigma,i}+n-i \end{equation}
for any $\sigma\in J_{K}$. We can similarly compute the Hodge types $p_{i}(\sigma,\varphi)$ for different embeddings $\varphi\in J_{E}$ by looking at the conjugates of $\Pi_{0}$.

We let $\Pi=\Pi_{0,L}$ be the base change of $\Pi_{0}$ to $L$. Thus, $\Pi$ is an automorphic representation of $\GL_{n}(\A_{L})$, conjugate self-dual (and self-dual), cohomological of type $(a_{\tau,1},\dots,a_{\tau,n})_{\tau\in J_{L}}$. We assume moreover that $\Pi_{0}\not\cong\Pi_{0}\otimes\varepsilon_{L}$ (this is always the case when $n$ is odd), which implies that $\Pi$ is cuspidal. Accordingly we expect a motive $M(\Pi)=M\times_{K}L$ attached to $\Pi$ with the same Hodge types, as in the previous subsection. We note that, starting with $K$ and $\Pi_{0}$, there always exists a CM extension $L/K$ such that $\Pi_{0}\not\cong\Pi_{0}\otimes\varepsilon_{L}$ (see Section 1 of \cite{clozelpurity}).

We fix an algebraic Hecke character $\psi$ of $L$, of infinity type $(m_{\tau})_{\tau\in J_{L}}$ and weight $w=m_{\tau}+m_{\overline\tau}$. We assume $\psi$ to be critical in the sense of Subsection~\ref{motivesRM}, so that $m_{\tau}\neq m_{\overline\tau}$ for any $\tau\in J_{L}$. Write $\psi|_{\A_{K}^{\times}}=\psi_{0}\|\cdot\|^{-w}$, with $\psi_{0}$ of finite order, and let
\[ \chi=\psi^{2}(\psi_{0}\circ N_{L/K})^{-1}.\]
Recall that we previously defined $\check{\chi}=\chi^{\iota,-1}$. If we let $\tilde\psi=\psi/\psi^{\iota}$, then
\[\tilde\psi/\check{\chi}=\|\cdot\|^{-w}.\]
We let $(n_{\tau})_{\tau\in J_{L}}$ be the infinity type of $\chi$, so that $n_{\tau}=2m_{\tau}$ and $\chi$ is critical of weight $2w$. Note that $\chi|_{\A_{K}^{\times}}=\|\cdot\|^{-2w}$, so that the finite order character that we denoted by $\chi_{0}$ in Subsection~\ref{motivesRM} is trivial. In everything that we say, we need to fix a CM field $\Phi$, and we will take it to be $\Phi=\Phi_{\psi}=\Phi_{\chi}$. Thus, an embedding $\tau\in J_{L}$ belongs to $\Phi$ if and only if $m_{\tau}>m_{\overline\tau}$. In the notation of Section~\ref{sec:factorization}, for $\sigma\in J_{K}$, 
\begin{align}\label{eqnpichi} p_{1}^{\chi}(\sigma,1)=n_{\tau}\\ 
\notag p_{2}^{\chi}(\sigma,1)=n_{\overline\tau}\end{align}
for $\tau\in\Phi$ extending $\sigma$.  

The reason we introduce the character $\chi$ is the following. If $G$ is a unitary group and $\pi$ is a descent of $\Pi$ from $\GL_{n}$ to $G$, then
\[ L^{\mot,S}(s,\pi\otimes\psi,\St)=L^{S}(M\otimes RM(\chi),s+w)\]
(see (3.5.2) of \cite{harriscrelle}). More precisely, the right hand side should be replaced by $L^{S}(1,M\otimes RM(\chi),s+w)$, where $1$ is the given embedding of $E\otimes\Q(\chi)$. For simplicity of notation, we will only deal with this component, but everything that follows also works after fixing a different embedding.

Suppose that $M\otimes RM(\chi)$ has critical values. Then, as in Subsection~\ref{periodhecketwists}, we can find integers $r_{\sigma}=r_{\sigma,1,1}(\chi)\in\{0,\dots,n\}$, for each  embedding $\sigma\in J_{K}$, such that $n_{\tau}-n_{\overline\tau}$ belongs to the interval $(n-1-2p_{r_{\sigma}}(\sigma,1),n-1-2p_{r_{\sigma}+1}(\sigma,1))$, where $\tau\in\Phi$ extends $\sigma$. Here we take $p_{0}(\sigma,1)=+\infty$ and $p_{n+1}(\sigma,1)=-\infty$. For $\tau\in J_{L}$, we let $r_{\tau}=r_{\sigma}$ if $\tau$ restricts to $\sigma$. We let $s_{\sigma}=s_{\tau}=n-r_{\tau}$. Assuming that there are critical values, we can express the set of integers $m$ which are critical for $L^{S}(M\otimes RM(\chi),s+w)$ using~(\ref{setofcriticalvalues}). Namely, it consists of the $m\in\Z$ satisfying
\[ \upsilon_{}^{(1)}<m+w\leq\upsilon_{}^{(2)}.\]
Thus, $m$ is a critical integer for $L^{S}(M\otimes RM(\chi),s+w)$ if and only if
\[ \upsilon_{}^{(1)}-w<m\leq \upsilon_{}^{(2)}-w.\]
Combining the definitions of the $\upsilon_{}^{(i)}$ with~(\ref{eqnpi}) and~(\ref{eqnpichi}), we can write
\[ \upsilon_{}^{(1)}-w=\max\{a_{\tau,r_{\tau}+1}+s_{\tau}-1+m_{\tau}-m_{\overline\tau},a_{\tau,s_{\tau}+1}+r_{\tau}-1+m_{\overline\tau}-m_{\tau}\}_{\tau\in\Phi} \]
and
\[ \upsilon_{}^{(2)}-w=\min\{a_{\tau,r_{\tau}}+s_{\tau}+m_{\tau}-m_{\overline\tau},a_{\tau,s_{\tau}}+r_{\tau}+m_{\overline\tau}-m_{\tau}\}_{\tau\in\Phi}. \]

Consider the tuple $(r_{\tau})_{\tau\in\Phi}$ and let $G$ be a unitary group with signatures $(r_{\tau},n-s_{\tau})_{\tau\in\Phi}$, as in Remark \ref{remexistGr}. In accordance with the discussion of Subsection~\ref{subsecdescent}, we assume the existence of a descent $\pi\in\Coh_{G}$ of $\Pi_{L}$. Here $\mu=\left((a_{\tau,1},\dots,a_{\tau,n})_{\tau\in\Phi};0\right)$. We assume furthermore that $\pi$ is self-dual, contributes to antiholomorphic cohomology and satisfies Hypothesis~\ref{hypomult}. The assumption that $a_{0}=0$ is reflected in the self-duality of $\pi$, and it's made for simplifying purposes. In turn, we are assuming that the character $\phi$ of Subsection~\ref{motivicinterp} is trivial. We will use the results of Section~\ref{sec:doubling} with the group $G$, the representation $\pi$ and the Hecke character $\psi$. The crucial fact is that the second inequality $m\leq\upsilon_{}^{(2)}-w$ is exactly the second inequality in~(\ref{ineqfort}). It follows, under all these assumptions, using version~(\ref{formulanueva}) of Theorem \ref{maintheorem}, and assuming $m>n$, that
\[ L^{S}(1,M\otimes RM(\chi),m+w)\sim_{E(\psi)EL'} \]
\[ (2\pi i)^{e\left(mn-n(n-1)/2\right)}D_{K}^{n/2}Q^{\hol}(\pi)p(\psi;\det\circ x)p(\psi^{-1};\det\circ\overline x).
\]
Note that we are taking $E(\pi)=E$. Also, we took $E(\psi)$ to contain $L'$, so we could write $\sim_{E(\psi)E}$, but we choose to leave the $L'$ in the notation for possible future refinements. Deligne's conjecture would then predict the following results.

\begin{prop}\label{deligneprediction} With the previous hypotheses, Deligne's conjecture for the motive $M\otimes RM(\chi)$ implies that
\[ c^{+}(M\otimes RM(\chi)(m+w))_{1}\sim_{E(\psi)EL'}\]
\[ (2\pi i)^{e\left(mn-n(n-1)/2\right)}D_{K}^{n/2}Q^{\hol}(\pi)p(\psi;\det\circ x)p(\psi^{-1};\det\circ\overline x).\]
\end{prop}
Here $c^{+}(M\otimes RM(\chi))_{1}$ means we are taking the restriction of scalars from $K$ to $\Q$ of $M\otimes RM(\chi)$; see Subsection~\ref{ssec:periods}.

\begin{rem} The prediction is actually a consequence of the weaker Conjecture~\ref{deligneweak2} over $L'$ for $M\otimes RM(\chi)$.
\end{rem}

In what follows, we will rewrite the prediction of Proposition~\ref{deligneprediction} and get an expression only depending on automorphic quadratic periods and quadratic periods of the motive $M$. First, we use~(\ref{factpchi}), together with Lemma 1.6 of \cite{harrisunitary}, to get
\[ p(\psi;\det\circ x)p(\psi^{-1};\det\circ\overline x)\sim_{E(\psi)}\prod_{\tau\in\Phi}p(\tilde{\psi}^{r_{\tau}-s_{\tau}};\tau).\]
Since $\tilde{\psi}=\check{\chi}\|\cdot\|^{-w}$, using Proposition 1.4 and Lemma 1.8.3 of \cite{harrisunitary}, we can write this as 
\begin{equation}\label{formpsidet} (2\pi i)^{w\sum_{\tau}r_{\tau}-s_{\tau}}\left(\prod_{\tau\in\Phi}p\left(\check\chi;\tau\right)^{r_{\tau}-s_{\tau}}\right). \end{equation}
Now, note that $\chi_{0}$ is trivial, and hence by~(\ref{formulacritica}) we can write
\[ c^{+}(M\otimes RM(\chi)(m+w))_{1}\sim_{E(\psi)EL'}\]\[(2\pi i)^{emn+w\sum_{\tau}r_{\tau}-s_{\tau}}\delta(M)_{1}\left(\prod_{\tau\in
\Phi}p\left(\check\chi;\tau\right)^{r_{\tau}-s_{\tau}}\right) {Q}^{\mathbf{s}}_{1}(M)     \]
(recall that the weight of $\chi$ is $2w$ and rank of $M$ is $n$). Let 
\[ \partial(M)=(2\pi i)^{-en(n-1)/2}\delta(M)^{-1}D_{K}^{n/2}.\]
Then, by~(\ref{formpsidet}), Deligne's prediction (Proposition~\ref{deligneprediction}) is equivalent to the statement

\begin{equation}\label{predictionsimple}
Q_{1}^{\mathbf{s}}(M)\sim_{E(\psi)EL'}\partial(M)_{1}Q^{\hol}(\pi).\end{equation}
For example, if $n$ is even, then $\delta(M)\sim_{E\otimes K'}D_{K}^{n/2}\prod_{\sigma\in J_{K}}\delta_{\sigma}(M)$ by~(\ref{form:yoshida2}), and $\delta_{\sigma}(M)\sim_{E\otimes K,\sigma}(2\pi i)^{-n(n-1)/2}$, so in this case $\partial(M)\sim_{E\otimes K'}1$ and~(\ref{predictionsimple}) is equivalent to the statement
\[
Q_{1}^{\mathbf{s}}(M)\sim_{E(\psi)EL'}Q^{\hol}(\pi).
\]

\begin{rem} As we said before, we can work with arbitrary embeddings of the coefficient field. In this case, Conjecture~\ref{deligneweak2} would mean
\[ Q^{\mathbf{s}}(M)\sim_{E(\psi)\otimes E;L'}\partial(M)\tilde Q^{\hol}(\pi).\]
The notation $\tilde Q^{\hol}(\pi)$ refers to the fact that for each embedding of the coefficient field, the corresponding signatures, and hence the transfers, are different, and we take $\tilde Q^{\hol}(\pi)$ to consist of the corresponding holomorphic automorphic periods in each component. 

\end{rem}

\subsection{Applications of Tate's conjecture}
Keep the assumptions and notation of the last subsection. Moreover, assume that we have isomorphisms as in ~(\ref{expectedTateGLn}) and~(\ref{expectedTateU}). As mentioned in Subsection~\ref{motivicinterp}, this would follow from Tate's conjecture. 

We can interpret the quadratic periods $Q_{j,\sigma}(M)$ in terms of automorphic quadratic periods. For this, fix $\sigma\in J_{K}$, and $\tau\in\Phi$ extending $\sigma$. Let $G_{\tau}$ and $\pi_{\tau}$ be as in~(\ref{expectedTateU}). Then
\[ M(\pi_{\tau})\cong M\times_{K,\sigma}L'.\]
Suppose that $\{\omega_{1,\sigma},\dots,\omega_{n,\sigma}\}$ is an $E\otimes K$-basis of $M_{\dR}$, as in Subsection~\ref{ssec:quadraticperiods}, and denote by the same letters the $E\otimes L'$-basis of $M_{\dR}\times_{K,\sigma}L'\cong M(\pi_{\tau})_{\dR}$ obtained by extension of scalars to $L'$ via $\sigma$. Let $\{\Omega_{1,\sigma},\dots,\Omega_{d,\sigma}\}$ be the basis of $M_{\sigma}\otimes\C$ constructed in Subsection~\ref{ssec:quadraticperiods}. Now, by Remark~\ref{descrW1n-11}, the Hodge decomposition of $M_{\sigma}\otimes\C$ is given by
\[ M_{\sigma}\otimes\C\cong M(\pi_{\tau})_{1}\otimes\C=\bigoplus_{j=1}^{n}M(\pi_{\tau})^{p_{j},q_{j}}. \]
Here $p_{j}=a_{\sigma,j}+n-j$, $q_{j}=n-1-p_{j}$ (more precisely, the decomposition is obtained after tensoring with $\C$ over $E$ with the given embedding) and $M(\pi_{\tau})^{p_{j},q_{j}}$ is the Weyl component corresponding to the element $w_{j}$ defined in Remark~\ref{descrW1n-11}. Moreover, the basis is chosen in a such a way that $\{\omega_{1,\sigma},\dots,\omega_{j,\sigma}\}$ is a basis of the filtration stage $F^{p_{j}}(M(\pi_{\tau})_{\dR})$ and $\Omega_{j,\sigma}$ defines a nonzero rational element of the Hodge component corresponding to the element $w_{j}\in\mathcal{W}^{1}$. It follows at once from the definitions that 
\begin{equation}\label{interpquad} Q_{j,\sigma}\sim Q(\pi_{\tau};\beta_{j}),\end{equation}
where $\beta_{j}$ contributes to coherent cohomology with respect to the Weyl component of type $w_{j}$. From~(\ref{expectedTateU}), we obtain the period relation
\begin{equation}\label{factopetersson} Q^{\hol}(\pi)\sim_{E(\psi)EL'}\prod_{\tau\in\Phi}\prod_{j=1}^{s_{\tau}}Q_{j,\sigma}.\end{equation}
Once again, this relation would follow from Tate's conjecture, so if we assume its veracity, we can prove the weak version of Deligne's conjecture predicted in~(\ref{predictionsimple}) when $n$ is even.

When $n$ is odd, we invoke another consequence of Tate's conjecture. Assuming that the central character of $\Pi_{0}$ is trivial, as we do, Tate's conjecture predicts the existence of an isomorphism
\[ \Lambda^{n}(M)\cong E(-n(n-1)/2) \]
of motives over $K$ and coefficients in $E$. Then
\[ \delta_{\sigma}(M)\sim_{E\otimes K,\sigma}\delta_{\sigma}(\Lambda^{n}(M))\sim_{E\otimes K,\sigma}\delta_{\sigma}(E(-n(n-1)/2))\sim_{E\otimes K,\sigma}(2\pi i)^{-n(n-1)/2}.\]
It follows, using~(\ref{form:yoshida2}), that
\[ \delta(M)\sim_{E\otimes K'}D_{K}^{n/2}(2\pi i)^{-en(n-1)/2}. \]
Thus, $\partial(M)\sim_{E\otimes K'}1$ in this case as well, and hence the weak version of Deligne's conjecture predicted in~(\ref{predictionsimple}) follows from~(\ref{factopetersson}). 

\begin{rem} We don't need the full strength of Tate's conjecture for this. We only need to know that an isomorphism at the level of $\lambda$-adic realizations implies an isomorphism at the level of Hodge-de Rham structures. 
\end{rem}

\begin{rem} When $K=\Q$, Harris proved in a series of works (\cite{cohomologicalI}, \cite{cohomologicalII}, \cite{siegelweil}) that the relations are true up to an unknown factor depending on the infinity type of the automorphic representation. It remains as a project for the coming future to extend these results to the situation of a general totally real field $K$.
\end{rem}

\bibliography{periodsunitary.bib} \bibliographystyle{amsalpha}

\providecommand{\bysame}{\leavevmode\hbox to3em{\hrulefill}\thinspace}
\providecommand{\MR}{\relax\ifhmode\unskip\space\fi MR }
\providecommand{\MRhref}[2]{%
  \href{http://www.ams.org/mathscinet-getitem?mr=#1}{#2}
}
\providecommand{\href}[2]{#2}
\begin{thebibliography}{KMSW14}

\bibitem[BG16]{blasiusguerb}
Don Blasius and Lucio Guerberoff, \emph{Complex conjugation and shimura
  varieties}, preprint, available at arXiv:1602.06572.

\bibitem[BHR94]{bhr}
Don Blasius, Michael Harris, and Dinakar Ramakrishnan, \emph{Coherent
  cohomology, limits of discrete series, and {G}alois conjugation}, Duke Math.
  J. \textbf{73} (1994), no.~3, 647--685.

\bibitem[BJ79]{bj}
Armand Borel and Herv{\'e} Jacquet, \emph{Automorphic forms and automorphic
  representations}, Automorphic forms, representations and {$L$}-functions,
  {P}art 1, Proc. Sympos. Pure Math., XXXIII, Amer. Math. Soc., Providence,
  R.I., 1979, With a supplement ``On the notion of an automorphic
  representation'' by R. P. Langlands, pp.~189--207.

\bibitem[Bla86]{blasiusannals}
Don Blasius, \emph{On the critical values of {H}ecke {$L$}-series}, Ann. of
  Math. (2) \textbf{124} (1986), no.~1, 23--63.

\bibitem[Bla97]{blasiusperiods}
\bysame, \emph{Period relations and critical values of {$L$}-functions},
  Pacific J. Math. (1997), no.~Special Issue, 53--83, Olga Taussky-Todd: in
  memoriam.

\bibitem[BW80]{borelwallach}
Armand Borel and Nolan~R. Wallach, \emph{Continuous cohomology, discrete
  subgroups, and representations of reductive groups}, Annals of Mathematics
  Studies, vol.~94, Princeton University Press, Princeton, N.J.; University of
  Tokyo Press, Tokyo, 1980.

\bibitem[CH13]{ch}
Ga{\"e}tan Chenevier and Michael Harris, \emph{Construction of automorphic
  {G}alois representations, {II}}, Camb. J. Math. \textbf{1} (2013), no.~1,
  53--73.

\bibitem[CHL]{chl}
Laurent Clozel, Michael Harris, and Jean-Pierre Labesse, \emph{Construction of
  automorphic {G}alois representations, {I}}, On the stabilization of the trace
  formula, Stab. Trace Formula Shimura Var. Arith. Appl., vol.~1, Int. Press,
  Somerville, MA, pp.~497--527.

\bibitem[Clo13]{clozelpurity}
Laurent Clozel, \emph{Purity reigns supreme}, Int. Math. Res. Not. IMRN (2013),
  no.~2, 328--346.

\bibitem[Del79a]{deligne}
Pierre Deligne, \emph{Valeurs de fonctions {$L$} et p\'eriodes d'int\'egrales},
  Automorphic forms, representations and {$L$}-functions, {P}art 2, Proc.
  Sympos. Pure Math., XXXIII, Amer. Math. Soc., Providence, R.I., 1979,
  pp.~313--346.

\bibitem[Del79b]{deligneSh}
\bysame, \emph{Vari\'et\'es de {S}himura: interpr\'etation modulaire, et
  techniques de construction de mod\`eles canoniques}, Automorphic forms,
  representations and {$L$}-functions, {P}art 2, Proc. Sympos. Pure Math.,
  XXXIII, Amer. Math. Soc., Providence, R.I., 1979, pp.~247--289.

\bibitem[DMOS82]{dmos}
Pierre Deligne, James~S. Milne, Arthur Ogus, and Kuang-yen Shih, \emph{Hodge
  cycles, motives, and {S}himura varieties}, Lecture Notes in Math., vol. 900,
  Springer-Verlag, Berlin, 1982.

\bibitem[Gar08]{garrett}
Paul~B. Garrett, \emph{Values of {A}rchimedean zeta integrals for unitary
  groups}, Eisenstein series and applications, Progr. Math., vol. 258,
  Birkh\"auser Boston, Boston, MA, 2008, pp.~125--148.

\bibitem[GH93]{harrisgarrett}
Paul~B. Garrett and Michael Harris, \emph{Special values of triple product
  {$L$}-functions}, Amer. J. Math. \textbf{115} (1993), no.~1, 161--240.

\bibitem[GH16]{grobnerharris}
Harald Grobner and Michael Harris, \emph{Whittaker periods, motivic periods,
  and special values of tensor product \$l\$-functions}, Journal of the
  Institute of Mathematics of Jussieu \textbf{FirstView} (2016), 1--59.

\bibitem[Gol14]{shin}
Wushi Goldring, \emph{Galois representations associated to holomorphic limits
  of discrete series}, Compos. Math. \textbf{150} (2014), no.~2, 191--228.

\bibitem[Har84]{harriseisenstein}
Michael Harris, \emph{Eisenstein series on {S}himura varieties}, Ann. of Math.
  (2) \textbf{119} (1984), no.~1, 59--94.

\bibitem[Har86]{harrisvb2}
\bysame, \emph{Arithmetic vector bundles and automorphic forms on {S}himura
  varieties. {II}}, Compos. Math. \textbf{60} (1986), no.~3, 323--378.

\bibitem[Har90]{harrisdeltabar}
\bysame, \emph{Automorphic forms of {$\overline\partial$}-cohomology type as
  coherent cohomology classes}, J. Differential Geom. \textbf{32} (1990),
  no.~1, 1--63.

\bibitem[Har93]{harrisunitary}
\bysame, \emph{{$L$}-functions of {$2\times 2$} unitary groups and
  factorization of periods of {H}ilbert modular forms}, J. Amer. Math. Soc.
  \textbf{6} (1993), no.~3, 637--719.

\bibitem[Har94]{harrismotives}
\bysame, \emph{Hodge-de {R}ham structures and periods of automorphic forms},
  Motives, Proc. Sympos. Pure Math., vol.~55, Amer. Math. Soc., Providence, RI,
  1994, pp.~573--624.

\bibitem[Har97]{harriscrelle}
\bysame, \emph{{$L$}-functions and periods of polarized regular motives}, J.
  Reine Angew. Math. \textbf{483} (1997), 75--161.

\bibitem[Har99]{cohomologicalI}
\bysame, \emph{Cohomological automorphic forms on unitary groups. {I}.
  {R}ationality of the theta correspondence}, Automorphic forms, automorphic
  representations, and arithmetic, Proc. Sympos. Pure Math., vol.~66, Amer.
  Math. Soc., Providence, RI, 1999, pp.~103--200.

\bibitem[Har07]{cohomologicalII}
\bysame, \emph{Cohomological automorphic forms on unitary groups. {II}.
  {P}eriod relations and values of {$L$}-functions}, Harmonic analysis, group
  representations, automorphic forms and invariant theory, Lect. Notes Ser.
  Inst. Math. Sci. Natl. Univ. Singap., vol.~12, World Sci. Publ., Hackensack,
  NJ, 2007, pp.~89--149.

\bibitem[Har08]{siegelweil}
\bysame, \emph{A simple proof of rationality of {S}iegel-{W}eil {E}isenstein
  series}, Eisenstein series and applications, Progr. Math., vol. 258,
  Birkh\"auser Boston, Boston, MA, 2008, pp.~149--185.

\bibitem[HK91]{harriskudla}
Michael Harris and Stephen~S. Kudla, \emph{The central critical value of a
  triple product {$L$}-function}, Ann. of Math. (2) \textbf{133} (1991), no.~3,
  605--672.

\bibitem[HLS11]{hls}
Michael Harris, Jian-Shu Li, and Binyong Sun, \emph{Theta correspondences for
  close unitary groups}, Arithmetic geometry and automorphic forms, Adv. Lect.
  Math. (ALM), vol.~19, Int. Press, Somerville, MA, 2011, pp.~265--307.

\bibitem[Jan90]{jannsen}
Uwe Jannsen, \emph{Mixed motives and algebraic {$K$}-theory}, Lecture Notes in
  Math., vol. 1400, Springer-Verlag, Berlin, 1990, With appendices by S. Bloch
  and C. Schoen.

\bibitem[KMSW14]{kmsw}
Tasho Kaletha, Alberto M{\'{\i}}nguez, Sug~Woo Shin, and Paul-James White,
  \emph{Endoscopic classification of representations: inner forms of unitary
  groups}, preprint.

\bibitem[Lab11]{labesse}
Jean-Pierre Labesse, \emph{Changement de base {CM} et s\'eries discr\`etes}, On
  the stabilization of the trace formula, Stab. Trace Formula Shimura Var.
  Arith. Appl., vol.~1, Int. Press, Somerville, MA, 2011, pp.~429--470.

\bibitem[Li92]{li}
Jian-Shu Li, \emph{Nonvanishing theorems for the cohomology of certain
  arithmetic quotients}, J. Reine Angew. Math. \textbf{428} (1992), 177--217.

\bibitem[Lin15]{thesislin}
Jie Lin, \emph{Special values of automorphic {$L$}-functions for {$GL_n\times
  GL_{n'}$} over {CM} fields, factorization and functoriality of arithmetic
  automorphic periods}, Ph.D. thesis, Universit\'e Paris VII, 2015.

\bibitem[Mil90]{milne}
James~S. Milne, \emph{Canonical models of (mixed) {S}himura varieties and
  automorphic vector bundles}, Automorphic forms, {S}himura varieties, and
  {$L$}-functions, {V}ol.\ {I}, Perspect. Math., vol.~10, Academic Press,
  Boston, MA, 1990, pp.~283--414.

\bibitem[Mok15]{mok}
Chung~Pang Mok, \emph{Endoscopic classification of representations of
  quasi-split unitary groups}, Mem. Amer. Math. Soc. \textbf{235} (2015),
  no.~1108, vi+248.

\bibitem[Pan94]{panch}
Alexei~A. Panchishkin, \emph{Motives over totally real fields and {$p$}-adic
  {$L$}-functions}, Ann. Inst. Fourier (Grenoble) \textbf{44} (1994), no.~4,
  989--1023.

\bibitem[Sch88]{schbook}
Norbert Schappacher, \emph{Periods of {H}ecke characters}, Lecture Notes in
  Math., vol. 1301, Springer-Verlag, Berlin, 1988.

\bibitem[Sch90]{schwermer}
Joachim Schwermer, \emph{Cohomology of arithmetic groups, automorphic forms and
  {$L$}-functions}, Cohomology of arithmetic groups and automorphic forms
  ({L}uminy-{M}arseille, 1989), Lecture Notes in Math., vol. 1447, Springer,
  Berlin, 1990, pp.~1--29.

\bibitem[Ser70]{serre}
Jean-Pierre Serre, \emph{Facteurs locaux des fonctions z\^eta des vari\'et\'es
  alg\'ebriques (d\'efinitions et conjectures)}, S\'eminaire
  Delange-Pisot-Poitou. Th\'eorie des nombres \textbf{11} (1969-1970), no.~2,
  1--15 (fre).

\bibitem[Shi76]{shimurazeta}
Goro Shimura, \emph{The special values of the zeta functions associated with
  cusp forms}, Comm. Pure Appl. Math. \textbf{29} (1976), no.~6, 783--804.

\bibitem[Shi78]{shimuraduke}
\bysame, \emph{The special values of the zeta functions associated with
  {H}ilbert modular forms}, Duke Math. J. \textbf{45} (1978), no.~3, 637--679.

\bibitem[Shi83a]{shimurarelations}
\bysame, \emph{Algebraic relations between critical values of zeta functions
  and inner products}, Amer. J. Math. \textbf{105} (1983), no.~1, 253--285.

\bibitem[Shi83b]{shimuraeisenstein}
\bysame, \emph{On {E}isenstein series}, Duke Math. J. \textbf{50} (1983),
  no.~2, 417--476.

\bibitem[Shi88]{shimuraperiods88}
\bysame, \emph{On the critical values of certain {D}irichlet series and the
  periods of automorphic forms}, Invent. Math. \textbf{94} (1988), no.~2,
  245--305.

\bibitem[Shi11]{shingalois}
Sug~Woo Shin, \emph{Galois representations arising from some compact {S}himura
  varieties}, Ann. of Math. (2) \textbf{173} (2011), no.~3, 1645--1741.

\bibitem[Sor]{sorensen}
Claus~M. Sorensen, \emph{A patching lemma}, manuscript available at
  http://www.imj-prg.fr/fa/spip.php?article10.

\bibitem[Tit71]{tits}
Jacques Tits, \emph{Repr\'esentations lin\'eaires irr\'eductibles d'un groupe
  r\'eductif sur un corps quelconque}, J. Reine Angew. Math. \textbf{247}
  (1971), 196--220.

\bibitem[Yos94]{yoshida}
Hiroyuki Yoshida, \emph{On the zeta functions of {S}himura varieties and
  periods of {H}ilbert modular forms}, Duke Math. J. \textbf{75} (1994), no.~1,
  121--191.

\end{thebibliography}

\end{document}